\documentclass[reqno,11pt]{amsart}
\usepackage{amsmath, latexsym, amsfonts, amssymb, amsthm, amscd}
\usepackage{mathrsfs}
\usepackage{graphics,epsf,psfrag}
\usepackage[lofdepth,lotdepth]{subfig}
\usepackage{graphicx}
\usepackage{url}

\numberwithin{equation}{section}

\newtheorem{theorem}{Theorem}[section]
\newtheorem{lemma}[theorem]{Lemma}
\newtheorem{proposition}[theorem]{Proposition}

\newtheorem{remark}[theorem]{Remark}
\newtheorem{definition}[theorem]{Definition}

\renewcommand{\tilde}{\widetilde}

\newcommand{\cF}{{\ensuremath{\mathcal F}} }

\newcommand{\cE}{{\ensuremath{\mathcal E}} }

\newcommand{\cC}{{\ensuremath{\mathcal C}} }
\newcommand{\cN}{{\ensuremath{\mathcal N}} }
\newcommand{\cM}{{\ensuremath{\mathcal M}} }
\newcommand{\cO}{{\ensuremath{\mathcal O}} }

\newcommand{\cD}{{\ensuremath{\mathcal D}} }
\newcommand{\cU}{{\ensuremath{\mathcal U}} }

\newcommand{\cZ}{{\ensuremath{\mathcal Z}} }

\newcommand{\bP}{{\ensuremath{\mathbf P}} }
\newcommand{\bE}{{\ensuremath{\mathbf E}} }

\newcommand{\bL}{{\ensuremath{\mathbf L}} }

\DeclareMathSymbol{\leqslant}{\mathalpha}{AMSa}{"36} % nicer `smaller or equal'
\DeclareMathSymbol{\geqslant}{\mathalpha}{AMSa}{"3E} % nicer `larger or equal'
\DeclareMathSymbol{\eset}{\mathalpha}{AMSb}{"3F}     % nicer `emptyset'
\renewcommand{\leq}{\;\leqslant\;}                   % redef. of < or =
\renewcommand{\geq}{\;\geqslant\;}                   % redef. of > or =
\newcommand{\dd}{\,\text{\rm d}}             % a straight d for differentials
\DeclareMathOperator*{\Supp}{Supp}
\DeclareMathOperator{\proj}{proj}
\newcommand{\bbC}{{\ensuremath{\mathbb C}} }

\newcommand{\bbE}{{\ensuremath{\mathbb E}} }

\newcommand{\bbR}{{\ensuremath{\mathbb R}} }

\newcommand{\bbT}{{\ensuremath{\mathbb T}} }

\newcommand{\bbZ}{{\ensuremath{\mathbb Z}} }

\newcommand{\ga}{\alpha}

\newcommand{\gd}{\delta}
\newcommand{\gep}{\varepsilon}       % \ge already exists...

\newcommand{\gz}{\zeta}

\newcommand{\go}{\omega}

\newcommand{\gO}{\Omega}
\newcommand{\gl}{\lambda}

\newcommand{\gs}{\sigma}

\makeatletter
\def\captionfont@{\footnotesize}
\def\captionheadfont@{\scshape}

\long\def\@makecaption#1#2{%
  \vspace{2mm}
  \setbox\@tempboxa\vbox{\color@setgroup
    \advance\hsize-6pc\noindent
    \captionfont@\captionheadfont@#1\@xp\@ifnotempty\@xp
        {\@cdr#2\@nil}{.\captionfont@\upshape\enspace#2}%
    \unskip\kern-6pc\par
    \global\setbox\@ne\lastbox\color@endgroup}%
  \ifhbox\@ne % the normal case
    \setbox\@ne\hbox{\unhbox\@ne\unskip\unskip\unpenalty\unkern}%
  \fi
  \ifdim\wd\@tempboxa=\z@ % this means caption will fit on one line
    \setbox\@ne\hbox to\columnwidth{\hss\kern-6pc\box\@ne\hss}%
  \else % tempboxa contained more than one line
    \setbox\@ne\vbox{\unvbox\@tempboxa\parskip\z@skip
        \noindent\unhbox\@ne\advance\hsize-6pc\par}%
\fi
  \ifnum\@tempcnta<64 % if the float IS a figure...
    \addvspace\abovecaptionskip
    \moveright 3pc\box\@ne
  \else % if the float IS NOT a figure...
    \moveright 3pc\box\@ne
    \nobreak
    \vskip\belowcaptionskip
  \fi
\relax
}
\makeatother
\def\writefig#1 #2 #3 {\rlap{\kern #1 truecm
\raise #2 truecm \hbox{#3}}}

\DeclareMathOperator{\Span}{Span}

\title[Disorder-induced traveling waves in the stochastic Kuramoto model]
{Long time dynamics and disorder-induced traveling waves in the stochastic Kuramoto model}

\author{Eric Lu\c{c}on}
\address{Laboratoire MAP5 (UMR CNRS 8145), Universit\'e Paris Descartes, Sorbonne Paris Cit\'e, 75270 Paris, France, \url{eric.lucon@paridescartes.fr}.
}

\author{Christophe Poquet}
\address{
Universit\`a degli Studi di Roma Tor Vergata,
Dipartimento di matematica,
I-00133 Roma, Italia,
\url{poquet@mat.uniroma2.it }.
}


\date{\today}

\begin{document}

\begin{abstract}
The aim of the paper is to address the long time behavior of the Kuramoto model of mean-field coupled phase rotators, subject to white noise and quenched frequencies. We analyse the influence of the fluctuations of both thermal noise and frequencies (seen as a disorder) on a large but finite population of $N$ rotators, in the case where the law of the disorder is symmetric. On a finite time scale $[0, T]$, the system is known to be self-averaging: the empirical measure of the system converges as $N\to \infty$ to the deterministic solution of a nonlinear Fokker-Planck equation which exhibits a stable manifold of synchronized stationary profiles for large interaction. On longer time scales, competition between the finite-size effects of the noise and disorder makes the system deviate from this mean-field behavior. In the main result of the paper we show that on a time scale of order $ \sqrt{N}$ the fluctuations of the disorder prevail over the fluctuations of the noise: we establish the existence of disorder-
induced traveling waves for the empirical measure along the stationary manifold. This result is proved for fixed realizations of the disorder and emphasis is put on the influence of the asymmetry of these quenched frequencies on the direction and speed of rotation of the system. Asymptotics on the drift are provided in the limit of small disorder.

\end{abstract}

\maketitle

%\tableofcontents

\section{Introduction}
\subsection{Long time dynamics of mean-field interacting particle systems}
The macroscopic behavior of numerous stochastic interacting particle systems appearing in physics or biology is usually described by nonlinear partial differential equations. In this context, systems of diffusions in all-to-all interactions, that is \emph{mean-field particle systems} \cite{McKean1967,Oelsch1984}, have attracted much attention in the past years, since they are relevant in many situations from statistical physics (synchronization of oscillators \cite{Acebron2005,Kuramoto1975,Strogatz2000}) to biology (emergence of synchrony in neural networks \cite{22657695, Bossy:2014fk}) and have provided particle approximations for various PDEs (see \cite{Malrieu2003, MR1410117} and references therein). From a statistical physics point of view, a natural extension of these models concerns similar particle systems in a random environment, that is when the particles obey to the influence of an additional randomness, or \emph{disorder}, 
representing inhomogeneous behaviors between 
particles. Such a modeling is particularly relevant in a biological context, where each particle/diffusion captures the state of one single individual (activity of a neuron, phase in a circadian rythm) and the disorder models intrinsic dynamical behavior for each individual (e.g. inhibition or excitation in populations of heterogeneous neurons \cite{22657695,Bossy:2014fk}).

The aim of the paper is to address the influence of the disorder on the long time dynamics of a large but finite population of mean-field interacting diffusions with noise. A crucial aspect in this perspective is the notion of \emph{self-averaging}: in the limit of a large number of individuals and/or on a long time scale (in a way that needs to be made precise), is the macroscopic behavior of the system the same for every typical realization of the disorder?  If not, is it possible to quantify the influence of the fluctuations of the random environment on the behavior of the system?

It appears that the analysis of such mean-field systems differs significantly depending on the time scale one considers. On a time scale of order $1$ (w.r.t. the size of the population), it is now well-known that the macroscopic behavior of mean-field particle systems are well described by nonlinear PDEs of McKean-Vlasov type \cite{Gartner,Oelsch1984}. A vast literature exists on the links between the microscopic system and its mean-field limit (fluctuations, large deviations and finite time dynamics) mostly in the non-disordered case (see e.g. \cite{Fernandez1997,MR876258,MR865013} and references therein) but also for disordered systems \cite{daiPra96, Lucon2011}.

When one considers longer time scales (w.r.t. the size of the population) and for a large but finite number of particles, some randomness remains in the system so that Brownian fluctuations generally induce microscopic dynamics that may differ significantly from the dynamics of the mean-field equation. For mean-field systems without disorder, a vast literature exists concerning fluctuations induced by thermal noise. In this respect, the notion of \emph{uniform propagation of chaos} has been addressed for several mean-field models by many authors (see e.g. \cite{Malrieu2003,MR2731396} for the granular media equation or \cite{Jourdain:2013aa, Reygner:2014ta} for ranked-based models). In case the mean-field PDE admits an isolated stable fixed point, due to large deviation phenomena, the finite-size system exits from any neighborhood of the fixed point at exponential times in $N$ ($N$ being the size of the population) \cite{DawGar,cf:OV}, whereas in case of an unstable fixed point, the system escapes at a time 
scale of order $\log N$ \cite{cf:Errico}. Fewer results exist in the case where the mean-field PDE admits a whole stable curve of stationary solutions. In \cite{Bertini:2013aa,dahms}, the effect of thermal noise is considered for the mean-field plane rotators model \cite{BGP}
 which is known to admit in the limit as $N\to\infty$ a stable circle of stationary solutions. In this case, the finite size particle system has Brownian fluctuations on time scales of order~$N$.

In the case of disordered systems, we are not aware of any similar analysis on long time dynamics of mean-field interacting particles. The present work could be seen as a first result in this direction. In particular, we provide in Theorem~\ref{th:main} a rigorous and quantitative justification to a phenomenon already observed by Balmforth and Sassi \cite{Balmforth2000} on the basis of numerical simulations.

\subsection{The stochastic Kuramoto model with disorder}
We address in this paper the long time behavior of the Kuramoto model with noise and disorder, which describes the evolution of a population of rotators (the $j^\text{th}$ rotator being defined by its phase $\varphi^{ \omega}_j(t)\in \bbT:=\bbR/2\pi \bbZ$), given by the system of $N\geq1$ stochastic differential equations of mean-field type
\begin{equation}
\label{eq:eds_Kur_general}
 \dd \varphi^\go_j(t)\,
 =\, \gd\go_j\dd t -\frac{K}{N} \sum_{l=1}^{N} \sin(\varphi^\go_j(t)-\varphi^\go_l(t))\dd t +\gs \dd B_j(t),\  j=1, \ldots, N,\ t\geq0\, ,
\end{equation}
where $(B_j)_{j=1,\ldots,N}$ is a family of standard independent Brownian motions, $K$, $\gs$ and $\gd$ are positive parameters. In particular, $ \delta>0$ is a scaling parameter. The main result will be stated for small $ \delta>0$, as it relies on perturbation results of the case where $ \delta=0$.

The Kuramoto model \cite{Acebron2005,Kuramoto1975,Strogatz2000} is the main prototype for synchronization phenomena and, due to its mathematical tractability, has been studied in details in the past years \cite{BGP,collet,MR3207725,GPP2012}.
\begin{remark}
Note that \eqref{eq:eds_Kur_general} is invariant by rotation: if $(\varphi^\go_j(t))_{j=1,\ldots,N}$ solves \eqref{eq:eds_Kur_general}, then so does $(\varphi^\go_j(t)+\ga)_{j=1,\ldots,N}$ for all $\ga\in \bbR$. Moreover, by the change of variables $t \to t/\sigma^{ 2}$, one can get rid of the coefficient $ \sigma$ in front of the Brownian motions (up to the obvious modifications $\gd \to \delta/ \sigma^{ 2}$ and $K \to K/ \sigma^{ 2}$). Hence, with no loss of generality, we suppose $\gs=1$ in the following.
\end{remark}
Following the point of view adopted at the beginning of this introduction, the system \eqref{eq:eds_Kur_general} presents two types of noise: in addition to the thermal noise $(B_{ j})$, the \emph{disorder} in \eqref{eq:eds_Kur_general} is given by a sequence $(\go_j)_{j=1,\ldots,N}$ of i.i.d random variables with distribution~$ \lambda$, independent from the Brownian motions. Each $ \omega_{ j}$ represents an intrinsic inhomogeneous frequency for the rotator $ \varphi_{ j}^{ \omega}$.  The index $\go$ in the notation $\varphi_j^\go$ is used to emphasize the dependency of the system in the disorder. 

A crucial aspect in the understanding of the dynamics of \eqref{eq:eds_Kur_general} concerns the (possible lack of) symmetry of the sequence $( \omega_{ j})_{ j\geq1}$. First note that, by the obvious change of variables $ \varphi_{ j}^{ \omega}(t) \mapsto \varphi_{ t}^{ \omega}(t) - \bbE( \omega) t$ in \eqref{eq:eds_Kur_general}, it is always possible to assume that the expectation of the disorder $\bbE( \omega)= \int_{ \bbR} \omega \lambda(\dd \omega)$ is zero (otherwise, we observe macroscopic traveling waves with speed $\bbE( \omega)$). The asymmetry of the disorder can be given at different scales. The most simple situation corresponds to a \emph{macroscopic asymmetry}, that is when the law $ \lambda$ itself is asymmetric. With no loss of generality, we can for example assume that, on a macroscopic level, a majority of rotators will be associated to a positive frequency whereas a minority will have negative frequencies. In the limit of an infinite population, this asymmetry makes the whole system rotate 
at a constant speed that depends only on the law $ \lambda$ and this rotation is noticeable at the scale of the nonlinear Fokker-Planck equation \eqref{eq:FP_Kur_general} associated to \eqref{eq:eds_Kur_general}. This case has been the object of a previous paper (see \cite{MR3207725}, Theorem~2.2 and Section~\ref{sec:links_existing_models} below).

The present paper is concerned with the situation where the law of the disorder is symmetric. Here, the previous argument cannot be applied since in the limit as $N\to \infty$, the population is equally balanced between positive and negative frequencies: the macroscopic speed of rotation found in \cite{MR3207725}, Theorem~2.2 vanishes. Hence, the analysis of long time dynamics of \eqref{eq:eds_Kur_general} requires a deeper understanding of the \emph{microscopic asymmetry} of the disorder, that is the finite-size fluctuations of the disorder w.r.t. the thermal noise. An informal description of the dynamics of \eqref{eq:eds_Kur_general} is the following (see Figure~\ref{fig:simu} below): if the constant $K$ is sufficiently large, the mean-field coupling term leads to synchronization of the whole system along a nontrivial density. Even if $ \lambda$ is symmetric, finite-size fluctuations of the sample $( \omega_{ j})_{ j=1, \ldots, N}$ make it \emph{not} symmetric so that the fluctuations of the disorder 
compete with 
the fluctuations of the Brownian motions $(B_{ j})_{ j=1, \ldots, N}$ and make the whole system rotate with speed and direction depending on the fixed realization of the disorder $(\omega_{ j})$ (and not only on the law $ \lambda$ itself). The main point of the paper is to give a rigorous meaning to this phenomenon, noticed numerically in \cite{Balmforth2000}: we will show that at times of order $\sqrt{N}$, the dynamics of \eqref{eq:eds_Kur_general} deviates from its mean-field limit, with the apparition of synchronized traveling waves induced by the finite-size fluctuations of the disorder. We refer to Paragraph~\ref{sec:longtime_intro} below for a precise description of this phenomenon.

We present in the following subsections some well-known properties of \eqref{eq:eds_Kur_general} which are needed to state our result. We describe in particular its infinite population limit on bounded time intervals and the existence of stationary measures for the limit system in case of symmetric disorder.

\subsection{Mean-field limit on bounded time intervals}
All the statistical information of \eqref{eq:eds_Kur_general} is contained in the empirical measure $(\mu^\go_{N,t})_{t\geq0}\in C([0,\infty),\cM_{ 1}(\bbT\times \bbR))$ ($\cM_{ 1}$ being the set of probability measures endowed with its weak topology) defined as
\begin{equation}\label{eq:def_mu_general}
 \mu^\go_{N,t}\, :=\, \frac{ 1}{ N}\sum_{j=1}^N \gd_{(\varphi^\go_j(t),\go_j)},\ t\geq0\, .
\end{equation}
When the distribution $ \lambda$ of the disorder satisfies $\int |\go| \lambda(\dd\go)<\infty$ and the initial condition $\mu^\go_{N,0}$ converges weakly to some $ p_0$ when $N\rightarrow \infty$, it is easy to see (\cite{daiPra96,Lucon2011}) that the empirical measure \eqref{eq:def_mu_general} converges weakly on bounded time intervals (that is in $C([0,T],\cM_{ 1}(\bbT\times \bbR))$ for all $T\geq0$) to a deterministic limit measure whose density $p_{ t}$ with respect $\ell\otimes \lambda$ (where $\ell$ denotes the Lebesgue measure on $\bbT$) satisfies the following system of nonlinear Fokker-Planck PDEs:
\begin{equation}\label{eq:FP_Kur_general}
 \partial_t p_t(\theta,\go)\, =\, \frac12 \partial^2_\theta p_t(\theta,\go)- \partial_\theta\Big(p_t(\theta,\go)\big(\langle J*p_t\rangle_{ \lambda}(\theta)+\gd\go\big)\Big),\ \omega\in\Supp(\lambda),\ \theta\in\bbT,\ t\geq0\, ,
\end{equation}
where 
\begin{equation}
\label{eq:def_J}
J(\theta)\, :=\, -K\sin(\theta)\, ,
\end{equation} and $\langle \cdot\rangle_{ \lambda}$ represents the integration with respect to $ \lambda$: $\langle J*u\rangle_{ \lambda}(\theta)=\int_\bbR \int_\bbT J(\psi) u(\theta-\psi,\go)\dd\psi \lambda(\dd\go) $. We insist on the fact that in \eqref{eq:FP_Kur_general}, $\go$ is a real number in the support of $ \lambda$, while in \eqref{eq:eds_Kur_general} and \eqref{eq:def_mu_general}, it is an index emphasizing the dependency in the disorder of the system.

Some properties of system \eqref{eq:FP_Kur_general} are detailed in \cite{MR3207725}. In particular, if $\gl$-almost surely, $p_0(\cdot,\go)$ is a probability measure then \eqref{eq:FP_Kur_general} admits a unique solution $p_t$ for all $t>0$ such that $\gl$-almost surely, $p_t(\cdot,\go)$ is also a probability measure, with positive density with respect to the Lebesgue measure and is an element of $C^\infty((0,\infty)\times \bbT,\bbR)$.

\subsection{Symmetric disorder}
\label{sec:finite_disorder}
As already mentioned, we consider the case where the law $ \lambda$ of the disorder is \emph{symmetric}. We restrict our analysis to finite disorder: fix $d\geq1$ and suppose that the frequencies $(\go_j)_{ j\geq1}$ take their values in $\{\go^{-d},\go^{-(d-1)},\ldots,\go^{d-1},\go^d\}$, where $\go^i=-\go^{-i}$ for all $i=0, \ldots, d$. We denote as $ (\lambda^{ i}\in [0, 1], i=-d, \ldots, d)$ the probability of drawing each $ \omega^{ i}$ and assume that $ \lambda^{ i}= \lambda^{-i}$ for all $i=1, \ldots, d$. From now on, the law of the disorder $ \lambda$ is identified with $( \lambda^{-d},\ldots, \lambda^{ d})$.
Note that we may suppose in the following that $\omega_{ 0}=0\not\in\Supp( \lambda)$. The result still holds with obvious changes in notations.

Under this hypothesis, almost surely, for sufficiently large $N$, each possible value $\go^i$ of the disorder appears at least once and we can rewrite \eqref{eq:eds_Kur_general} by regrouping the rotators into $(2d+1)$ sub-populations: for all $i=-d, \ldots, d$, denote as $N^{ i}$ the number of rotators $(\varphi^i_j(t))_{j=1,\ldots,N^i}$ with frequency $ \omega^{ i}$. Obviously, $N=\sum_{ i=-d}^{ d} N^{ i}$ and the system \eqref{eq:eds_Kur_general} becomes
\begin{equation}\label{eq:eds_Kur}
 \dd \varphi^i_j(t)\, =\, \gd\go^i\dd t 
 -\frac{K}{N}\sum_{k=-d}^d \sum_{l=1}^{N^k} \sin(\varphi^i_j(t)-\varphi^k_l(t))\dd t +\dd B^i_j(t)\, ,\ j=1,\ldots, N^{ i},\ i=-d, \ldots, d\, .
\end{equation}
In this framework, the empirical measure $\mu^\go_{N,t}$ in \eqref{eq:def_mu_general} can be identified with $(\mu^{-d}_{N,t},\ldots,\mu^{d}_{N,t})$, where $ \mu^{ i}_{ N}$ is the empirical measure of the rotators with frequency $ \omega^{ i}$:
\begin{equation}\label{eq:def_mu_finite}
 \mu^i_{N,t}\, =\, \frac{1}{N^i} \sum_{j=1}^{N^i} \gd_{\varphi^i_j(t)},\ t\geq0,\ i=-d,\ldots,d\, ,
\end{equation}
and its mean-field limit \eqref{eq:FP_Kur_general} can be identified with $p_{t}=(p^{-d}_t,\ldots,p^d_t)$, solution to
\begin{equation}\label{eq:FP_Kur_finite}
 \partial_t p^i_t(\theta)\, =\, \frac12 \partial^2_\theta p^i_t(\theta)- \partial_\theta\left(p^i_t(\theta)\left(\sum_{k=-d}^d \lambda^k J*p^k_t (\theta)+\gd\go^i\right)\right),\ t\geq0,\ i=-d, \ldots, d\, .
\end{equation}
\subsection{Stationary solutions and phase transition}
\label{sec:stationary_solutions_disorder}
A remarkable aspect of the Kuramoto model is that one can compute semi-explicitly the stationary solutions of \eqref{eq:FP_Kur_finite}, when $ \lambda$ is symmetric (see e.g. \cite{Sakaguchi1988}): each stationary solution to \eqref{eq:FP_Kur_finite} is the rotation of a profile $q=(q^{-d},\ldots,q^d)$ (i.e. given by $q(\cdot+\ga)$ for some $\ga\in \bbT)$) of the form
\begin{equation}\label{eq:def_q}
 q^i(\theta)\, =\, \frac{S^i_\gd(\theta,2Kr)}{Z^i_\gd(2Kr)}\, ,
\end{equation}
where for each $i=-d,\ldots,d$, $q^i(\cdot)$ is a probability density on $\bbT$, $S^i_\gd(\theta,2Kr)$ is given by
\begin{multline}
\label{eq:S_delta}
S^i_\gd(\theta,x)\, =\, e^{x\cos \theta+2\gd \go^i\theta}\bigg[(1-e^{4\pi \gd \go^i})
 \int_0^\theta e^{-x\cos u-2\gd\go^i u}\dd u\\
 +e^{4\pi \gd \go^i}\int_0^{2\pi} e^{-x\cos u -2\gd\go^i u}\dd u\bigg]\, ,
\end{multline}
$Z^i_\gd(2Kr)$ is a normalization constant and $r$ is a solution of the fixed-point problem
\begin{equation}
\label{eq:fixed_point}
 r\, =\, \Psi_\gd(2Kr)\, ,
\end{equation}
with
\begin{equation}\label{eq:fixed_point_Psi}
 \Psi_\gd(x)\, =\, \sum_{k=-d}^d \lambda^k\frac{\int_0^{2\pi} \cos (\theta) S^k_\gd(\theta,x)\dd \theta}{Z^k_\gd(x)}\, .
\end{equation}
We refer to \cite{Sakaguchi1988} or \cite{Luconthesis}, p.~75 for more details on this calculation. Computing the solution to the fixed-point relation \eqref{eq:fixed_point} enables to exhibit a phase transition for \eqref{eq:FP_Kur_finite}: the value $r=0$ always solves \eqref{eq:fixed_point} and corresponds to the uniform stationary solution $q\equiv(1/2\pi,\ldots,1/2\pi)$. It is the only stationary solution to \eqref{eq:FP_Kur_finite} as long as $K \leq K_c$, for a certain critical parameter $K_c= K_{ c}( \delta, (\omega^{ i})_{ i}, (\lambda^{ i})_{ i})>1$. This characterizes the absence of synchrony in case of small interaction. When $K> K_{ c}$, this flat profile coexists with circles of synchronized solutions corresponding to positive fixed-points in \eqref{eq:fixed_point}: each solution $r>0$ to \eqref{eq:fixed_point} gives rise to a nontrivial stationary profile $q$ given by \eqref{eq:def_q}
and to the circle of all its translation $q(\cdot+\ga)$, by invariance by rotation of the system (see Figure~\ref{fig:fixedpoint}).
\begin{figure}[h]
\centering
\subfloat[Correspondance between fixed-points of $ \Psi_{ \delta}(\cdot)$ and stationary solutions to \eqref{eq:FP_Kur_finite}.]{\includegraphics[width=0.45\textwidth]{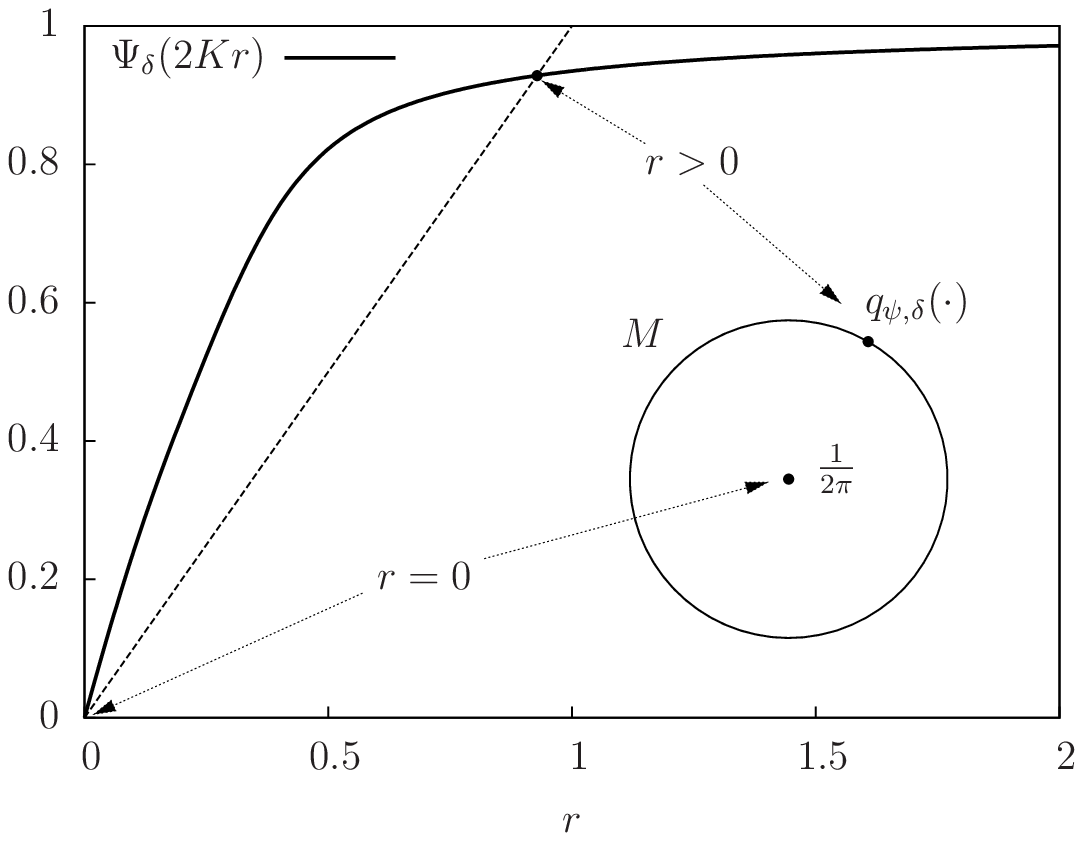}
\label{subfig:fonctionPsi}}
\quad\subfloat[A synchronized profile with $d=2$, $q=(q^{ -2}, q^{ -1}, q^{ 1}, q^{ 2})$.]{\includegraphics[width=0.48\textwidth]{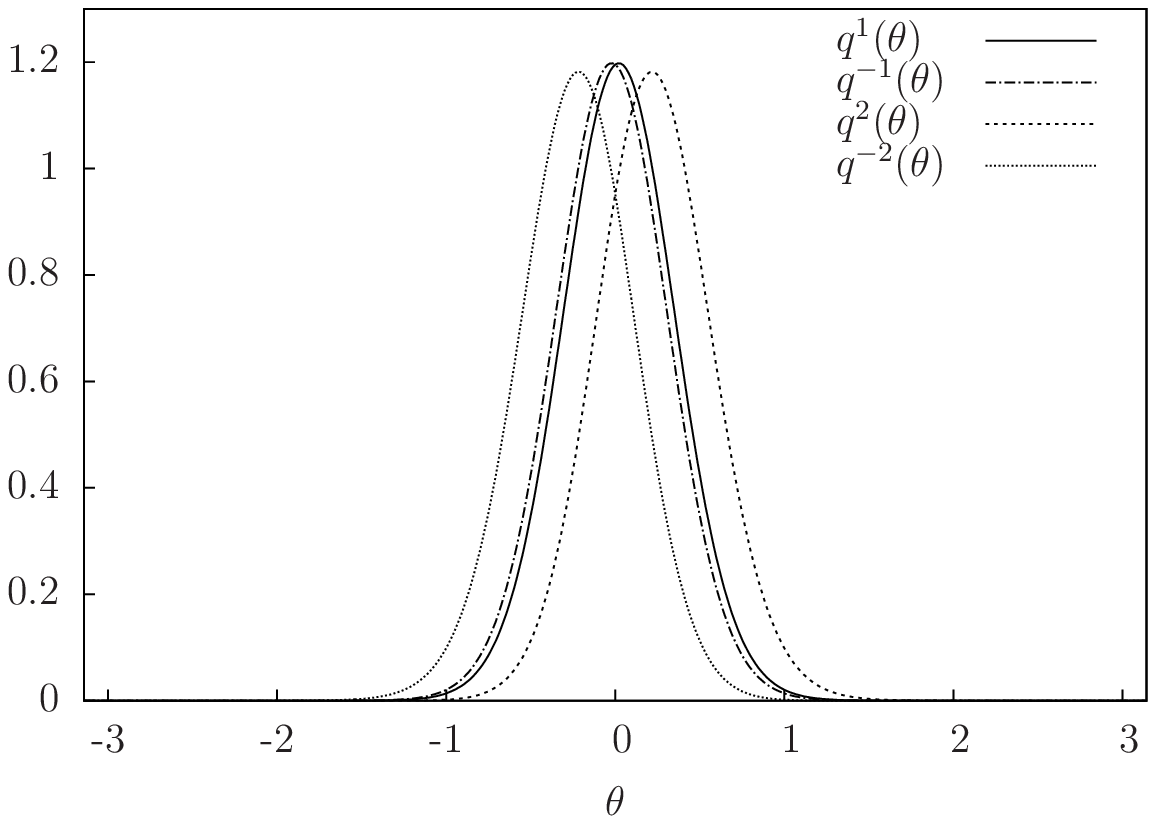}
\label{subfig:fonctionq}}
\caption{Fixed-point function $ \Psi_{ \delta}(\cdot)$ and stationary profiles when $K=5$, $d=2$, $ \omega_{ 1}=1$, $ \omega_{ 2}=10$ and $ \delta=0.1$.}%
\label{fig:fixedpoint}
\end{figure}

However, several circles may coexist when $K>K_c$ and these circles may not be locally stable (even the characterization of these circles in full generality is unclear, see e.g. \cite{Luconthesis}, \S~2.2.2). To ensure uniqueness and stability of a circle of non-trivial profiles, fix $K>1$ and restrict to small values of $\gd$: it is indeed proved in \cite{MR3207725}, Lemma~2.3 that there exists $ \delta_{ 1}=\gd_1(K)>0$ such that for all $\gd\leq \gd_1$, the fixed-point problem \eqref{eq:fixed_point} admits a unique positive solution $r_\gd$. We denote by $q_{0,\gd}$ the corresponding profile given by \eqref{eq:def_q} with $r=r_{ \delta}$, by $q_{\psi,\gd}$ its rotation of angle $\psi\in\bbT$ (i.e. $q_{\psi,\gd}(\cdot):=q_{0,\gd}(\cdot-\psi)$) and by $M$ the corresponding circle of stationary profiles (see Figure~\ref{fig:fixedpoint}):
\begin{equation}\label{eq:def_M}
 M\, :=\, \{q_{\psi,\gd}\, : \psi\in \bbT\}\, .
\end{equation}
It is proved in \cite{MR3207725}, Theorems~2.2 and~2.5 that the circle $M$ is stable under the evolution \eqref{eq:FP_Kur_finite}: the solution of \eqref{eq:FP_Kur_finite} starting from an initial condition sufficiently close to $M$ converges to a element $q_{\psi,\gd}$ of $M$ as $t\to\infty$. More details about this stability are given in Section \ref{sec:lin stab}. Whenever it is clear from the context, we will use the notations $q_{\gd}$ or $q_{\psi}$ instead of $q_{\psi,\gd}$, depending on the parameter we want to emphasize. 

\subsection{Long time behavior}
\label{sec:longtime_intro}
\begin{figure}[h]
\includegraphics[width=10cm]{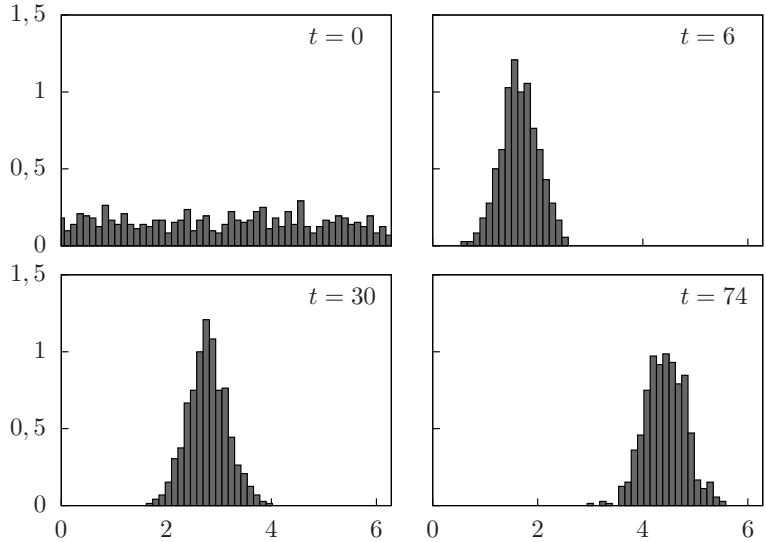}
\caption{Evolution of the marginal of the empirical measure \eqref{eq:def_mu_general} on $\bbT$ for a fixed choice of the disorder ($N=600$, $\lambda=\frac{1}{2}(\delta_{-1}+ \delta_{1})$, $K=6$). Starting from uniformly distributed rotators on $\bbT$ ($t=0$), the empirical measure converges to a synchronized profile on the manifold $M$ ($t=6$) and then moves (here to the right) at a constant speed, on a time scale compatible with $N^{ 1/2}$.}
\label{fig:simu}
\end{figure}
\begin{figure}[h]
\includegraphics[width=0.7\textwidth]{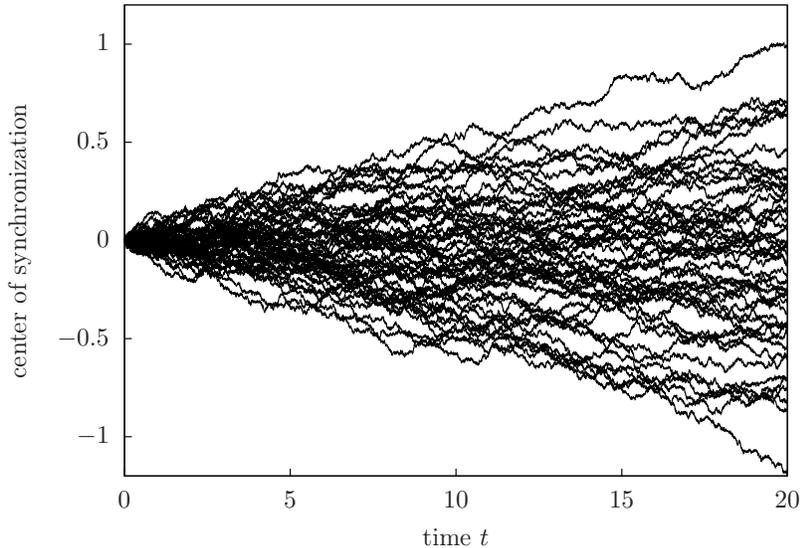}
\caption{Trajectories of the center of synchronization for different realizations of the disorder ($\lambda=\frac{1}{2}(\delta_{-0.5}+ \delta_{0.5})$, $K=4$, $N=400$).}
\label{fig:drift}
\end{figure}
Simulations of \eqref{eq:eds_Kur} (Figure \ref{fig:simu}) suggest an initial transition of the system from an incoherent state to a synchronized one, during  which the empirical measures of the rotators approaches the circle $M$ of synchronized stationary profiles. Secondly, the empirical measure remains close to $M$ and travels at first order at constant speed (which is random, depending on the realization of the disorder, see Figure~\ref{fig:drift}) along $M$ on the time scale $N^{1/2}t$. Let us give some intuition of this phenomenon: to fix ideas, consider the case where $d=1$, $ \omega^{ 1}= -\omega^{ -1}=1$ and $ \lambda^{ -1}= \lambda^{ 1}= \frac{ 1}{ 2}$. This corresponds to the simplest decomposition in \eqref{eq:eds_Kur} between two subpopulations, one naturally rotating clockwise ($ \omega_{ i}=+1$) and the second rotating anti-clockwise ($ \omega_{ i}=-1$). One can imagine that fluctuations in the finite sample $(\omega_1, \ldots, \omega_N)\in \{\pm1\}^N$ may lead, for example, to a majority 
of $+1$ with respect to $-1$, so that the rotators with positive frequency induce a global rotation of the whole system in the direction of the majority. When $N$ is large, this asymmetry is small, typically of order $N^{ -1/2}$ and is not sufficient to make the empirical measure drift away from the attracting manifold $M$, but induces a small drift that becomes macroscopic at times of order $N^{ 1/2}$. 

The purpose of the paper is precisely to prove the existence of this random traveling wave and show that it is indeed an effect of the fluctuations of the disorder. Our approach consists in a precise analysis of the dynamics of the empirical measure \eqref{eq:def_mu_finite}, which involves both disorder and thermal noise. One of the main difficulties is to control the thermal noise term and prove that it does not play any role at first order on the $N^{1/2}$-time scale.

\section{Main results and strategy of proof}
\subsection{The result}
\subsubsection{Admissible sequence of disorder}
We stress the fact that the random traveling waves described above is essentially a \emph{quenched} phenomenon, that is, true for a fixed realization of the disorder $( \omega_{ i})_{ i\geq1}$. In particular, the result does not really depend on the underlying mechanism that produced the sequence $( \omega_{ i})_{ i\geq1}$, it only depends on the asymmetry of this sequence. We prove our result for any {\sl admissible} sequence of disorder $( \omega_{ i})_{ i\geq1}$, defined as follows.
\begin{definition}
\label{def:disorder}
Fix a sequence $(\omega_{ i})_{ i\geq1}$ taking values in $\{ \omega^{ -d}, \omega^{ -(d-1)}, \ldots, \omega^{ d-1}, \omega^{ d}\}$ and for all $N\geq1$,  define the empirical proportions of frequencies in the $N$-sample $( \omega_{ 1}, \ldots, \omega_{ N})$
\begin{equation}
\label{eq:def_pNk}
\lambda_{ N}^k\, :=\, \frac{N^k}{N},\ k=-d,\ldots, d\, ,
\end{equation}
where $N^{ k}$ is the number of rotators with frequencies equal to $ \omega^{ k}$ (recall Section~\ref{sec:finite_disorder}). Define also the fluctuation process associated to $( \omega_{ i})_{ i\geq1}$ by $ \xi_{ N}:=( \xi_{ N}^{ -d}, \ldots, \xi_{ N}^{ d})$, where 
\begin{equation}
\label{eq:def_xiNk}
 \xi^k_N\, :=\, N^{1/2}( \lambda_{ N}^k- \lambda^k),\ k=-d,\ldots, d,\ N\geq1\, ,
\end{equation}
where $( \lambda^{ -d}, \ldots, \lambda^{ d})$ is given in Section~\ref{sec:finite_disorder}. Note that $\sum_{ k=-d}^{ d} \xi_{ N}^{ k}=0$ for all $N\geq1$. We say that the sequence $(\omega_{ i})_{ i\geq1}$ is \emph{admissible} if the following holds
\begin{enumerate}
\item \emph{Law of large numbers}: for all $k=-d, \ldots, d$, $ \lambda_{ N}^{ k}$ converges to $ \lambda^{ k}$, as $N\to\infty$.
\item \emph{Central limit behavior}: for all $ \zeta>0$, there exists $N_{ 0}$ (possibly depending on the sequence $(\omega_{ i})_{ i\geq1}$) such that for all $N\geq N_{ 0}$,\[  \max_{ k=-d, \ldots, d}\left\vert \xi_{ N}^{ k} \right\vert\leq N^{ \zeta}\, .\]
\end{enumerate}
\end{definition}
\begin{remark}[Admissibility for i.i.d. variables]
An easy application of the Borel-Cantelli Lemma shows that any independent and identically distributed sequence of disorder $(\omega_{ i})_{ i\geq1}$ with law $ \lambda$ is almost surely admissible, in the sense of Definition~\ref{def:disorder}.
\end{remark}
\subsubsection{Main result}
From now on, we fix once and for all an admissible sequence $(\omega_{ i})_{ i\geq1}$ in the sense of Definition~\ref{def:disorder}. A convenient framework for the analysis of the dynamics of \eqref{eq:def_mu_finite} and \eqref{eq:FP_Kur_finite} corresponds to the space $H^{-1}_d$, dual of the space $H_{ d}^{ 1}$, which is the closure of the set of vectors $(u^{-d},\ldots,u^d)$ of regular functions $u^k$
with zero mean value on $\bbT$ under the norm 
\begin{equation}
\label{eq:norm_H1_intro}
\left\Vert u \right\Vert_{ 1, d}\, :=\,   \left(\sum_{ k=-d}^{ d} \lambda^{ k}\int_{ \bbT} (\partial_{ \theta} u^{ k}( \theta))^{ 2} \dd \theta\right)^{ 1/2}\, .
\end{equation}
Remark that if $u$ is a vector of probability measures on $\bbT$, then $u$ naturally belongs to $H^{-1}_d$, since the family
of vectors given by $a_{n,k}(\theta)= (0,\ldots,0,\frac{\sqrt{2} cos(n\theta)}{n\sqrt{\gl^k}},0,\ldots,0)$
and $b_{n,k}(\theta)=(0,\ldots,0,\frac{\sqrt{2} sin(n\theta)}{n\sqrt{\gl^k}},0,\ldots,0)$
form an orthonormal basis of $H^1_d$ and for each such vector $u$
\begin{equation}\label{eq:proba in Hminus1}
 \Vert u\Vert_{-1,d}\, =\, \sqrt{\sum_{k=-d}^d \sum_{n=1}^\infty \big(\langle u, a_{n,k}\rangle^2+\langle u,b_{n,k}\rangle^2\big)}
 \, \leq\,  \pi\sqrt{\frac{2}{3}\sum_{k=-d}^d (\gl^k)^{-1}}\, .
\end{equation}
More details on the construction of $H_{ d}^{ -1}$ are given in Appendix~\ref{sec:appendix_rigged_spaces}. The main result of the paper is the following.
\begin{theorem}\label{th:main}
For all $K>1$, there exists $ \delta(K)$ such that, for all $ \delta\leq \delta(K)$, there exists a linear form $b:\bbR^{ 2d+1}\rightarrow \bbR$ (depending in $K$, $\gd$, the probability distribution $ \lambda$ and the possible values of the disorder $\go^i$) and a real number $\gep_0>0$ such that, for any admissible sequence $(\omega_{ i})_{ i\geq1}$, any vector of probability measures $p_0$ satisfying $\text{dist}_{H^{-1}_d}(p_0,M)\leq \gep_0$ such that for all $\gep>0$,
 \begin{equation}
  \bP \left( \left\Vert \mu_{N,0} - p_0  \right\Vert_{-1,d} \geq \gep \right)  \rightarrow 0,\, \text{as $N\to \infty$}\, , 
 \end{equation}
then, there exists $ \theta_0\in \bbT$ (depending on $p_0$) and a constant $c$
such that for each finite time $t_f>0$ and all $\gep>0$, denoting $t_0^N=cN^{-1/2}\log N$, we have
 \begin{equation}
 \label{eq:main_th_drift}
   \bP \left( \sup_{t\in [t_0^N,t_f]}\left\Vert \mu_{N,N^{1/2} t} - q_{ \theta_0+b(\xi_N)t}
  \right\Vert_{-1,d} \geq \gep  \right)
  \rightarrow 0,\, \text{as $N\to\infty$}\, .
 \end{equation}
 Moreover, $ \xi \mapsto b( \xi)$ has the following expansion in $ \delta$: for all $\xi$ such that $\sum_{k=-d}^d \xi^k=0$, we have
  \begin{equation}
  b(\xi)\, =\, \gd \sum_{k=-d}^d \xi^k \go^k+O(\gd^2)\, .
 \end{equation}
\end{theorem}
Theorem~\ref{th:main} is simply saying that, on a time scale of order $N^{ 1/2}$, the empirical measure \eqref{eq:def_mu_finite} is asymptotically close to a synchronized profile $q\in M$, traveling at speed $b( \xi_{ N})$ along $M$. This drift depends on the asymmetry $ \xi_{ N}$ of the quenched disorder $( \omega_{ i})_{ i\geq1}$. In \eqref{eq:main_th_drift}, $ t_{ 0}^{ N}$ represents the time necessary for the system to get sufficiently close to the manifold~$M$. 
\subsubsection{Some particular cases and extensions}
First remark that the situation where the sample of the disorder $( \omega_{ i})_{ i= 1, \ldots, N}$ is perfectly symmetric corresponds to $ \xi_{ N}^{ -i}= \xi_{ N}^{ i}$ for all $i= 1, \ldots, d$. In this case, the drift in \eqref{eq:main_th_drift} vanishes:
\begin{proposition}
\label{prop:drift_symmetric}
If for all $i=1, \ldots, d$, $ \xi^{ -i} = \xi^{ i}$, then $b( \xi)=0$.
\end{proposition}
\noindent
In particular, if one chooses the disorder in such a way that $(\omega_{ i})_{ i=1, \ldots, N}$ is always symmetric (e.g. choose an even number of particles $N$ and define each $ \omega_{ i}$ to be alternatively $\pm 1$), the drift is always zero. We believe in this case that one would need to look at larger time scales of order $N$ to see the first order of the expansion of the empirical measure $ \mu_{ N}$. Proof of Proposition~\ref{prop:drift_symmetric} is given in Section~\ref{sec:drift_symmetric}.

\medskip

In case the sequence $( \omega_{ i})_{ i\geq1}$ is i.i.d. with law $ \lambda$, a standard Central Limit Theorem shows that the drift $b( \xi_{ N})$ converges in law to a Gaussian distribution $\cN(0,v^2)$, where $v^2$ depends on $K$, $\gd$, the probability distribution $ \lambda$ and the possible values of the disorder $\go^i$. 
\begin{proposition}\label{prop:xpansion v}
The following asymptotic of $v^{ 2}$ holds when $\gd\rightarrow 0$:
\begin{equation}
 v^2\, =\, \gd \sum_{k=-d}^d \lambda^k(\go^k)^2 +O(\gd^2) \, .
 \end{equation}
\end{proposition}
Proof of Proposition~\ref{prop:xpansion v} is given in Section~\ref{sec:asymptotic_drift}.
\begin{remark}
Without much modification in the proof, the result can be easily extended to sequences $(\omega_{ i})_{ i\geq1}$ with fluctuations of order different from $ \sqrt{N}$, that is when for some $a\in(0,1)$,
\begin{equation}
 \xi^a_N \rightarrow \xi^a, \text{as $N\to\infty$}\, ,
\end{equation}
for some vector $\xi^a$ where $\xi_{ N}^a:= N^a(\lambda_N- \lambda)$. In this case, the correct time renormalization is $N^a$ and we obtain a result of the type
 \begin{equation}
   \bP \left( \sup_{t\in [t_0^N,t_f]}\left\Vert \mu_{N,N^{a} t} - q_{ \theta_0+b(\xi^a)t}
  \right\Vert_{-1,d} \geq \gep  \right)
  \rightarrow 0, \text{as $N\to\infty$}\, .
 \end{equation}
Here, we only treat the case $a=1/2$ for simplicity. For smaller fluctuations of size $N^{-a}$ with $a\geq 1$, the time renormalization should be of order $N$. Since at this scale the effects of the thermal noise appear, the limit phase dynamics should be of diffusive type and
a precise analysis of the different terms and symmetries that occur would be necessary to get the proper drift in this case.
\end{remark}

\subsection{Links with existing models}
\label{sec:links_existing_models}
\subsubsection{Symmetric versus non-symmetric disorder}
This work is the natural continuation of \cite{MR3207725}, Theorems~2.2 and 2.5 in the case of a symmetric disorder. The purpose of \cite{MR3207725} was to analyze the dynamics of the nonlinear Fokker-Planck equation \eqref{eq:FP_Kur_general} for both symmetric and asymmetric law of the disorder. The main point is that understanding \eqref{eq:FP_Kur_general} is not sufficient in itself for the analysis of the finite size system \eqref{eq:eds_Kur_general} in the symmetric case, since it does not account for the finite-size effects of the disorder that are crucial here.

As already mentioned, in the case where $ \lambda$ is asymmetric, one observes macroscopic travelling waves with deterministic drift at the scale of the nonlinear Fokker-Planck equation \eqref{eq:FP_Kur_general}. It is reasonable to think that an analysis similar to what has been done in this paper  would also show the existence of a finite order correction to this deterministic drift for a large but finite system with quenched disorder.

Some previous results already suggested the possibility of these disorder-induced traveling waves in the Kuramoto model. Namely, the purpose of previous work \cite{Lucon2011} was to prove a quenched fluctuation result for the empirical measure \eqref{eq:def_mu_general} around its mean-field limit \eqref{eq:FP_Kur_general} on a finite time horizon $[0, T]$. The main conclusion of \cite{Lucon2011} was that these fluctuations are disorder dependent and the long time analysis of the limiting fluctuations \cite{Lucon20146372} suggested a non-self-averaging phenomenon for \eqref{eq:eds_Kur_general} similar to the one observed here.

\subsubsection{The case $ \delta = 0$}

This paper uses techniques previously developed in \cite{Bertini:2013aa} in the context of the stochastic Kuramoto model without disorder, that is when one takes $\gd=0$ in \eqref{eq:eds_Kur_general}:
\begin{equation}\label{eq:Kur_without_disorder}
  \dd \varphi_j(t)\,
 =\,  -\frac{K}{N} \sum_{l=1}^{N} \sin(\varphi_j(t)-\varphi_l(t))\dd t + \dd B_j(t)\, ,\ j=1, \ldots, N\, ,
\end{equation}
associated in the limit $N\to \infty$ to the mean-field PDE
\begin{equation}\label{eq:FP_without_disorder}
  \partial_t p_t(\theta)\, =\, \frac12 \partial^2_\theta p_t(\theta)- \partial_\theta\Big( p_t(\theta)J*p_t(\theta)\Big)\, .
\end{equation}
Similarly to \eqref{eq:FP_Kur_finite} in Section~\ref{sec:stationary_solutions_disorder}, evolution \eqref{eq:FP_without_disorder} generates a stable circle $M_{ 0}$ of stationary synchronized profiles when $K>K_{ c}(0)=1$ (see Section~\ref{sec:case_delta_0} for further details). The model \eqref{eq:Kur_without_disorder}-\eqref{eq:FP_without_disorder} has been the subject of a series of recent papers \cite{BGP,Bertini:2013aa,GPP2012,doi:10.1137/110846452}, addressing the linear and nonlinear stability of the circle of synchronized profiles $M_{ 0}$ as well as the long time dynamics of the microscopic system \eqref{eq:Kur_without_disorder}. The analysis of \eqref{eq:FP_without_disorder} strongly relies on the reversibility of \eqref{eq:Kur_without_disorder} (with the existence of a proper Lyapunov functional, see \cite{BGP} for more details), whereas reversibility is lost when $ \delta>0$.

Concerning the long time behavior of \eqref{eq:Kur_without_disorder}, it is shown in \cite{Bertini:2013aa} that under very general hypotheses on the initial condition, the empirical measure of \eqref{eq:Kur_without_disorder} first approaches the circle $M_0$ exponentially fast (that corresponds to the synchronization of the system \eqref{eq:Kur_without_disorder} along a stationary profile solving \eqref{eq:FP_without_disorder}) and then stays close to $M_0$ for a long time with high probability, while the phase of its projection on $M_0$ performs a Brownian motion as $N\rightarrow \infty$
which corresponds to a macroscopic effect of the thermal noise.
The persistence of proximity of the empirical measure to $M$ for long times and the convergence of this phase to a Brownian motion were in fact already established in the unpublished PhD Thesis \cite{dahms} the authors of \cite{Bertini:2013aa} were not aware of, using in particular moderate deviations estimates of the mean field process. Note that the techniques of \cite{dahms} do not apply here, since a similar analysis would involve moderate (or large) deviations in a quenched set-up, result that, to the best of our knowledge, has not been proven so far (for \emph{averaged} large deviations, see \cite{daiPra96}).

A significant difference between \cite{Bertini:2013aa,dahms} and the present analysis is that the Brownian excursions in \cite{Bertini:2013aa,dahms} occur on a time scale of order $N$ whereas it is sufficient to look at times of order $N^{ 1/2}$ to see the traveling waves in the disordered case. This will entail significant simplifications in the analysis of \eqref{eq:eds_Kur}, since the detailed analysis on the thermal noise performed in \cite{Bertini:2013aa,dahms} will not be required here.

Note also that, contrary to \cite{Bertini:2013aa}, we do not prove the first step of the phenomenon described in Figure~\ref{fig:simu}, that is the initial approach of the system to a neighborhood of the manifold $M$ in an exponentially short time, regardless of the initial condition. This result would require a global stability result for the system of PDEs \eqref{eq:FP_Kur_finite} which has not been proved for the moment, due to the absence of any Lyapunov functional for \eqref{eq:FP_Kur_finite} when $ \delta>0$. We prove our result for initial conditions belonging to some macroscopic neighborhood of $M$ (see Section~\ref{sec:approach} for more details).

\subsubsection{SPDE models with vanishing noise}

This paper is related to previous works in the context of SPDE models for phase separation.
In \cite{cf:BDMP,cf:Funaki}, the authors studied the Allen-Cahn model with symmetric bistable potential and vanishing noise.
They showed that for an initial data close a profile connecting the two phase, the interface performs a Brownian motion. Some techniques initially introduced in these works, as the discretization of the dynamics in an iterative scheme, were developed in \cite{Bertini:2013aa} in the context of the Kuramoto model without disorder (making use of Sobolev spaces with negative exponents to deal with empirical measures) and will have a central role in our analysis (see Section~\ref{subsec:dyn M}). The results of \cite{cf:BDMP,cf:Funaki} have been extended in \cite{cf:BBjsp} by considering small asymmetries in the potential which induce a drift in the interface dynamics and by considering macroscopically finite volumes \cite{cf:BBB}, with effect a repulsion at the boundary for the phase. Stochastic interface motions have also been recently studied in the context of the Cahn-Hilliard model with vanishing colored noise \cite{Antonopoulou}. In this model, the limit behavior of the interface is given by a SDE (or system 
of 
SDE's in the case of several interfaces) with drift and diffusion coefficients depending on coloration of the noise and on the length of the interface.

\subsection{Linear stability of stationary solutions}\label{sec:lin stab}
In the whole paper, we suppose that $K>1$ and that $ \delta>0$ is smaller than some $ \delta(K)>0$. This critical value $ \delta(K)$ is determined by $ \delta(K)= \min( \delta_{ 1}(K), \delta_{ 2}(K))$, where $ \delta_{ 1}(K)$ ensures the existence of a unique circle $M$ of stationary solutions (recall Section~\ref{sec:stationary_solutions_disorder}) and where $ \delta_{ 2}(K)$ comes from the stability analysis of this circle (see Appendix~\ref{sec:appendix_spectral_semigroups} for more details). 

More precisely, our result relies deeply on the linear stability of the dynamical system induced by the limit system of PDEs \eqref{eq:FP_Kur_finite} in the neighborhood of the circle of stationary profiles $M$. For $ \psi\in\bbT$, $ \delta>0$, consider the operator $L_{ \psi, \delta}$ of the linearized evolution around $q_{\psi,\gd}\in M$ given by
\begin{equation}
\label{eq:def_Lpsiq}
 (L_{ \psi, \delta} u)^i\, =\, \frac 12 \partial_{ \theta}^{ 2}u^i-\gd \go^i \partial_{ \theta}u^i-\partial_{ \theta}\left(u^i\sum_{k=-d}^d \lambda^k (J\ast q_{\psi,\gd}^k) +q_{\psi,\gd}^i\sum_{k=-d}^d \lambda^k (J* u^k)\right)\, ,
\end{equation}
for all $i=-d, \ldots, d$ with domain
\begin{equation}
\left\{u=(u^{-d},\ldots,u^d)\, :
u^i \in C^2(\bbT)\quad  \text{and}\quad  \int_\bbT u^i(\theta)\dd\theta =0,\quad \forall i=-d,\ldots,d \right\} \, .
\end{equation}
Due to the invariance by rotation of the model \eqref{eq:FP_Kur_finite}, $L_{\psi,\gd}$ is linked to $L_{0,\gd}$ in an obvious way:
$L_{\psi,\delta}u_\psi(\cdot)=L_{0,\gd}u(\cdot)$, where $u_\psi(\cdot)=u(\cdot-\psi)$, so that the operators $(L_{\psi,\gd})_{ \psi\in\bbT}$ obviously share the same spectral properties. For any operator $L$, the usual notations $ \sigma(L)$ (resp. $ \rho(L)$ and $R(\lambda, L)$) will be used for the spectrum of $L$ (resp. its resolvent set and its resolvent operator for $ \lambda\in \rho(L)$).

One can prove (see \cite{MR3207725}, Theorem~2.5 and Appendix~\ref{sec:appendix_spectral_semigroups} below) that for all $ 0\leq \delta\leq \delta_{ 2}(K)$, $L_{ \psi, \delta}$ is closable in $H^{-1}_d$, sectorial, has $0$ for eigenvalue, associated to the eigenvector $\partial_{ \theta} q_{ \psi, \delta}$, which belongs to the tangent space of $M$ in $q_{ \psi, \delta}$ (this reflects the fact that the dynamics induced by \eqref{eq:FP_Kur_finite} on $M$ is trivial) and that the rest of the spectrum is negative, separated from the
eigenvalue $0$ by a spectral gap $\gamma_L>0$. More details about these questions are given in Appendix~\ref{sec:appendix_spectral_semigroups}.

The fact that the eigenvalue $0$ is isolated from the rest of the spectrum $ \sigma(L_{ \psi, \delta})\smallsetminus\{0\}$ implies that $H_{ d}^{ -1}$ can be decomposed into a direct sum $T_{\psi,\gd}\oplus N_{\psi,\gd}$, where $T_{  \psi, \delta}= \Span(\partial_{ \theta} q_{ \psi, \delta})$ such that the spectrum of the restriction of $L_{ \psi, \delta}$ to $N_{ \psi, \delta}$ (resp. $T_{ \psi, \delta}$) is $ \sigma(L_{ \psi, \delta})\smallsetminus\{0\}$ (resp. $\{0\}$). We denote by $P^{ 0}_{ \psi, \delta}$ the projection on $T_{ \psi, \delta}$ along $N_{ \psi, \delta}$ and $P^{ s}_{ \psi, \delta}=1-P^{ 0}_{ \psi, \delta}$. Both $P^{ 0}_{ \psi, \delta}$ and $P^s_{ \psi, \delta}$ commute with $L_{\psi,\gd}$. In particular, for all $\psi\in \bbT$, $ \delta>0$, there exists a linear form $\mathtt{p}_{ \psi, \delta}$ satisfying, for all $u\in H_{ d}^{ -1}$
\begin{equation}\label{eq:def p psi}
 P^{ 0}_{ \psi, \delta} u\, =\,\mathtt{p}_{ \psi, \delta} (u) \partial_{ \theta}q_{ \psi, \delta} \, .
\end{equation}
We also denote by $C_{ P}$ and $C_{ L}$ positive constants such that for all $u\in H_{ d}^{ -1}$, $t>0$:
\begin{align}
\Vert P^{ 0}_{ \psi, \delta} u\Vert_{-1,d}&\leq\,  C_P \Vert  u\Vert_{-1,d}\, ,\\
\Vert P^s_{ \psi, \delta} u\Vert_{-1,d}&\leq\,  C_P \Vert  u\Vert_{-1,d}\, ,\\
\left\Vert e^{t L_{\psi,\gd}} P^s_{ \psi, \delta} u \right\Vert_{-1,d}\, &\leq\, C_L e^{-\gamma_L t}\left\Vert P^s_{ \psi, \delta} u\right \Vert_{-1,d}\, ,\label{ineq:L 1}\\
 \left\Vert e^{t L_{\psi,\gd}} u \right\Vert_{-1,d}\, &\leq\, C_L \left(1+\frac{1}{\sqrt{t}} \right)\left\Vert  u\right \Vert_{-2,d}\, .\label{ineq:L 2}
\end{align}
Inequality \eqref{ineq:L 1} is a consequence of \cite{Henry1981}, Theorem~1.5.3, p. 30 and \eqref{ineq:L 2} is proved in Proposition~\ref{prop:reg_semigroup} in Appendix~\ref{sec:appendix_spectral_semigroups}. Once again, we will often drop the dependency in the parameters $ \psi$ or $ \delta$ in $P^{ 0}_{ \psi, \delta}$ and $P^{ s}_{ \psi, \delta}$ for simplicity of notations.

A consequence of the contraction \eqref{ineq:L 1} along the space $N_{ \psi, \delta}$ is that $M$ is locally stable with respect
to the evolution given by \eqref{eq:FP_Kur_finite} (see for example exercise $6^*$ of the Chapter 6 of \cite{Henry1981}, or Theorem 2.2 of \cite{MR3207725}
for our particular model): for any $p_0$ in a neighborhood of $M$, there exists $\psi\in \bbT$ such that the solution of \eqref{eq:FP_Kur_finite}
converges to $q_{\psi,\gd}$ exponentially fast (with rate given by $\gamma_L$).

\subsection{Dynamics of the empirical measure}
\label{sec:mild_formulation_intro}
The starting point of the proof of Theorem~\ref{th:main} is to write the semi-martingale decomposition (see Proposition~\ref{prop:SPDE_weak}) of the difference between the empirical measure $ \mu_{ N, t}$ defined in \eqref{eq:def_mu_finite} and any element of $q_{\psi,\gd}\in M$. Namely, define the process $ t \mapsto \nu_{N, t}$, $t\geq0$ by
\begin{equation}
\label{eq:nut}
\nu_{ N, t}^{ i}\, :=\,  \mu_{ N, t}^{ i} - q_{ \psi,\gd}^{ i},\ i=-d,\ldots, d.
\end{equation}
The point is to write a mild formulation of this semi-martingale decomposition that makes sense in the space $H_{ d}^{ -1}$ (recall that $\mu_{N,t}$ and $\nu_{N,t}$ belong to $H^{-1}_d$ due to \eqref{eq:proba in Hminus1}). This mild formulation involves in particular the semi-group $e^{tL_{ \psi, \delta}}$ of the operator $L_{  \psi, \delta}$ \eqref{eq:def_Lpsiq} so that one can take advantage of the contraction properties of this semi-group in the neighborhood of the manifold $M$. 
\begin{proposition}
\label{prop:SPDE_mild}
For all $K>1$, for all $ 0\leq\delta\leq \delta(K)$, the process $ (\nu_{ N, t})_{ t\geq0}$ defined by \eqref{eq:nut} satisfies the following stochastic partial differential equation in $\cC([0, +\infty), H_{ d}^{ -1})$, written in a mild form:
\begin{equation}
\label{eq:SPDE_mild}
\nu_{ N, t}\, =\,  e^{ tL_{ \psi, \delta}}\nu_{ N, 0} +\int_0^t  e^{ (t-s)L_{ \psi, \delta}} \left( D_{ N} - \partial_{ \theta}R_{ N}( \nu_{ N, s})\right) \dd s + Z_{ N, t},\ N\geq1,\ t\geq0\, ,
\end{equation}
where 
\begin{equation}
\label{eq:D}
D_{ N}\, =\,  D_{ N, \psi, \delta}:= -\partial_{ \theta} \left(q_{ \psi, \delta}\sum_{k=-d}^d (\lambda_{ N}^k-\lambda^k) (J*q^k_{ \psi, \delta})\right) \, ,
\end{equation}
\begin{multline}
\label{eq:R}
R_{ N}( \nu_{ N, s})\, =\, R_{ N, \psi, \delta}(\nu_{ N, s})\, :=\,  \left( \sum_{k=-d}^d (\lambda_{ N}^k-\lambda^k)J*q^k_{ \psi, \delta}\right) \nu_{ N, s} + q_{ \psi, \delta}\sum_{k=-d}^d (\lambda_{ N}^k-\lambda^k)(J*\nu_{ N, s}^{ k}) \\ + \left(\sum_{k=-d}^d \lambda_{ N}^kJ*\nu_{ N, s}^k\right)\nu_{ N, s}\, ,
\end{multline}
and $Z_{ N, t}$ is the limit in $H_{ d}^{ -1}$ as $t^{ \prime}\nearrow t$ of $Z_{ N, t, t^{ \prime}}$ defined by
\begin{equation}
\label{eq:SPDE_Ztth}
Z_{ N, t, t^{ \prime}}(h)\, =\, \sum_{i=-d}^d \frac{\lambda^i}{N^i}\sum_{j=1}^{N^i} \int_{0}^{t^{ \prime}}\partial_\theta\left[\left(e^{(t-s)L_\psi^*}h\right)^i\right](\varphi^i_j(s))\dd B^i_j(s)\, ,
\end{equation}
that we denote
\begin{equation}
Z_{ N, t}(h)\, =\, \sum_{i=-d}^d \frac{\lambda^i}{N^i}\sum_{j=1}^{N^i} \int_{0}^{t}\partial_\theta\left[\left(e^{(t-s)L_\psi^*}h\right)^i\right](\varphi^i_j(s))\dd B^i_j(s)\, ,
\end{equation}
and where all the terms in \eqref{eq:SPDE_mild} make sense as elements of $\cC([0, \infty), H^{ -1}_{ d})$.
\end{proposition}
The proof of Proposition~\ref{prop:SPDE_mild} may be found in Section~\ref{sec:mild_formulation}.
The term $Z_{N,t}$ in \eqref{eq:SPDE_mild} represents the effect of the thermal noise on the system. The term involving $D_N$ is the one
that produces the drift we are after on the time scale $N^{1/2}t$, when the empirical measure $\mu_{N,t}$ is close to the manifold $M$.
To make this drift appear, we rely on an iterative procedure, as explained in Section \ref{subsec:dyn M}.

\subsection{Moving closer to the manifold $M$}
\label{sec:closer_to_manifold_intro}
We place ourselves in the framework of Theorem~\ref{th:main}: we fix $ \varepsilon_{ 0}>0$ and suppose the existence of a probability measure $p_{ 0}\in H_{ d}^{ -1}$ such that ${\rm dist}_{ H_{ d}^{ -1}}( p_{ 0}, M)\leq \varepsilon_{ 0}$ with $ \bP \left( \left\Vert \mu_{ N, 0} - p_{ 0}\right\Vert_{ -1, d}\geq \varepsilon\right) \rightarrow 0$ as $N\to \infty$, for all $ \varepsilon>0$. The constant $ \varepsilon_{ 0}$ will be chosen small enough in Section~\ref{sec:approach}.

The first step in proving our result is to show that the empirical measure $ \mu_{ N, t}$ reaches a neighborhood of size $N^{-1/2}$
in a time of order $\log N$. We use the projection defined in the following lemma, whose proof can be found in Appendix \ref{sec:appendix projections}, along with several regularity results.
\begin{lemma}\label{lem:existence projM}
There exists $ \sigma>0$ such that for all $h$ such that ${\rm dist}_{ H_{ d}^{ -1}}(h, M)\leq \sigma$, there exists a unique phase $\psi=:\proj_M(h)\in \bbT $ such that
$P^{ 0}_\psi(h-q_\psi)=0$ and the mapping $h \mapsto \proj_{ M}(h)$ is $\cC^{ \infty}$.
\end{lemma}
From now on, we fix a sufficiently small constant $\zeta$, more precisely satisfying
\begin{equation}\label{ineq:choice zeta}
 \zeta\, <\, \frac 18\, .
\end{equation}
We prove the following result: 
\begin{proposition}\label{prop:approach_M}
Under the above hypotheses, there exists a phase $\theta_0\in \bbT$, an event $B^N$ such that $\bP(B^N)\rightarrow 1$ and a constant $c>0$
such that for all $\gep>0$, for $N$ sufficient large, on the event $B^{ N}$, the projection $\psi_0= \psi_{ 0}^{ N}=\proj_M\big(\mu_{N,N^{1/2} t_0^N}\big)$ is well-defined and 
 \begin{equation}
  \left\Vert \mu_{N,N^{1/2} t_0^N}-q_{\psi_0} \right\Vert_{-1,d} \, \leq \, N^{-1/2+2\zeta} \, ,
 \end{equation}
and
 \begin{equation}
  |\psi_0-\theta_0|\, \leq\, \gep\, ,
 \end{equation}
 where $t_0^N=c N^{-1/2}\log N$.
\end{proposition}
We refer to Section~\ref{sec:approach} for a proof of this result. Since it relies on a discretization scheme similar to the one 
we introduce in the next paragraph, we leave the details to Section~\ref{sec:approach}.

\subsection{Dynamics on the manifold $M$}\label{subsec:dyn M}

We now place ourselves on the event $B^N$ (see Proposition \ref{prop:approach_M}), so that on the time
$N^{1/2}t_0^N$ we have $\Vert \mu_{N,N^{1/2} t_0^N}-q_{\psi_0}\Vert_{-1,d} \leq  N^{-1/2+2\zeta}$
where $\psi_0= \proj_M\big(\mu_{N,N^{1/2} t_0^N}\big)$. The point is to analyse the dynamics of \eqref{eq:SPDE_mild} on a time scale of order $ N^{ 1/2}$, using the knowledge we have on stability of the manifold $M$ (recall \eqref{eq:def_M}). The following iterative scheme we introduce is similar to ones used in \cite{Bertini:2013aa,cf:BDMP}.
\subsubsection{The iterative scheme}  We divide the evolution of the dynamics \eqref{eq:SPDE_mild} in time intervals $[T_n, T_{n+1}]$ with $T_n=N^{1/2}t_0^N+nT$ where $T$ is a constant independent of $N$, satisfying $T\geq 1$ and
\begin{equation}
\label{eq:hyp_bound_T}
e^{-\gamma_L T}\, \leq\,  \frac{1}{4C_L C_P}\, ,
\end{equation}
where the constants $C_L$ and $C_P$ where introduced in Section \ref{sec:lin stab}.
The number of steps $n_f$ is chosen as
\begin{equation}
 n_f\, :=\, \left\lfloor \frac{N^{1/2}}{T}(t_f-t^N_0)\right\rfloor\, .
\end{equation}
The intuition of this discretization is the following: if for a certain $n=0,1,\ldots, n_{ f}-1$, the process $\mu_{T_n}= \mu_{ N, T_{ n}}$ is close enough to the manifold $M$, we can define the phase $\ga_n$ of its projection on $M$ by:
\begin{equation}
 \ga_{n}\,: =\, \proj_M(\mu_{N, T_{n}})\, .
\end{equation}
This projection is in particular well defined when $\Vert \mu_{T_n}-q_{\ga_{n-1}} \Vert_{-1,d} \leq \gs$, where the constant $\gs>0$ is given by Lemma~\ref{lem:existence projM}.

To ensure that the process does not escape too far from $M$, we introduce the following stopping couple (where the infimum corresponds to the lexicographic order):
\begin{equation}
\label{eq:stopping_couple}
 (n_\tau,\tau)\, =\, \inf\{(n,t)\in \{1,\ldots,n_f\}\times [0,T] :\, 
 \Vert \mu_{T_{n-1}+t}-q_{\ga_{n-1}}\Vert_{-1,d}\geq \gs\}\, .
\end{equation}
Using \eqref{eq:stopping_couple}, we can define the following sequence of stopping times $(\tau^n, n=1\, \ldots\, n_f)$:
\begin{equation}\label{eq:tau_n}
  \tau^n\, :=\, 
  \left\{
  \begin{array}{ll}
   T & \text{if}\quad n<n_\tau\, ,  \\
   \tau & \text{if} \quad n\geq n_\tau\, ,
  \end{array}
  \right.
\end{equation}
and consider the stopped process $\mu_{(n\wedge n_\tau-1)T+t\wedge \tau^n}$. The projection of this stopped process is well defined on the whole interval $[T_0,T_{n_f}]$, so that we can now define rigorously the random phases  $\psi_{n-1}$ defined as
\begin{equation}
\label{eq:psi_n}
 \psi_{n-1}\, :=\, \proj_M(\mu_{(n\wedge n_\tau-1)T+t\wedge \tau^n}),\ n=1,\ldots,n_f\, .
\end{equation}
$ \psi_{ n-1}$ corresponds to the phase of the projection of the process $\mu$ unless it has been stopped and in that case, it is the phase at the stopping time. The object of interest here is the process $\nu_{n,t}$ defined for $n=1,\ldots,n_f$ as
\begin{equation}
\label{eq:nu_stopped}
 \nu_{n,t}\, :=\, \mu_{(n\wedge n_\tau-1)T+t\wedge \tau^n}-q_{\psi_{n-1}}\, .
\end{equation}
Using \eqref{eq:SPDE_mild}, we see that this process satisfies the mild equation
\begin{equation}\label{eq:mild_nu_n}
 \nu_{n,t}\, =\, e^{(t\wedge \tau^n) L_{{\psi_{n-1}}}}\nu_{n,0}-\int_0^{t\wedge\tau^n}
 e^{(t\wedge\tau^n-s)L_{{\psi_{n-1}}}}(D_{\psi_{n-1}}
 +R_{\psi_{n-1}}(\nu_{n,s}))\dd s + Z_{n,t\wedge \tau^n}\, ,
\end{equation}
where $D_{ \psi_{ n-1}}:= D_{ N, \psi_{ n-1}, \delta}$ (recall \eqref{eq:D}), $R_{ \psi_{ n-1}}(\nu_{ n}):= R_{ N, \psi_{ n-1}, \delta}(\nu_{ n})$ (recall \eqref{eq:R}) and $Z_{n,t}$ is defined as
\begin{equation}
\label{eq:noise_Znt}
 Z_{n,t}(h)\, =\, \sum_{i=-d}^d\frac{\gl^i}{N^i}\sum_{j=1}^{N^i} \int_{0}^{t}  
 \partial_\theta\left[ \left(e^{(t-s)L_{\psi_{n-1}}^*}h\right)^i\right](\varphi^i_j(T_{n-1}+s))\dd B^i_j(T_{n-1}+s)\, .
\end{equation}
Note that we drop here the dependence in $N$ and $ \delta$ for simplicity.
\subsubsection{Controlling the noise and a priori bound on the fluctuation process}
A key point in the analysis of \eqref{eq:mild_nu_n} is to show that one can control the behavior of the noise part $Z_{ n, t}$ in \eqref{eq:mild_nu_n} along the discretization introduced in the last paragraph. More precisely, for $\zeta$ chosen according to \eqref{ineq:choice zeta} and some positive constant $C_Z$ and defining the event
\begin{equation}
\label{eq:AN}
A^{ N}= A^{ N}(C_{ Z})\, :=\,  \left\lbrace \sup_{1\leq n\leq n_f}\sup_{0\leq t\leq T} \Vert Z_{n,t} \Vert_{-1, d} \leq C_Z\sqrt{\frac{T}{N}}N^\gz\right\rbrace\, ,
\end{equation}
the purpose of Section~\ref{sec:control_noise} is precisely to prove that $\bP(A^{ N})$ tends to $1$ as $N\to\infty$. With the knowledge of \eqref{eq:AN}, one can prove that the process $ \nu_{ n}$ remains a priori bounded: using that the sequence of the disorder $( \omega_{ i})_{ i\geq1}$ is admissible (recall Definition~\ref{def:disorder}), we prove in Proposition~\ref{prop:bound_nu}, Section~\ref{sec:dynamics_manifold}, that on the event $A^{ N}\cap B^N$, 
\begin{equation}
\label{eq:bound_nunt_intro}
\sup_{ 1\leq n\leq n_{ f}} \sup_{ t\in[0, T]} \left\Vert \nu_{ n, t} \right\Vert_{ -1, d}\,  =\,  O (N^{ -1/2 + 2 \zeta})\, ,
\end{equation}
as $N\to\infty$.
\subsubsection{Expansion of the dynamics on the manifold $M$}
The last step of the proof consists in looking at the rescaled dynamics of the phase of the projection of the empirical measure on $M$, that is the process
\begin{equation}
 \varPsi^N_t\, :=\,\psi_{n_t}\, ,
\end{equation}
where $(\psi_{ n})_{0\leq n\leq n_{ f}}$ is given by \eqref{eq:psi_n} and
\begin{equation}
 n_t\, :=\, \left\lfloor \frac{N^{1/2}}{T}(t-t_0^N)\right\rfloor \, .
\end{equation}
Namely, we prove in Propositions~\ref{prop:expansion_psi} and \ref{prop:muN_qbxiN} that, with high probability as $N\to\infty$, the following expansion holds:
 \begin{equation}\label{eq:def PsiNt}
 \varPsi^N_t\, =\,\psi_0+  b(\xi_N) t+O(N^{-1/4+2\zeta})\, ,
\end{equation}
where $b$ is the linear form of Theorem~\ref{th:main} and that $\mu_{N,N^{1/2}t}$ is close to $q_{\Psi^N_t}$ with high probability.
\subsection{Organization of the rest of the paper}
Section~\ref{sec:mild_formulation} is devoted to prove the mild formulation described in Paragraph~\ref{sec:mild_formulation_intro}. The control of the noise term in \eqref{eq:noise_Znt} is addressed in Section~\ref{sec:control_noise}. The dynamics on the manifold $M$ and the approach to the manifold are studied in Sections~\ref{sec:dynamics_manifold} and~\ref{sec:approach} respectively. The asymptotics of the drift as $ \delta\to 0$ is studied in Section~\ref{sec:asymptotic_drift}. We compile in the appendix several spectral estimates and expansions in small $\gd$ used throughout the paper.
\section{Proof of the mild formulation}
\label{sec:mild_formulation}
Define $\bL_{ 0, d}^{ 2}$ as the closure of the space of regular test functions $(u^{-d},\ldots,u^d)$ such that $ \int_{ \bbT} u^{ k}=0$ for all $k=-d, \ldots, d$ under the norm 
\begin{equation}
\label{eq:norm_L0_intro}
\left\Vert u \right\Vert_{ 0, d}\, :=\,   \left(\sum_{ k=-d}^{ d} \lambda^{ k}\int_{ \bbT}  u^{ k}( \theta)^{ 2} \dd \theta\right)^{ 1/2}\, ,
\end{equation}
and the space $H^\ga_d$ ($\ga\geq 0$) closure of the same set of test functions under the norm (denoting $\Vert \cdot\Vert_0$ the $L^2$-norm on $\bbT$)  
\begin{equation}
\left\Vert u \right\Vert_{ \ga, d}\, :=\,   \left\Vert (-\Delta_d)^{\ga/2} u\right\Vert_{0,d}=
\left(\sum_{ k=-d}^{ d} \lambda^{ k}\left\Vert (-\Delta)^{\ga/2} u^k\right\Vert_0^2 \right)^{ 1/2}\, ,
\end{equation}
where $\Delta_d$ denotes the Laplacian on $\bbT^{ 2d+1}$. We denote by $H^{-\ga}_d$ the dual space of $ H^{ \alpha}_{ d}$. We also write, for any bounded signed measure $m$ on $\bbT$, the usual distribution bracket as \[ \left\langle m\, ,\, f\right\rangle:= \int_{ \bbT} f(\theta) m(\dd \theta),\] and for any vector $(m^{ 1}, \ldots, m^{ d})$ of such measures \[ \left\langle m\, ,\, F\right\rangle_{ d}:= \sum_{ i=-d}^{ d} \lambda^{ i} \left\langle m^{ i}\, ,\, F^{ i}\right\rangle= \sum_{ i=-d}^{ d} \lambda^{ i} \int_{ \bbT} F^{ i}(\theta) m^{ i}(\dd \theta),\]the corresponding bracket weighted w.r.t. the disorder. Obviously, when the above measure coincide with an $L^{ 2}$ function, this expression coincides with the $L^{ 2}$ scalar product $ \left\langle \cdot\, ,\, \cdot\right\rangle_{ 2, d}$ associated to \eqref{eq:norm_L0_intro}. 
 
This section is devoted to prove Proposition~\ref{prop:SPDE_mild}. We begin first with a weak formulation of the SPDE \eqref{eq:SPDE_mild}.
\begin{proposition}
\label{prop:SPDE_weak}
For all $K>1$, for all $0\leq \delta\leq \delta(K)$, for any $ (t, \theta) \mapsto F_{ t}(\theta)= (F_{ t}^{ -d}(\theta), \ldots, F_{ t}^{ d}(\theta))\in \cC^{ 1, 2}([0, +\infty)\times \bbT, \bbR)$ such that $\int_{ \bbT} F_{ t}(\theta)\dd \theta=0$,
\begin{align}
\langle \nu_{ N, t}, F_t \rangle_{ d}\, &=\, \langle \nu_{ N, 0}, F_0\rangle_{ d} +\int_0^t \left\langle \nu_{ N, s}\, ,\, \partial_s F_s + L_{ \psi, \delta}^{ \ast} F_{ s}\right\rangle_{ d}\dd s +\int_0^t  \left\langle D_{ N}\, ,\, F_{ s}\right\rangle_{ d} \dd s \nonumber\\ &+ \int_{0}^{t} \left\langle R_{ N}( \nu_{ N, s})\, ,\, \partial_{ \theta} F_{ s}\right\rangle_{ d}\dd s + M_{ N, t}^{ F},\ N\geq1,\ t\geq0\, ,\label{eq:SPDE_weak}
\end{align}
where $D_{ N}$, $R_{ N}( \nu_{ N})$ are respectively defined in \eqref{eq:D} and \eqref{eq:R} and
\begin{equation}
\label{eq:ZNF}
M_{ N, t}^{ F}\, :=\, \sum_{i=-d}^d \frac{\lambda^i}{N^i}\sum_{j=1}^{N^i}\int_0^t \partial_\theta F^i_s(\varphi^i_j(s))\dd B^i_j(s).
\end{equation}
In \eqref{eq:SPDE_weak}, the operator $L_{ \psi, \delta}^{ \ast}$ is the dual in $\bL_{ 0, d}^{ 2}$ of the operator $L_{ \psi, \delta}$ in \eqref{eq:def_Lpsiq}:
\begin{multline}
\label{eq:L_delta_psi0_ast}
(L^*_{ \psi, \delta} v)^i \, =\, \frac12 \partial_{ \theta}^{ 2}v^i +\gd\go^i \partial_{ \theta}v^i+(\partial_{ \theta}v^i)\sum_{k=-d}^d \lambda^k J*q^k_{ \psi, \delta} - \int_{ \bbT}  \left((\partial_{ \theta} v^{ i}) \sum_{ k=-d}^{ d} \lambda^{ k}J\ast q_{ \psi, \delta}^{ k}\right)\dd \theta\\-\sum_{k=-d}^d \lambda^k J*(q^k_{ \psi, \delta} \partial_{ \theta}v^k)\, .
\end{multline}
\end{proposition}
We refer to Appendix~\ref{sec:appendix_spectral_semigroups} (see in particular Propositions~\ref{prop:def_adjoint} and ~\ref{prop:L_ast_sectorial}) for a detailed analysis of the spectral properties of the operator $L_{ \psi, \delta}$ and its dual $L_{ \psi, \delta}^{ \ast}$. All we need to retain here is that when $ \delta$ is small, the operator $L_{ \psi, \delta}$ is sectorial in $H^{ -1}_{d}$ and generates a $\cC_{ 0}$-semi-group $ t \mapsto e^{ tL_{ \psi, \delta}}$ in this space. Moreover, on the space $\bL_{ 0, d}^{ 2}$, one has that $(e^{ tL_{ \psi, \delta}})^{ \ast}= e^{ tL_{ \psi, \delta}^{ \ast}}$. Since the phase $ \psi$ is not relevant in this paragraph, we write for simplicity $ q_{ \delta}$,  $L_{ \delta}$ instead of $q_{ \psi, \delta}$ and $L_{ \psi, \delta}$.
\begin{proof}[Proof of Proposition~\ref{prop:SPDE_weak}]
Note that, using the definition of $J(\cdot)$ in \eqref{eq:def_J} and of the empirical measure $\mu_{ N, t}$ in \eqref{eq:def_mu_finite}, the system \eqref{eq:eds_Kur} may be rewritten as
\begin{equation}
 \dd \varphi^i_j(t)\, =\, \gd\go^i\dd t +\sum_{k=-d}^d \lambda_{ N}^k J*\mu^k_t\big(\varphi^i_j(t)\big)\dd t+\dd B^i_j(t),\ i=-d, \ldots, d\, .
\end{equation}
Consider $(t, \theta) \mapsto F_{ t}(\theta)=(F^{ i}_{ t}( \theta))_{i=1,\ldots,d}\in C^{1,2}([0, +\infty)\times \bbT,\bbR)^{ d}$ such that for all $t\geq0$, $\int_{ \bbT} F_{ t}(\theta)\dd \theta=0$. An application of It\^o Formula to \eqref{eq:eds_Kur} gives, for $i=1, \ldots, d$, $j= 1, \ldots, N^{ i}$, $ t\geq0$,
\begin{align*}
F^i_t(\varphi^i_j(t))\, &=\, F^i_0(\varphi^i_j(0))+\int_0^t \partial_s F^i_s(\varphi^i_j(s))\dd s +\frac12 \int_0^t \partial^2_\theta F^i_s(\varphi^i_j(s))\dd s \\&+\int_0^t \partial_\theta F^i_s(\varphi^i_j(s))\left(\gd \go^i+\sum_{k=-d}^d \lambda_{ N}^k J*\mu_{ N, s}^{ k}(\varphi^i_j(s))\right)\dd s +\int_0^t \partial_\theta F^i_s(\varphi^i_j(s))\dd B^i_j(s)\, .
\end{align*}
After summation over $j=1,\ldots,N^i$, we obtain, for $i=1, \ldots, d$,
\begin{multline}
\label{eq:mui_F}
 \langle \mu^i_{ N, t}, F^i_t\rangle\, =\, \langle \mu^i_{ N, 0}, F^i_0\rangle +\int_0^t \left\langle \mu^i_{ N, s},\partial_s F^i_s + \frac12 \partial^2_\theta F^i_s +\partial_\theta F^i_s\left(\gd \go^i +\sum_{k=-d}^d \lambda_{ N}^k(J*\mu^k_{ N, s})\right)\right\rangle\dd s\\
+\frac{1}{N^i}\sum_{j=1}^{N^i}\int_0^t \partial_\theta F^i_s(\varphi^i_j(s))\dd B^i_j(s)\, . 
\end{multline}
Replacing $\mu_{ N, t}^i$ by $\nu_{ N, t}^{ i} + q^i_{ \delta}$ in \eqref{eq:mui_F} (recall \eqref{eq:nut}), we obtain
\begin{multline}\label{eq:expand_mui}
\langle \nu_{ N, t}^i, F^i_t\rangle + \langle q^i_{ \delta},F^i_t\rangle\, =\, \langle \nu_{ N, 0}^i, F^i_0\rangle + \langle q^i_{ \delta},F^i_0\rangle\\
+\int_0^t \left\langle \nu_{ N, s}^i + q^i_\delta,\partial_s F^i_s + \frac12 \partial^2_\theta F^i_s +\partial_\theta F^i_s\left(\gd \go^i +\sum_{k=-d}^d  \lambda_{ N}^kJ*(\nu_{ N, s}^k + q^k_\psi)\right)\right\rangle\dd s\\ 
+\frac{1}{N^i}\sum_{j=1}^{N^i}\int_0^t \partial_\theta F^i_s(\varphi^i_j(s))\dd B^i_j(s)\\
=\, \langle \nu_{ N, 0}^i, F^i_0\rangle +\int_0^t \left\langle \nu_{ N, s}^i, \partial_s F^i_s + \frac12 \partial^2_\theta F^i_s +\partial_\theta F^i_s\left(\gd \go^i +\sum_{k=-d}^d \lambda_{ N}^kJ*q^k_\delta\right)\right\rangle\dd s\\
+\int_0^t \left\langle \nu_{ N, s}^i, \partial_\theta F^i_s\sum_{k=-d}^d \lambda_{ N}^k(J*\nu_{ N, s}^k)\right\rangle\dd s +\int_0^t \left\langle q^i_\delta, \partial_\theta F^i_s\sum_{k=-d}^d \lambda_{ N}^k(J*\nu_{ N, s}^k)\right\rangle\dd s\\ + \langle q^i_\delta,F^i_0\rangle +\int_0^t \left\langle q^i_\delta, \partial_s F^i_s + \frac12 \partial^2_\theta F^i_s +\partial_\theta F^i_s\left(\gd \go^i +\sum_{k=-d}^d \lambda_{ N}^kJ*q^k_\delta\right)\right\rangle\dd s\\
+\frac{1}{N^i}\sum_{j=1}^{N^i}\int_0^t \partial_\theta F^i_s(\varphi^i_j(s))\dd B^i_j(s)\, .
\end{multline}
Since by definition $q_{\delta}$ is a stationary solution to \eqref{eq:FP_Kur_finite}, one easily sees that
\begin{equation}\label{eq:mild_qi}
 \langle q^i_\delta,F^i_t\rangle\, =\,  \langle q^i_\delta,F^i_0\rangle +\int_0^t \langle q^i_\delta, \partial_{ s}F^i_s\rangle \dd s,\quad i=1, \ldots, d,\ t\geq0\, ,
\end{equation}
and
\begin{align}
 0\, &=\, \left\langle \frac12 \partial_{ \theta}^{ 2}q^i_\delta-\gd \go^i \partial_{ \theta}q^i_\delta-\partial_{ \theta}\left(q^i_{ \delta} \sum_{k=-d}^d \lambda^k J*q^k_\delta \right), F^i_s\right\rangle \nonumber\\
 &=\, \left\langle q^i_\delta,\frac12 \partial^2_\theta F^i_s+\gd \go^i \partial_\theta F^i_s+\partial_\theta F^i_s \sum_{k=-d}^d \lambda^k J*q^k_\delta \right\rangle \, .\label{eq:using_qi_stat}
\end{align}
Summing \eqref{eq:expand_mui} over $i=-d,\ldots,d$ and using \eqref{eq:mild_qi} and \eqref{eq:using_qi_stat}, we obtain
\begin{multline}\label{eq:first_attempt_mild}
 \langle \nu_{ N, t}, F_t \rangle_{ d}\, =\, \langle \nu_{ N, 0}, F_0\rangle_{ d} +\int_0^t \left\langle \nu_{ N, s}, \partial_s F_s + \frac12 \partial^2_\theta F_s +\gd \partial_\theta F_s\otimes w+\partial_\theta F_s\sum_{k=-d}^d \lambda_{ N}^kJ*q^k_\delta\right\rangle_{ d}\dd s\\
+\int_0^t \left\langle \nu_{ N, s}, \partial_\theta F_s\sum_{k=-d}^d \lambda_{ N}^kJ*\nu_{ N, s}^k\right\rangle_{ d}\dd s +\int_0^t \left\langle q_\delta, \partial_\theta F_s\sum_{k=-d}^d \lambda_{ N}^kJ*\nu_{ N, s}^k\right\rangle_{ d}\dd s\\
 +\int_0^t \left \langle q_\delta \sum_{k=-d}^d ( \lambda_{ N}^k-\lambda^k)J*q^k_\delta, \partial_\theta F_s \right\rangle_{ d}\dd s+M_{ N, t}^{ F}\, ,
\end{multline}
where $M_{ N, t}^{ F}$ is defined in \eqref{eq:ZNF} and where we have used the notation $F\otimes \go =(F^i \go^i)_{i_1,\ldots,d}$. Note that
\begin{align}
\left\langle q_\delta, \partial_\theta F_s\sum_{k=1}^d \lambda^k J*\nu_{ N, s}^k \right\rangle_{ d} & =\, \sum_{i=1}^d\sum_{k=-d}^d \lambda^i \lambda^k \langle q^i_\delta, \partial_\theta F^i_s J*\nu_{ N, s}^k\rangle \nonumber\\
&=\, -\sum_{i=-d}^d\sum_{k=-d}^d \lambda^i \lambda^k \langle \nu_{ N, s}^i, J*(q^i_\delta  \partial_\theta F^i_s) \rangle \nonumber\\
&=\, -\left\langle \nu_s,\sum_{ k=-d}^{ d} \lambda^k J*(q_{ \delta}^k\partial_\theta F^k_s)\right\rangle_{ d}\, .\label{aux:SPDE_weak1}
\end{align}
The result of Proposition~\ref{prop:SPDE_weak} is a simple reformulation of \eqref{eq:first_attempt_mild} using \eqref{aux:SPDE_weak1} and the definition of $L_{ \delta}^{ \ast}$ in \eqref{eq:L_delta_psi0_ast}.
\end{proof}
We are now in position to prove Proposition~\ref{prop:SPDE_mild}:
\begin{proof}[Proof of Proposition~\ref{prop:SPDE_mild}]
Let us apply the identity \eqref{eq:SPDE_weak} of Proposition~\ref{prop:SPDE_weak} in the case of test functions $F_{ s}$ of the form
\[F_{ s}= e^{ (t-s)L_{ \delta}^{ \ast}}h,\] for any test functions $h$ of class $\cC^{ 2}$ on $\bbT$. Then $ \partial_{ s} F_{ s}= - L_{ \delta}^{ \ast} F_{ s}$ and one obtains
\begin{multline}
\langle \nu_{ N, t}, h\rangle_{ d}\, =\, \langle \nu_{ N, 0}, e^{ tL_{ \delta}^{ \ast}}h\rangle_{ d} +\int_0^t  \left\langle D_{ N}\, ,\, e^{ (t-s)L_{ \delta}^{ \ast}}h\right\rangle_{ d} \dd s + \int_{0}^{t} \left\langle R_{ N}( \nu_{ N, s})\, ,\, \partial_{ \theta} e^{ (t-s)L_{ \delta}^{ \ast}}h \right\rangle_{ d}\dd s\\+M_{ N, t}^{ F}\, .\label{eq:SPDE_weak_F}
\end{multline}
We aim at proving that one can write a mild version of this weak equation and that this mild formulation makes sense in $H^{ -1}_{ d}$. Consider a sequence $(v_{ l})_{ l\geq1}$ of elements of $\bL_{ 0, d}^{ 2}$ converging as $l\to \infty$ in $H^{ -1}_{ d}$ to $ \nu_{ N, 0}\in H^{ -1}_{ d}$. Then, for $h$ of class $\cC^{ 2}$,
\begin{equation}
\left\langle v_{ l}\, ,\, e^{ tL_{ \delta}^{ \ast}} h\right\rangle_{ d} \, =\,  \left\langle v_{ l}\, ,\, e^{ tL_{ \delta}^{ \ast}} h\right\rangle_{ 2, d} \, =\,  \left\langle  e^{ tL_{ \delta}} v_{ l}\, ,\, h\right\rangle_{ 2, d}\, =\,  \left\langle e^{ tL_{ \delta}} v_{ l}\, ,\, h\right\rangle_{ d}.
\end{equation}
By continuity of $e^{ tL_{ \delta}}$ on $H^{ -1}_{ d}$, $e^{ tL_{ \delta}} v_{ l}$ converges in $H^{ -1}_{ d}$ to $e^{ tL_{ \delta}} \nu_{ N, 0}$, as $l\to\infty$. In particular, for all $h\in H^{ 1}_{ d}$,
\begin{equation}
 \left\vert \left\langle e^{ tL_{ \delta}} \nu_{ N, 0}\, ,\, h\right\rangle_{ d} - \left\langle e^{ tL_{ \delta}} v_{ l}\, ,\, h\right\rangle_{ d}\right\vert \, \leq\,  \left\Vert h \right\Vert_{ 1, d} \left\Vert e^{ tL_{ \delta}} \nu_{ N, 0}- e^{ tL_{ \delta}} v_{ l} \right\Vert_{ -1, d}\to_{ l\to\infty} 0,
\end{equation}
so that, at the limit for $l\to\infty$, for all $t\geq0$,
\begin{equation}
\left\langle \nu_{ N, 0}\, ,\, e^{ tL_{ \delta}^{ \ast}} h\right\rangle_{ d}\, =\,  \left\langle e^{ tL_{ \delta}} \nu_{ N, 0}\, ,\, h\right\rangle_{ d}.
\end{equation}
Since the function $D_{ N}$ defined in \eqref{eq:D} is regular, it is straightforward to prove in the same way that
\begin{equation}
\left\langle D_{ N}\, ,\, e^{ (t-s)L_{ \delta}^{ \ast}} h\right\rangle_{ d}\, =\,  \left\langle e^{ (t-s) L_{ \delta}} D_{ N}\, ,\, h\right\rangle_{ d}.
\end{equation}
The continuity of the mapping $t \mapsto e^{ tL_{ \delta}} \nu_{ N, 0}$ and $t \mapsto \int_{0}^{t} e^{ (t-s)L_{ \delta}}D_{ N} \dd s$ in $H_{ d}^{ -1}$ is immediate from the continuity of the semigroup in $H_{ d}^{ -1}$.

We now focus on the term $R_{ N}( \nu_{ N, s})$ defined in \eqref{eq:R}: consider $(w_{s, l})_{ l\geq1}$ a sequence of elements of $\bL_{ 0, d}^{ 2}$ converging in $H^{ -1}_{ d}$ to $ \nu_{ N, s}$ (consider for example $w_{s, l}= \phi_{ l}\ast \nu_{ N, s}$ for a regular approximation of identity $( \phi_{ l})_{ l\geq1}$) and define
\begin{multline}
R_{ s, l}\, :=\,  \left( \sum_{k=-d}^d (\lambda_{ N}^k-\lambda^k)J*q^k_{ \delta}\right) w_{ s, l} + q_{ \delta}\sum_{k=-d}^d (\lambda_{ N}^k-\lambda^k)(J*\nu_{ N, s}^k) \\ + \left(\sum_{k=-d}^d \lambda_{ N}^kJ*\nu_{ N, s}^k\right) w_{ s, l}\, .
\end{multline}
For any $l\geq1$, the following identity holds:
\begin{align}
\label{aux:Rsp}
\left\langle R_{s, l}\, ,\, \partial_{ \theta} e^{ (t-s)L_{ \delta}^{ \ast}}h\right\rangle_{ d} &= \left\langle R_{ s, l}\, ,\, \partial_{ \theta} e^{ (t-s)L_{ \delta}^{ \ast}} h\right\rangle_{ 2, d} = - \left\langle e^{ (t-s) L_{ \delta}}\partial_{ \theta} R_{ s, l}\, ,\, h\right\rangle_{ 2, d}.
\end{align}
Since $h$ is regular and $R_{ s, l}$ converges in $H^{ -1}_{ d}$ to $R_{ N}( \nu_{ N, s})$, the lefthand part of the previous identity converges as $l\to\infty$ to $ \left\langle R_{ N}(\nu_{ N, s})\, ,\, \partial_{ \theta} e^{ (t-s)L_{ \delta}^{ \ast}} h\right\rangle_{ d}$. Moreover, for all $h$ regular, using the estimate \eqref{eq:reg_semigroup_m1} on the regularity of the semigroup $ e^{ tL_{ \delta}}$, (note in particular that $e^{ tL_{ \delta}}$ can be extended to a continuous operator to $H^{ -2}_{ d}$ to $H^{ -1}_{ d}$, see Proposition~\ref{prop:reg_semigroup} below),
\begin{align*}
\left\vert \left\langle e^{ (t-s)L_{ \delta}}\partial_{ \theta}(R_{ s, l} - R_{ N}(\nu_{ N, s}))\, ,\, h\right\rangle_{ d}\right\vert &\leq\,  \left\Vert h \right\Vert_{ 1, d} \left\Vert e^{ (t-s)L_{ \delta}}\partial_{ \theta}(R_{ s, l} - R_{ N}(\nu_{ N, s}))\right\Vert_{ -1, d},\\
&\leq \, C\left\Vert h \right\Vert_{ 1, d} \left(1+ \frac{ 1}{ \sqrt{t-s}}\right) \left\Vert \partial_{ \theta}(R_{ s, l} - R_{ N}(\nu_{ N, s})) \right\Vert_{ -2, d},\\
&\leq \, C\left\Vert h \right\Vert_{ 1, d} \left(1+ \frac{ 1}{ \sqrt{t-s}}\right) \left\Vert R_{ s, l} - R_{ N}(\nu_{ N, s}) \right\Vert_{ -1, d}.
\end{align*}
Since the last estimate is true for all $h$ regular, one obtains that
\begin{equation}
 \left\Vert e^{ (t-s)L_{ \delta}}\partial_{ \theta}(R_{ s, l} - R_{ N}(\nu_{ N, s})) \right\Vert_{ -1, d}\,  \leq\,  C \left(1 + \frac{ 1}{ \sqrt{t-s}}\right) \left\Vert R_{ s, l} - R_{ N}( \nu_{ N, s}) \right\Vert_{ -1, d}.
\end{equation}
Since $R_{ s, l}$ converges to $R_{ N}(\nu_{ N, s})$ in $H^{ -1}_{ d}$, one obtains that one can make $l\to\infty$ in \eqref{aux:Rsp}:
\[\left\langle R_{ N}(\nu_{ N, s})\, ,\, \partial_{ \theta} e^{ (t-s)L_{ \delta}^{ \ast}}h\right\rangle_{ d} = - \left\langle e^{ (t-s)L_{ \delta}}\partial_{ \theta} R_{ N}( \nu_{ N, s})\, ,\, h\right\rangle_{ d}.\]
The same argument as before shows also that
\begin{multline}\label{aux:TRnus}
 \left\Vert e^{ (t-s)L_{ \delta}}\partial_{ \theta}R_{ N}(\nu_{ N, s}) \right\Vert_{ -1, d} \leq C \left(1 + \frac{ 1}{ \sqrt{t-s}}\right) \left\Vert R_{ N}( \nu_{ N, s}) \right\Vert_{ -1, d}\\ \, \leq\,  C \left(1 + \frac{ 1}{ \sqrt{t-s}}\right) \left\Vert \nu_{ N, s}\right\Vert_{ -1, d}\leq C 
 \pi \sqrt{\frac 23 \sum_{k=-d}^d (\gl^k)^{-1}}\left(1 + \frac{ 1}{ \sqrt{t-s}}\right),
\end{multline}
where we used \eqref{eq:proba in Hminus1}. The inequality \eqref{aux:TRnus} implies that the integral $ \int_{0}^{t} \left\Vert e^{ (t-s)L_{ \delta}}\partial_{ \theta}R_{ N}(\nu_{ N, s}) \right\Vert_{ -1, d}\dd s$ is almost surely finite. Using \cite{MR617913}, Theorem~1, p.~133, we deduce that $ \int_{0}^{t} e^{ (t-s)L_{ \delta}}\partial_{ \theta}R_{ N}(\nu_{ N, s}) \dd s$ makes sense as a Bochner integral in $H^{ -1}_{ d}$. The continuity of $ t \mapsto \int_{0}^{t} e^{ (t-s)L_{ \delta}}\partial_{ \theta}R_{ N}(\nu_{ N, s}) \dd s$ in $H^{ -1}_{ d}$ is a direct consequence of the bounds found in Proposition~\ref{prop:reg_semigroup}.

It remains to treat the noise term in \eqref{eq:SPDE_weak_F}. The precise control of this term is made in Section~\ref{sec:control_noise} below (see in particular Proposition~\ref{prop:estim_noise}). We prove actually more in Section~\ref{sec:control_noise} since we have to take into account the dependence in $N$, which is not important for this proof. Let us admit for the moment that the proof of Proposition~\ref{prop:estim_noise} is valid. In particular, one deduces from \eqref{eq:Znttss} and an application of the Kolmogorov Lemma that the almost-sure limit when $t^{ \prime} \nearrow t$ of $Z_{N, t, t^{ \prime}}$ defined in \eqref{eq:SPDE_Ztth} exists in $H_{ d}^{ -1}$. The continuity of the limiting process $t \mapsto Z_{ N, t}$ in $H_{ d}^{ -1}$ comes from \eqref{eq:Znts} and the rest of the proof of Proposition~\ref{prop:estim_noise}. It is immediate to see from \eqref{eq:ZNF} that for all $F$ regular, $ \left\langle Z_{ N, t}\, ,\, F\right\rangle_{ d}= M_{ N, t}^{ F}$, where we see the term $Z_{n,t}$ 
as a vector $(Z_{n,t}^{-d}, \ldots,Z_{n,t}^d)$, with $Z_{n,t}^k(h)=\frac{1}{\gl^k}Z_{n,t}(\hat h_k)$ and $\hat h_k=(0,\ldots,0, h^k,0,\ldots,0)$. One concludes from everything 
that we have done that, for all $h$ regular that
\begin{multline}
\left\langle \nu_{ N, t}\, ,\, h\right\rangle_{ d}\, =\,  \left\langle e^{ tL_{ \delta}} \nu_{ N, 0}\, ,\, h\right\rangle_d + \left\langle \int_{0}^{t}  \left(e^{ (t-s)L_{ \delta}} D_{ N} - e^{ (t-s)L_{ \delta}}\partial_{ \theta}R_{ N}(\nu_{ N, s})\right)\dd s\, ,\, h\right\rangle_{ d} \\+ \left\langle Z_{ N, t}\, ,\, h\right\rangle_{ d},
\end{multline}
where everything above makes sense as element of $H^{ -1}_{ d}$. Since this is true for all $h$ regular, the identity \eqref{eq:SPDE_mild} follows. Proposition~\ref{prop:SPDE_mild} is proved.
\end{proof}

\section{Controlling the noise}
\label{sec:control_noise}
This section is devoted to control the noise term $Z_{ n, t}$ defined in \eqref{eq:noise_Znt}. More precisely, we prove the following proposition
(recall the definition of $A^N=A^{ N}(C_{ Z})$ given in \eqref{eq:AN}). 
\begin{proposition}
\label{prop:estim_noise} For all $\zeta>0$, there exists a constant $C_Z$ such that $\bP(A^N)\rightarrow 1$, as $N\to\infty$.
\end{proposition}

To prove Proposition~\ref{prop:estim_noise}, we rely on the two following lemmas:
\begin{lemma}[Garsia-Rademich-Rumsey]
\label{th:GRR}
Let $ \chi$ and $\Psi$ be continuous, strictly increasing functions on $(0,\infty)$ such that $ \chi(0)=\Psi(0)=0$ and $\lim_{t\nearrow\infty} \Psi(t)=\infty$.
Given $T>0$ and $\phi$ continuous on $(0,T)$ and taking its values in a Banach space $(E,\Vert.\Vert)$, if
\begin{equation}
 \int_0^T\int_0^T \Psi\left(\frac{\Vert \phi(t)-\phi(s)\Vert}{ \chi(|t-s|)}\right)\dd s\dd t \, \leq \, B\, < \, \infty\, ,
\end{equation}
then for $0\leq s\leq t\leq T$,
\begin{equation}
 \Vert\phi(t)-\phi(s)\Vert\, \leq \, 8\int_0^{t-s} \Psi^{-1}\left(\frac{4B}{u^2}\right) \chi(\dd u)\, .
\end{equation}
\end{lemma}
Proof of Lemma~\ref{th:GRR} may be found in \cite{StrookVaradhan1979}, Theorem~2.1.3. The second result estimates the moments of the process $Z_{ n, t}$:
\begin{lemma}\label{lem:moments Z}
For all $\gep>0$ and all integer $m>0$, there exists a positive constant $C_{m,\gep}$ such that for all $0\leq s< t\leq T$, 
\begin{equation}
\label{eq:Znts}
\bE \Vert Z_{n,t}-Z_{n,s}\Vert^{2m}_{-1,d}\, \leq\,  \frac{C_{m,\gep}}{N^m}\left((t-s)^{m(1/2-2\gep)}+(t-s)^m\right)\, . 
\end{equation}
\end{lemma}
Let us first prove Proposition~\ref{prop:estim_noise}, relying on these two lemmas.
\begin{proof}[Proof of Proposition~\ref{prop:estim_noise}]
Using Lemma \ref{lem:moments Z}, we can apply Lemma~\ref{th:GRR} with the choices
\begin{equation}
\phi(t)\, =\, Z_{n,t}\, , \ \ \ \chi(u)\, =\, u^{\frac{2+\gz}{2m}}\ \text{  and } \ \Psi(u)\, =\, u^{2m}\, ,
\end{equation}
which implies that there exist a constant $C$ (depending in $m$, $\gep$ and $\zeta$) and a positive random variable $B$ such that for every $0\leq s<t \leq T$:
\begin{equation}
\Vert Z_{n,t}-Z_{n,s} \Vert_{-1,d}^{2m}\, \le \, C ( t-s ) ^\zeta B \, ,
\end{equation}
where $B$ satisfies
\begin{equation}
\bE (B)\, \le \, \frac C{N^m} \int_0^T \int_0^T \left(\left\vert t-s \right\vert^{m(1/2-2\gep)-2- \zeta}+ \left\vert t-s \right\vert^{m-2-\zeta}
\right)\dd s \dd t\, .
\end{equation}
A simple integration shows that $\bE(B)\leq \frac{ C}{ N^{ m}} \left(T^{ m(1/2-2 \varepsilon) - \zeta} + T^{ m - \zeta}\right)$, whenever $m(1/2- 2 \varepsilon) - \zeta>1$ and $ m- \zeta>1$, that is when $m> \frac{ 2(1+ \zeta)}{ 1- 4 \varepsilon}$. We can fix for example $ \varepsilon= 1/8$ and choose an integer $m$ such that $m> 4(1+ \zeta)$. Since $T\geq1$, we have $\bE(B)\leq C\frac{T^{m-\zeta}}{N^m}$ and we obtain:
\begin{equation}
 \bE\left(\sup_{0\leq s< t\leq T}\frac{\Vert Z_{n,t}-Z_{n,s} \Vert_{-1,d}^{2m}}{|t-s|^\gz}\right)\, \leq \, C \frac{T^{m-\gz}}{N^m}\, ,
\end{equation}
which implies
\begin{multline}
 \bP\left( \sup_{0\leq t\leq T} \Vert Z_{n,t} \Vert_{-1,d} \geq \sqrt{\frac{T}{N}}N^\gz\right)\, \leq \, 
 \frac{N^m}{T^m}N^{-2 m\gz}
 \bE\left(\sup_{0\leq t\leq T}\Vert Z_{n,t}\Vert_{-1,d}^{2m}\right)\\
 \leq\, \frac{N^m}{T^{m-\zeta}}N^{-2 m\gz}
 \bE\left(\sup_{0\leq t\leq T}\frac{\Vert Z_{n,t} \Vert_{-1,d}^{2m}}{t^\gz}\right)\,
 \leq\, CN^{-2m\gz}\, .
\end{multline}
We deduce
\begin{equation}
 \bP\left( \sup_{1\leq n\leq n_f}\sup_{0\leq t\leq T} \Vert Z_{n,t} \Vert_{-1,d} \, \geq\,  \sqrt{\frac{T}{N}}N^\gz\right)\, \leq \, 
 C n_f N^{-2m\gz}\, ,
\end{equation}
which tends to $0$ as $N\to\infty$ if we choose $m> \frac{ 1}{ 4 \zeta}$, since $n_f=O(N^{1/2})$. Proposition~\ref{prop:estim_noise} is proved.
\end{proof}

Let us now prove Lemma \ref{lem:moments Z}.

\begin{proof}[Proof of Lemma \ref{lem:moments Z}]
Our aim here is to get the appropriate bounds for the process $Z$.  We follow mostly the ideas of \cite{MR3207725}. Recall that
\begin{equation}
 Z_{n,t}(h)\, =\,\sum_{i=-d}^d\frac{\lambda^i}{N^i}\sum_{j=1}^{N^i} \int_{T_{n-1}}^{T_{n-1}+t}\partial_\theta\left[ \left(
 e^{(t-u)L_{\psi_{n-1}}^*}h\right)^i\right]\big(\varphi^i_j(u)\big)\dd B^i_j(u)\, .
\end{equation}
Let us define the process $Z_{n,t,t'}$ for $0<t'<t$ as
\begin{equation}
 Z^i_{n,t,t'}(h)\, =\, \sum_{i=-d}^d\frac{\lambda^i}{N^i}\sum_{j=1}^{N^i} \int_{T_{n-1}}^{T_{n-1}+t'}\partial_\theta\left[ \left(
 e^{(t-u)L_{\psi_{n-1}}^*}h\right)^i\right]\big(\varphi^i_j(u)\big)\dd B^i_j(u)\, .
\end{equation}
Our aim is to estimate for $0<s'<s<t$, $s'<t'<t$ and for all integers $m>0$ the moments $\bE( \Vert Z_{n,t,t'}-Z_{n,s,s'}\Vert_{-1,d}^{2m})$.
We can decompose $Z_{n,t,t'}-Z_{n,s,s'}$ as follows:
\begin{equation}
 Z_{n,t,t'}-Z_{n,s,s'}\, =\, M^1_{n,s',s,t}+M^2_{n,s',t',t}\, ,
\end{equation}
where
\begin{equation}
 M^1_{n,s',s,t}(h)\, =\,  \sum_{i=-d}^d\frac{\lambda^i}{N^i}\sum_{j=1}^{N^i} \int_{T_{n-1}}^{T_{n-1}+s'}\partial_\theta\left[ \left(
 \left( e^{(t-u)L_{\psi_{n-1}}^*}-e^{(s-u)L_{\psi_{n-1}}^*}\right)h\right)^i\right]\big(\varphi^i_j(u)\big)\dd B^i_j(u)\, ,
\end{equation}
and
\begin{equation}
 M^2_{n,s',t',t}(h)\, =\, \sum_{i=-d}^d\frac{\lambda^i}{N^i}\sum_{j=1}^{N^i} \int_{T_{n-1}+s'}^{T_{n-1}+t'}\partial_\theta\left[ \left(
 e^{(t-u)L_{\psi_{n-1}}^*}h\right)^i\right]\big(\varphi^i_j(u)\big)\dd B^i_j(u)\, .
\end{equation}
The processes $(M^1_{n,s',s,t}(h))_{ s'\in[0,s)}$ and $(M^2_{n,s',t',t}(h))_{t'\in (s',t)}$ are martingales, with It\^o brackets
\begin{equation}
 \left[M^1_{n,\cdot,s,t}(h)\right]_{s'}\, =\, \sum_{i=-d}^d\sum_{j=1}^{N^i}\int_{T_{n-1}}^{T_{n-1}+s'}
 \left( U^{1,i,j}_{n,u,s,t}(h)\right)^2\dd u\, ,
\end{equation}
and
\begin{equation}
 \left[M^2_{n,s',\cdot,t}(h)\right]_{t'}\, =\, \sum_{i=-d}^d\sum_{j=1}^{N^i}\int_{T_{n-1}+s'}^{T_{n-1}+t'}
 \left(U^{2,i,j}_{n,u,t}(h)  \right)^2\dd u\, ,
\end{equation}
where we have used the notations
\begin{equation}
 U^{1,i,j}_{n,u,s,t}(h)\, =\, \frac{\lambda^i}{N^i}\partial_\theta\left[ \left(
 \left( e^{(t-u)L_{\psi_{n-1}}^*}-e^{(s-u)L_{\psi_{n-1}}^*}\right)h\right)^i\right]\big(\varphi^i_j(u)\big)\, ,
\end{equation}
and
\begin{equation}
 U^{2,i,j}_{n,u,t}(h)\, =\, \frac{\lambda^i}{N^i}\partial_\theta\left[ \left(
 e^{(t-u)L_{\psi_{n-1}}^*}h\right)^i\right]\big(\varphi^i_j(u)\big)\, .
\end{equation}
Let $(h_l)_{l\geq 1}$ be a complete orthonormal basis in $H^1_d$. Using Parseval's identity, we obtain
\begin{multline}
\bE \Vert Z_{n,t,t'}-Z_{n,s,s'}\Vert^2_{-1,d}\, =\,  \sum_{l=1}^\infty \bE|(Z_{n,t,t'}-Z_{n,s,s'})(h_l)|^2\\
 \leq\, 2 \sum_{l=1}^\infty \bE |M^1_{n,s',s,t}(h_l)|^2+2 \sum_{l=1}^\infty \bE|M^2_{n,s',t',t}(h_l)|^2\\
 \leq\, 2\sum_{l=1}^\infty \sum_{i=-d}^d\sum_{j=1}^{N^i}\int_{T_{n-1}}^{T_{n-1}+s'}
 \bE\left( U^{1,i,j}_{n,u,s,t}(h_l)\right)^2\dd u
 +2\sum_{l=1}^\infty \sum_{i=-d}^d\sum_{j=1}^{N^i}\int_{T_{n-1}+s'}^{T_{n-1}+t'}
 \bE \left( U^{2,i,j}_{n,u,t}(h_l)\right)^2\dd u\\
 \leq\,  2\sum_{i=-d}^d\sum_{j=1}^{N^i}\int_{T_{n-1}}^{T_{n-1}+s'}
 \bE\Vert U^{1,i,j}_{n,u,s,t}\Vert_{-1,d}^2\dd u
 +2 \sum_{i=-d}^d\sum_{j=1}^{N^i}\int_{T_{n-1}+s'}^{T_{n-1}+t'}
\bE\Vert U^{2,i,j}_{n,u,t}\Vert_{-1,d}^2\dd u\, .
\end{multline}
For $m>1$, we have
\begin{multline}
\bE \Vert Z_{n,t,t'}-Z_{n,s,s'}\Vert^{2m}_{-1,d}\, =\, \bE \left(\sum_{l=1}^\infty |(Z_{n,t,t'}-Z_{n,s,s'})(h_l)|^2
\right)^m\\
 \leq\, m \bE\left( \sum_{l=1}^\infty |M^1_{n,s',s,t}(h_l)|^2\right)^m
 +m \bE\left(\sum_{l=1}^\infty |M^2_{n,s',t',t}(h_l)|^2\right)^m\, ,
\end{multline}
and using  H\"older and Burkholder-Davis-Gundy inequalities, we obtain for the terms involving $M^1$
\begin{align}
\bE \left( \sum_{l=1}^\infty  |M^1_{n,s',s,t}(h_l)|^2\right)^m
&=\, \sum_{l_1,l_2,\ldots,l_m=1}^\infty \bE |M^1_{n,s',s,t}(h_{l_1})|^2 \ldots |M^1_{n,s',s,t}(h_{l_m})|^2 \nonumber\\
&\leq\,  \sum_{l_1,l_2,\ldots,l_m=1}^\infty (\bE |M^1_{n,s',s,t}(h_{l_1})|^{2m})^{1/m} \ldots  (\bE|M^1_{n,s',s,t}(h_{l_m})|^{2m})^{1/m} \nonumber\\
&\leq\, C_m \sum_{l_1,l_2,\ldots,l_m=1}^\infty\bE \left[M^1_{n,\cdot,s,t}(h_{l_1})\right]_{s'} \ldots \bE\left[M^1_{n,\cdot,s,t}(h_{l_m})\right]_{s'} \nonumber\\
&\leq\,  C_m \sum_{l_1,l_2,\ldots,l_m=1}^\infty \left(\sum_{i=-d}^d  \sum_{j=1}^{N^i}\int_{T_{n-1}}^{T_{n-1}+s'}\bE \left( U^{1,i,j}_{n,u,s,t}(h_{l_1})\right)^2\dd u\right)\nonumber\\
 &\qquad\qquad\qquad\ldots
 \left(\sum_{i=-d}^d  \sum_{j=1}^{N^i}\int_{T_{n-1}}^{T_{n-1}+s'}\bE \left( U^{1,i,j}_{n,u,s,t}(h_{l_m})\right)^2\dd u\right)\nonumber\\
 &=\, C_m \left(\sum_{l=1}^\infty \sum_{i=-d}^d  \sum_{j=1}^{N^i}\int_{T_{n-1}}^{T_{n-1}+s'}\bE\left(  U^{1,i,j}_{n,u,s,t}(h_{l})\right)^2\dd u\right)^m \nonumber\\ 
& =\, C_m \left( \sum_{i=-d}^d\sum_{j=1}^{N^i}\int_{T_{n-1}}^{T_{n-1}+s'}
 \bE\Vert U^{1,i,j}_{n,u,s,t}\Vert_{-1,d}^2\dd u\right)^m\, .
\end{align}
The same work can be done for the terms involving $M^2$, which leads to
\begin{multline}\label{eq:first bound moment 2m}
 \bE \Vert Z_{n,t,t'}-Z_{n,s,s'}\Vert^{2m}_{-1,d}\, \leq\,  C'_m\left( \sum_{i=-d}^d\sum_{j=1}^{N^i}\int_{T_{n-1}}^{T_{n-1}+s'}
 \bE\Vert U^{1,i,j}_{n,u,s,t}\Vert_{-1,d}^2\dd u\right)^m\\
 +C'_m\left( \sum_{i=-d}^d\sum_{j=1}^{N^i}\int_{T_{n-1}+s'}^{T_{n-1}+t'}
 \bE\Vert U^{2,i,j}_{n,u,t}\Vert_{-1,d}^2\dd u\right)^m\, .
\end{multline}
It remains now to find appropriate bounds for 
$\bE\Vert U^{1,i,j}_{n,u,s,t}\Vert_{-1,d}^2$ and $\bE\Vert U^{2,i,j}_{n,u,t}\Vert_{-1,d}^2$. On one hand, for $h\in H^1_d$, since $\gd_{\theta_0}\in H^{-1/2-\gep}$ for all $\gep>0$, we have
\begin{equation}
 |U^{2,i,j}_{n,u,t}(h)|\, \leq\, \frac{C}{N^i}
 \left\Vert \partial_\theta\left[\left(e^{(t-u)L^*_\gd} h\right)^i\right]\right\Vert_{1/2+\gep, d}
 \, \leq\,  \frac{C}{N^i}
 \left\Vert \left(e^{(t-u)L^*_\gd} h\right)^i\right\Vert_{3/2+\gep, d}\, .
\end{equation}
For the rest of the proof, we set $ \varepsilon= \frac{ 1}{ 8}$ (any $ \varepsilon\in(0, 1/4)$ would be sufficient). Applying Proposition~\ref{prop:reg_semigroup_adjoint} with $\beta=1/4+\gep/2$, we obtain, for any $ 0<\gamma< \gamma_{L_{ \delta}^{ \ast}}$,
\begin{equation}
 |U^{2,i,j}_{n,u,t}(h)|\, \leq\, \frac{C}{N^i}\left(1+e^{- \gamma(t-u)}(t-u)^{-1/4-\gep/2}\right)\Vert h\Vert_{1,d}\, ,
\end{equation}
which means that $\Vert U^{2,i,j}_{n,u,t}\Vert_{-1,d}\leq  \frac{C}{N^i}\left(1+e^{- \gamma(t-u)}(t-u)^{-1/4-\gep/2}\right)$. On the other hand, proceeding as before, we get the bound:
\begin{equation}
 |U^{1,i,j}_{n,u,s,t}(h)|\, \leq\,  \frac{C}{N^i}
 \left\Vert \left(\left[e^{(t-s)L^*_\gd}-1\right]e^{(s-u)L^*_\gd} h\right)^i\right\Vert_{3/2+\gep, d}\, .
\end{equation}
Applying Proposition \ref{prop:reg_semigroup_adjoint} with $\beta'=1/4+\gep/2$ 
and $\beta=1/4-\gep$, we get for all $\tilde h \in H^{2-\gep}_d$,
\begin{equation}
 \Vert [e^{(t-s)L^*_\gd}-1]\tilde h\Vert_{3/2+\gep,d}\, \leq\, C_\gep (t-s)^{1/4-\gep}\Vert (1-P^{ 0, \ast})\tilde h\Vert_{2-\gep,d}\, .
\end{equation}
For $\tilde h=e^{(s-u)L^*_\gd} h$ and using again Proposition \ref{prop:reg_semigroup_adjoint} with this time $\beta=1/2-\gep/2$, this leads to
\begin{equation}
 |U^{1,i,j}_{n,u,s,t}(h)|\, \leq\, \frac{C}{N^i}(t-s)^{1/4-\gep}(s-u)^{-1/2+\gep/2}e^{- \gamma(s-u)}\Vert h\Vert_{1,d}\, ,
\end{equation}
which means that $\Vert U^{1,i,j}_{n,u,s,t}\Vert_{-1,d}\leq  \frac{C}{N^i}(t-s)^{1/4-\gep}(s-u)^{-1/2+\gep/2}e^{- \gamma(s-u)}$. We can now estimate \eqref{eq:first bound moment 2m}: using that $N\leq cN^i\leq C N$, we obtain
\begin{align}
\bE \Vert Z_{n,t,t'}-Z_{n,s,s'}\Vert^{2m}_{-1,d}\, &\leq\,  \frac{C''_m}{N^m} \left((t-s)^{1/2-2\gep}\int_{0}^{s'}(s-u)^{-1+2\gep}e^{-2 \gamma(s-u)}\dd u\right)^m \nonumber\\
 &+\frac{C''_m}{N^m}\left(\int_{s'}^{t'}\left(1+e^{-2 \gamma(t-u)}(t-u)^{-1/2-\gep}\right)\dd u\right)^m \nonumber\\
 &\leq\, \frac{C'''_m}{N^m}\big((t-s)^{m(1/2-2\gep) }+(t'-s')^{m(1/2-\gep)} +(t'-s')^{m}\big)\, .\label{eq:Znttss}
\end{align}
Taking $t'\nearrow t$ and $s'\nearrow s$ and using Fatou Lemma, we deduce the result.
\end{proof}

\section{Dynamics on the manifold $M$}
\label{sec:dynamics_manifold}
The purpose of this section is to prove the results described in Section~\ref{subsec:dyn M} concerning the process $ \nu_{ n}$ defined in \eqref{eq:nu_stopped}.

Recall that the scheme defined in Section~\ref{subsec:dyn M} starts at a time $t_0^N=O(N^{ -1/2}\log N)$, such that there exists an event $B^N$ with $\bP(B^N)\rightarrow 1$ such that on $B^N$ if we denote $\psi_0=\proj_M(\mu_{N, N^{ 1/2} t^N_0})$ then $\Vert \mu_{N, N^{ 1/2} t^N_0}-q_{\psi_0}\Vert_{-1,d}\leq  N^{-1/2+2\zeta}$. In other words, the initial condition of the scheme satisfies $\Vert \nu_{1,0}\Vert_{-1,d}\leq N^{-1/2+2\zeta}$ on $B^N$.The existence of these times $t_0^N$ and event $B^N$ will be proved in the Section~\ref{sec:approach}. The first result proves estimate \eqref{eq:bound_nunt_intro}:
\begin{proposition}\label{prop:bound_nu}
There exists an event $ \Omega_{ 1}^{ N}$ with $\bP( \Omega_{ 1}^{ N})\rightarrow 1$ as $N\to \infty$ 
such that, almost surely on $ \Omega_{ 1}^{ N}$,
 \begin{equation}
 \sup_{1\leq n\leq n_f}\sup_{t\in [0,T]} \Vert \nu_{n,t}\Vert_{-1,d}\, = \, O( N^{-1/2+2\zeta})\, ,
\end{equation}
where the error $O( N^{-1/2+2\zeta})$ is uniform on $\gO_1^N$.
\end{proposition}

\begin{proof}[Proof of Proposition~\ref{prop:bound_nu}]
Recall the definition of the event $A^{ N}$ in \eqref{eq:AN} and define $\gO^N_1\, :=\, A^N\cap B^N$. Since the purpose of Section~\ref{sec:control_noise} was precisely to prove that $\bP(A^{ N})\to 1$, we obviously have that $\bP(\gO^N_1)\rightarrow 1$, as $N\to\infty$.

Throughout this proof we work on the event $\gO^N_1$ and proceed by induction. We already know that $\Vert \nu_{1,0}\Vert_{-1,d}\leq N^{-1/2+2\zeta}$.
If we suppose that $\Vert \nu_{n,0}\Vert_{-1,d}\leq N^{-1/2+2\zeta}$,
then from the mild formulation \eqref{eq:mild_nu_n}, from \eqref{ineq:L 1} and \eqref{ineq:L 2} and from the estimates on the noise term $Z_{n,t}$ on $\gO_1^N\subset A^{ N}$, we obtain
\begin{multline}
 \Vert \nu_{n,t}\Vert_{-1,d}\, \leq\, C_Le^{-\gamma_L t\wedge \tau^{ n}}N^{-1/2+2\zeta}\\
 +2TC_L\Vert D_{\psi_{n-1}}\Vert_{-1,d}
 + C_L(T+2T^{1/2})\sup_{0\leq s\leq t}\Vert R_{\psi_{n-1}}(\nu_{n,s})\Vert_{-1,d}\\
 +T^{1/2}N^{-1/2+\zeta}\, .\label{eq:est_nu_nt}
\end{multline}
Since the sequence $(\go_i)_{i\geq 1}$ is admissible (recall Definition~\ref{def:disorder}), we have 
\begin{equation}\label{eq:bound_D}
 \Vert D_{\psi_{n-1}}\Vert_{-1,d}\, \leq\, C N^{-1/2}\max_{k=-d,\ldots,d} |\xi_{ N}^k|\, \leq\, C N^{-1/2+\zeta}\, .
\end{equation}
Define the time $t^*$ as
\begin{equation}
 t^* \, :=\, \inf\left\{t\in [0,T]:\, \Vert \nu_{n,t}\Vert_{-1,d} \geq 2 C_LN^{-1/2+2\zeta}\right\}\, .
\end{equation}
Obviously $t^*>0$ and if $t\leq t^*$, one readily sees from \eqref{eq:R} that
\begin{multline}
 \sup_{0\leq s\leq t}\Vert R_{\psi_{n-1}}(\nu_{n,s})\Vert_{-1,d}\\ \leq\, 
 C\bigg(\sup_{0\leq s\leq t}\Vert \nu_{n,s}\Vert_{-1,d}^2 
 + N^{-1/2} \max_{k=-d,\ldots,d} |\xi_{ N}^k|
 \sup_{0\leq s\leq t}\Vert \nu_{n,s}\Vert_{-1,d}\bigg) \leq\, CN^{-1+4\zeta}\, .\label{eq:bound_R_psi_n}
\end{multline}
Putting together \eqref{eq:est_nu_nt}, \eqref{eq:bound_D} and \eqref{eq:bound_R_psi_n} gives that $t^*=T$ if $N$ is large enough. Consequently, by construction of the stopping time $\tau^{ n}$ in \eqref{eq:tau_n}, one has that $\tau^{ n}=T$ and the choice of $T$ (recall \eqref{eq:hyp_bound_T}) implies that
\begin{equation}
\Vert \nu_{n,T}
\Vert_{-1,d}\, \leq\,   \frac{1}{2 C_P}N^{-1/2+2\zeta}\, . 
\end{equation}
To conclude the recursion it remains to show that $\Vert \nu_{n+1,0}\Vert_{-1,d}\leq N^{-1/2+2\zeta}$.
To do this, let us write $\nu_{n+1,0}$ in terms of $\nu_{n,T}$:
\begin{equation}
 \nu_{n+1,0}\, =\, q_{\psi_{n-1}}+\nu_{n,T}-q_{\psi_n}\, .
\end{equation}
Since $P^s_{\psi_n}\nu_{n+1,0}=\nu_{n+1,0}$, where we recall that $P^s_{\psi_{n}}$ is the projection on the space $N_{ \psi_{ n}}$, we can rewrite it as
\begin{multline}\label{decomp_nu0}
 \nu_{n+1,0}\, =\, P^s_{\psi_n}(q_{\psi_{n-1}}+\nu_{n,T}-q_{\psi_n})\\
 =\, P^s_{\psi_n}(q_{\psi_{n-1}}-q_{\psi_n})+(P^s_{\psi_n}-P^s_{\psi_{n-1}})\nu_{n,T} +P^s_{\psi_{n-1}}\nu_{n,T}\, .
\end{multline}
Since $q_{\psi_{n-1}}-q_{\psi_n} = (\psi_{n-1}-\psi_n)q'_{\psi_n}+O((\psi_n-\psi_{n-1})^2)$ (and this estimate makes sense in $H_{ d}^{ -1}$) and $P^{ s}_{ \psi_{ n}} \partial_{ \theta}q_{\psi_n}=0$, 
the first term of the second line of \eqref{decomp_nu0} is of order $O((\psi_n-\psi_{n-1})^2)$. Using the smoothness of the projection $\proj_M$ (Lemma \ref{lem:existence projM}),
\begin{align}
 |\psi_n-\psi_{n-1}| &=\, |\proj_M(\mu_{(n\wedge n_\tau-1)T+t\wedge \tau^n}) -\proj_M(\mu_{((n-1)\wedge n_\tau-1)T+t\wedge \tau^{n-1 }})| \nonumber\\
 &\leq\, C\Vert \mu_{(n\wedge n_\tau-1)T+t\wedge \tau^n}-\mu_{((n-1)\wedge n_\tau-1)T+t\wedge \tau^{ n-1}}\Vert_{-1,d} \nonumber\\
 &\leq\, C\Vert \nu_{n-1,T}\Vert_{-1,d}+C\Vert \nu_{n-1,0}\Vert_{-1,d}\, \leq\, CN^{-1/2+2\zeta}\, .\label{eq:bound_delta_psi}
\end{align}
Combining the last two arguments, we obtain that the first term of the second line of \eqref{decomp_nu0} is of order $O(N^{-1+4\zeta})$.
For the second term, the smoothness of the mapping $\psi\mapsto P^s_\psi$ gives
\begin{equation}
 \Vert (P^s_{\psi_n}-P^s_{\psi_{n-1}})\nu_{n,T}\Vert_{-1,d}\, \leq\, C|\psi_n-\psi_{n-1}|\Vert \nu_{n,T}\Vert_{-1,d}
 \, \leq\, CN^{-1+4\zeta}\, .
\end{equation}
Taking the $H^{-1}_d$ norm on the two sides in \eqref{decomp_nu0}, we obtain
\begin{equation}
 \Vert \nu_{n+1,0}\Vert_{-1,d}\, \leq\, \Vert P^s_{\psi_{n-1}}\nu_{n,T}\Vert_{-1,d}+ O(N^{-1+4\zeta})
 \,\leq\, \frac{1}{2}N^{-1/2+2\zeta}+ O(N^{-1+4\zeta})\, ,
\end{equation}
which implies the result for $N$ large enough.
\end{proof}

We are interested in the rescaled dynamics of the phase of the projection of the empirical measure on $M$ and in particular use the rescaled discretization of this phase dynamics given by the process $\varPsi^N_t$ (recall \eqref{eq:def PsiNt}).

\begin{proposition}\label{prop:expansion_psi}
 There exist a linear form $b:\bbR^{ 2d+1}\rightarrow \bbR$ and an event
 $\gO^N_2$ satisfying $\bP(\gO^N_2)\rightarrow 1$ as $N\to \infty$
 such that on the event $\gO^N_2$ we have for $t\in [t_0^N ,t_f]$:
 \begin{equation}
 \label{eq:PsiN_t}
 \varPsi^N_t\, =\,\psi_0+  b(\xi_N) t+O(N^{-1/4+2\zeta})\, ,
\end{equation}
where the $O(N^{-1/4+2\zeta})$ is uniform on $\gO^N_2$.
\end{proposition}

\begin{proof}[Proof of Proposition~\ref{prop:expansion_psi}]
We work for the moment on the event $\gO^N_1$ defined in the proof of Proposition~\ref{prop:bound_nu}. Using Proposition \ref{prop:bound_nu},
Lemma~\ref{lem:first order projM} below and the fact that $\psi_n=\proj_M(q_{\psi_{n-1}}+\nu_{n,T})$, we have the following first order expansion of $ \varPsi^N_t$ in \eqref{eq:PsiN_t} (recall the definition of $\mathtt{p}$ in \eqref{eq:def p psi} and note that there are $O(N^{1/2})$ terms in the sum):
\begin{equation}
  \varPsi^N_t\, :=\, \psi_0 + \sum_{n=1}^{n_t} \mathtt{p}_{\psi_{n-1}}(\nu_{n,T}) + O(N^{-1/2+4\zeta})\, .
\end{equation}
Let us now decompose the term $\mathtt{p}_{\psi_{n-1}}(\nu_{n,T})$, using the mild formulation \eqref{eq:mild_nu_n}. Remark that $\mathtt{p}_{\psi_{n-1}}(e^{tL_{\psi_{n-1}}}\nu_{n,0})=\mathtt{p}_{\psi_{n-1}}(\nu_{n,0})=0$ and that $\mathtt{p}_{\psi_{n-1}}(e^{(t-s)L_{\psi_{n-1}}}D_{\psi_{n-1}})=\mathtt{p}_{\psi_{n-1}}(D_{\psi_{n-1}})$. Note that Proposition~\ref{prop:bound_nu} shows that $\tau^{n_f}=T$ on $\gO_{ 1}^N$, so that the time integration in the mild formulation \eqref{eq:mild_nu_n} does not involve any stopping time. Hence it remains, since $D_{\psi_{n-1}}$ has no dependency in time,
\begin{equation}
 \mathtt{p}_{\psi_{n-1}}(\nu_{n,T})\, =\, T \mathtt{p}_{\psi_{n-1}}(D_{\psi_{n-1}})-\int_0^T \mathtt{p}_{\psi_{n-1}}\left(e^{(t-s)L_{\psi_{n-1}}} \partial_{ \theta}R_{\psi_{n-1}}(\nu_{n,s})\right)\dd s+ \mathtt{p}_{\psi_{n-1}}(Z_{n,T})\, .
\end{equation}
Using \eqref{ineq:L 2} and \eqref{eq:bound_R_psi_n}
\begin{align}
\left|\int_0^T \mathtt{p}_{\psi_{n-1}}\left(e^{(t-s)L_{\psi_{n-1}}} \partial_{ \theta}R_{\psi_{n-1}}(\nu_{n,s})\right)\dd s\right|&\leq\, \int_{0}^{T} \left\Vert e^{(t-s)L_{\psi_{n-1}}} \partial_{ \theta}R_{\psi_{n-1}}(\nu_{n,s}) \right\Vert_{ -1, d}\dd s \nonumber\\
&\leq\, C\int_{0}^{T} \left(1+ \frac{ 1}{ \sqrt{t-s}}\right)\left\Vert R_{\psi_{n-1}}(\nu_{n,s}) \right\Vert_{ -1, d}\dd s \nonumber\\
&\leq\, C(T+\sqrt{T})N^{-1+4\zeta}\, ,
\end{align}
which leads to
\begin{equation}\label{decom ppsinu}
 \mathtt{p}_{\psi_{n-1}}(\nu_{n,T})\, =\, T \mathtt{p}_{\psi_{n-1}}(D_{\psi_{n-1}})+ \mathtt{p}_{\psi_{n-1}}(Z_{n,T})
 +O(N^{-1+4\zeta})\, .
\end{equation}
We would like to keep only $T \mathtt{p}_{\psi_{n-1}}(D_{\psi_{n-1}})$, since the sum of these terms produce the drift we are looking for,
but unfortunately at each step $\mathtt{p}_{\psi_{n-1}}(Z_{n,T})$
has the same order as $T \mathtt{p}_{\psi_{n-1}}(D_{\psi_{n-1}})$.
To get rid of this extra term $\mathtt{p}_{\psi_{n-1}}(Z_{n,T})$, we use the fact that it is an increment of a martingale and thus averages to $0$ under summation. More precisely, denoting $z_n:=\mathtt{p}_{\psi_{n-1}}(Z_{n,T\wedge\tau^n})$ and using Doob's inequality we obtain,
\begin{equation}\label{eq:bound proba Omega 2}
 \bP \left(\sup_{1\leq m\leq n_f} \left| \sum_{1\leq n\leq m} z_n\right|\, \geq\,  N^{-1/4+2\zeta}\right)\leq N^{1/2-4\zeta}
 \bE \left( \left| \sum_{1\leq n\leq n_f}z_n\right|^2\right)\, ,
\end{equation}
and we have the following decomposition:
\begin{multline}\label{bound z}
 \bE \left( \left| \sum_{1\leq n\leq n_f}z_n\right|^2\right)\, \leq\, \bE\left(\sum_{1\leq n\leq n_f-1} \bE \left[|z_{n+1}|^2|\cF_{T_{n}}\right]\right)\\
 \leq\,C \sum_{1\leq n\leq n_f-1} \bE \left[\Vert Z_{n,T\wedge \tau^n} \Vert^2_{-1,d}\right] 
\,   \leq\, C n_f T N^{-1}\, ,
\end{multline}
where we have used \eqref{eq:Znts}. Since $n_f$ is of order $N^{1/2}$, the probability in \eqref{eq:bound proba Omega 2}
tends to $0$ when $N\rightarrow \infty$ and recalling \eqref{decom ppsinu}, we deduce that there exists an event $\gO^N_2$ satisfying $\bP(\gO^N_2)\rightarrow_{ N\to\infty} 1$ such that on $\gO^N_2$
\begin{equation}\label{exp X}
 \varPsi^N_t\, =\,\psi_0+ T \sum_{n=1}^{n_t} \mathtt{p}_{\psi_{n-1}}(D_{\psi_{n-1}})+O(N^{-1/4+2\zeta})\, .
\end{equation}
The quantity $\mathtt{p}_{\psi_{n-1}}(D_{\psi_{n-1}})=  N^{ -1/2} \mathtt{p}_{ \psi_{ n-1}} \left( -\partial_{ \theta} \left(\xi_{ N}\cdot (J\ast q_{ \psi_{ n-1}}) q_{ \psi_{ n-1}}\right)\right)$ depends linearly in $\xi_N$ and since the model is invariant by rotation, the projection does not depend on $\psi_{n-1}$. So we can write it as $N^{-1/2}b(\xi_N)$, where the linear form $b$ is given by
\begin{equation}
\label{eq:drift_b}
b( \xi)\,:=\,  \mathtt{p} \left( -\partial_{ \theta} \left(\xi\cdot (J\ast q) q\right)\right)= \mathtt{p} \left( -\partial_{ \theta} \left(q\sum_{ k=-d}^{ d}\xi^{ k} (J\ast q^{ k})\right)\right)\, .
\end{equation}
We can rewrite
\eqref{exp X} as
\begin{equation}
 \varPsi^N_t\, =\,\psi_0+ \frac{T}{N^{1/2}} \left\lfloor \frac{ N^{1/2}}{T} (t-t_0^N)\right\rfloor b(\xi_N) +O(N^{-1/4+2\zeta})\, .
\end{equation}
Since $\left| t-t^N_0-\frac{T}{N^{1/2}} \left\lfloor \frac{ N^{1/2}}{T} (t-t_0^N)\right\rfloor\right|\leq \frac{T}{N^{1/2}} $
and $b(\xi_N)=O(N^{\zeta})$, we deduce
\begin{equation}
 \varPsi^N_t\, =\,\psi_0+ b(\xi_N) (t-t_0^N) +O(N^{-1/4+2\zeta})\, ,
\end{equation}
which implies the result, since $t_0^N=O(N^{-1/2}\log N)$. Proposition~\ref{prop:expansion_psi} is proved.
\end{proof}
We can now prove the following result, which together with Proposition \ref{prop:approach_M} implies directly Theorem~\ref{th:main}:
\begin{proposition}
\label{prop:muN_qbxiN}
There exists $N$ sufficiently large such that, on the event $\gO^N_2$, 
 \begin{equation}
 \sup_{t\in [t_0^N, t_f]} \left\Vert \mu_{N,N^{1/2}t}- q_{\psi_0+ b(\xi_N)t}
 \right\Vert_{-1,d}\, =\, O(N^{-1/4+2\zeta})\, ,
\end{equation}
where the error $O(N^{-1/4+2\zeta})$ is uniform on $\gO^N_2$.
\end{proposition}
\begin{proof}[Proof of Proposition~\ref{prop:muN_qbxiN}]
We place ourselves on the event $\gO^N_2$ introduced in the proof of Proposition~\ref{prop:expansion_psi}. For each $t$ such that $N^{1/2}t\in [T_n,T_{n+1}]$ we can decompose $\mu_{N,N^{1/2}t}$ as
\begin{equation}
 \mu_{N,N^{1/2}t}\, =\, q_{\psi_n}+\nu_{n+1,N^{1/2}t-T_n}\, .
\end{equation}
But Proposition~\ref{prop:bound_nu} implies that $\nu_{n+1,N^{1/2}t-T_n}=O(N^{-1/2+2\zeta})$ and for such time $t$ we have
\begin{equation}
 q_{\psi_n}\, =\, q_{\varPsi^N_t}\, =\,  q_{\psi_0+ b(\xi_N)t} + O(N^{-1/4+2\zeta})\, ,
\end{equation}
where we have used Proposition~\ref{prop:expansion_psi}.
\end{proof}

\section{Approaching the manifold}
\label{sec:approach}
The purpose of this section is to prove Proposition~\ref{prop:approach_M}. We follow here the same ideas as in \cite{Bertini:2013aa}, Section~5. From now on, we fix $ \varepsilon_{ 0}>0$ and $ p_{ 0}\in H_{ d}^{ -1}$ such that ${\rm dist}_{ H_{ d}^{ -1}}(p_{ 0}, M)\leq \varepsilon_{ 0}$. The parameter $ \varepsilon_{ 0}$ will be chosen sufficiently small in the following. We proceed in three steps:
\begin{enumerate}
\item We rely on the convergence in finite time of the empirical measure $ \mu_{ N, t}$ to the solution $p_{t}$ of \eqref{eq:FP_Kur_finite}
starting from $p_0$ in order to show that $\mu_{N,t}$ approaches $M$ (up to a distance of order $\gep_{ 0}$). This step requires a time interval of order $\log \gep_{ 0}$.
\item We use the linear stability of $M$ under \eqref{eq:FP_Kur_finite} and control the noise terms of the dynamics to show that the empirical measure approaches $M$ up to a distance of order $N^{-1/2+2\zeta}$. This step requires a time interval of order $\log N$.
\item We show that the empirical measure stays at distance $N^{-1/2+2\zeta}$ from $M$ up to the time $t_0^N$.
\end{enumerate}

\textbf{First step.} As explained in Section \ref{sec:lin stab}, the stability of $M$ implies that if $\gep_0$ is small enough
the deterministic solution $p_t$ of the limit PDE \eqref{eq:FP_Kur_finite} with initial condition $p_0$ converges to a $q_{\theta_0}\in M$. 
In particular, after a time $s_1$, $p_t$ satisfies $\Vert p_{s_1}-q_{\theta_0}\Vert_{-1,d}\leq \gep_{ 0}$. Due to the linear stability of $M$, this time $s_1$ is of order $-\frac 1\gamma_L \log \gep_{ 0}$.

In order to show that the empirical measure is close to the deterministic trajectory $p_t$ when $N$ is large, we use a mild formulation similar to the one obtained in Section~\ref{sec:mild_formulation}, but this time relying on the $(2d+1)$-dimensional Laplacian operator $\Delta_d$. More precisely using similar argument as in Section~\ref{sec:mild_formulation}, one can obtain the following equality
in $H^{-1}_d$:
\begin{multline}\label{eq:mild_Delta}
\mu_{N,t}-p_t\, =\, e^{\frac t2\Delta_d}(\mu_{N,0}-p_0)-\int_0^t e^{\frac {t-s}{2}\Delta_d}\Bigg[\partial_\theta\left( \mu_{N,s}\otimes \go+\mu_{N,t}\sum_{k=-d}^d \gl^k_N J*\mu^k_{N,s}\right)\\
-\partial_\theta \left( p_{s}\otimes \go+p_s \sum_{k=-d}^d \gl^k J*p^k_{s}\right)\Bigg]\dd s+z_t\, ,
\end{multline}
where $z_t$ satisfies, for all test function $f= (f^{ -d}, \ldots, f^{ d})$
 \begin{equation}
  z_t(f)\, =\,\sum_{i=-d}^d \frac{\gl^i}{N^i}\sum_{j=1}^{N^i}
  \int_0^t\partial_\theta\left[ \left( e^{\frac {t-s}{2}\Delta_d}
  f\right)^i\right](\varphi^i_j(s))\dd B^i_j(s)\, .
 \end{equation}
Since $ \Delta_d$ is simply the classical one-dimensional Laplacian operator $\Delta$ on each coordinate, it is sectorial (in fact self-adjoint) with negative spectrum. Using the classical bound $\Vert e^{t\Delta} f\Vert_{-1}\leq \frac{C}{\sqrt t}\Vert f\Vert_{-2}$ for the one-dimensional Laplacian operator, we directly obtain
\begin{equation}\label{eq:smoothness_Delta}
 \Vert e^{t\Delta_d} f\Vert_{-1,d}\, \leq\,  \frac{C}{\sqrt t}
\Vert f\Vert_{-2,d}\, ,
\end{equation}
and with similar estimates as the one used in Section~\ref{sec:control_noise}, one can show that the event $B_1^N$ defined as
\begin{equation}
 B^N_1\, :=\,  \left\{\sup_{0\leq t\leq s_1} \Vert z_t\Vert_{-1,d}\,  \leq\,  \sqrt{\frac{t_1}{N}}N^\zeta\right\}
\end{equation}
satisfies $\bP(B^N_1)\rightarrow 1$ as $N\to\infty$. Let us write the shortcut
\[U_{ N, s, t}:= e^{\frac {t-s}{2}\Delta_d}\Bigg[\partial_\theta\left( \mu_{N,s}\otimes \go+\mu_{N,t}\sum_{k=-d}^d \gl^k_N J*\mu^k_{N,s}\right) -\partial_\theta \left( p_{s}\otimes \go+p_s \sum_{k=-d}^d \gl^k J*p^k_{s}\right)\Bigg]\, ,\]for the term within the integral in \eqref{eq:mild_Delta}. Note that the mapping $(\mu,\nu)\mapsto \partial_\theta(\mu J*\nu)$ satisfies (see \cite{Bertini:2013aa}, Lemma A.3 for a proof)
\begin{equation}
\label{eq:bound_partial_Jnumu}
 \Vert \partial_\theta(\mu J*\nu)\Vert_{-2}\, \leq\, C\Vert \mu\Vert_{-1}\Vert \nu\Vert_{-1}\, .
\end{equation}
Using \eqref{eq:smoothness_Delta} and \eqref{eq:bound_partial_Jnumu}, we obtain
\begin{align}\label{eq:bound_muJmu}
\left\Vert U_{ N, s, t} \right\Vert_{ -1, d} &=\, \left\Vert e^{\frac {t-s}{2}\Delta_d} \partial_\theta \left( \mu_{N,s}\sum_{k=-d}^d  \gl^k_N J*\mu^k_{N,s}\right) -\partial_\theta \left( p_{s} \sum_{k=-d}^d \gl^k J*p^k_{s}\right)\right\Vert_{-1,d} \nonumber\\
&\leq\,  \frac{ C}{ \sqrt{t-s}}\left\Vert \partial_\theta \left( \mu_{N,s}\sum_{k=-d}^d  \gl^k_N J*\mu^k_{N,s}\right) -\partial_\theta \left( p_{s} \sum_{k=-d}^d \gl^k J*p^k_{s}\right)\right\Vert_{-2,d} \nonumber\\
&\leq\,  \frac{ C}{ \sqrt{t-s}} \sum_{i=-d}^d\sum_{k=-d}^d \gl^k\left\Vert \partial_\theta \left( \mu^i_{N,s} J*\mu^k_{N,s}\right)-\partial_\theta \left( p^i_{s}J*p^k_{s}\right)\right\Vert_{-2}\\
&\quad \quad +\frac{ C}{ \sqrt{t-s}}\sum_{i=-d}^d\sum_{k=-d}^d \left\vert \gl^k_N-\gl^k \right\vert\left\Vert \partial_\theta \left( \mu^i_{N,s}J*\mu^k_{N,s}\right)\right\Vert_{-2}\\
&\leq\,  \frac{ C}{ \sqrt{t-s}} (\Vert p_s\Vert_{-1,d}+\Vert \mu_{N,s}\Vert_{-1,d}) \Vert \mu_{N,s}-p_s\Vert_{-1,d}+\frac{ C}{ \sqrt{t-s}} N^{-1/2+\zeta} \Vert \mu_{N,s}\Vert_{-1,d}^2\\
&\leq \, \frac{ C'}{ \sqrt{t-s}} \left( \Vert \mu_{N,s}-p_s\Vert_{-1,d}+ N^{-1/2+\zeta}\right)\, ,
\end{align}
where we have used in particular \eqref{eq:proba in Hminus1}, since both $p_{ s}$ and $ \mu_{ N, s}$ are probabilities. Let us place ourselves on the event
\begin{equation}
\label{eq:B2N}
 B^N_2\, :=\, \left\lbrace \Vert \mu_{N,0}-p_0\Vert_{-1,d}\leq \frac{ \varepsilon_{ 0}}{ 2} \right\rbrace \cap B^N_1\, ,
\end{equation} 
which satisfies obviously $\bP(B_{ 2}^{ N})\to 1$ as $N\to \infty$. Then, for all $t\leq s_1$, \eqref{eq:smoothness_Delta} and \eqref{eq:bound_muJmu} imply that \eqref{eq:mild_Delta} can be rewritten on the event $B^N_2$ as
\begin{equation}
 \Vert \mu_{N,t}-p_t\Vert_{-1,d}\, \leq\, \frac{ \varepsilon_{ 0}}{ 2} +C\sqrt{\frac{s_1}{N}}N^\zeta +C \int_0^t \frac{1}{\sqrt{t-s}}\Vert \mu_{N,s}-p_s\Vert_{-1,d}\dd s\, ,
\end{equation}
so applying the Gronwall-Henry inequality (\cite{Henry1981}, Lemma~7.1.1 and Exercise 1), one obtains that for some $a>0$ (independent from $N$ and $\gep_{ 0}$), on the event $B^N_2$ and for all $t\leq s_1$ 
\begin{equation}
 \Vert \mu_{N,t}-p_t\Vert_{-1,d}\, \leq\, 2\left( \frac{ \varepsilon_{ 0}}{ 2} +C\sqrt{\frac{s_1}{N}}N^\zeta\right)e^{a s_1}\, .
\end{equation}
We deduce that for $N$ large enough, the projection \[\psi_0^1:=\proj_M(\mu_{N,s_1} )\] is well defined and $\Vert \mu_{N,s_1}-p_{s_1}\Vert_{-1,d}\leq \varepsilon_{ 0}$ on $B^N_2$, which means that $|\psi_0^1-\theta_0|\leq C\gep_{ 0} $ and $ \left\Vert \mu_{ N, s_{ 1}} - q_{ \theta_{ 0}}\right\Vert\leq 2 \varepsilon_{ 0}$.

\medskip

\textbf{Second step.} 
Now that we know that $\text{dist}(\mu_{N,s_1},M)\leq 2\gep_{ 0}$ with increasing probability as $N\to\infty$, we can use a similar scheme as the one defined in Section~\ref{subsec:dyn M} to show that the empirical measure approaches $M$ up to a distance $N^{-1/2+2\zeta}$ with high probability. Since this part is very similar to the work done in Section~\ref{sec:dynamics_manifold}, we do not specify all the details.

We consider the evolution of the dynamics on time intervals $[\tilde T_n,\tilde T_{n+1}]$ with $\tilde T_n=s_1+n\tilde T$ where $\tilde T$ is such that $e^{-\gamma_L \tilde T}\leq \frac{1}{4 C_L C_P}$. We consider also a sequence of real numbers $h_n$ satisfying $h_1=2\gep_{ 0}$ and $h_{n+1}=\frac{h_{n}}{2}$ and take this time the number of step $\tilde n_f$ of our scheme as
\begin{equation}
 \tilde n_f\, :=\, \inf\left\{n:\, h_n\leq N^{-1/2+2\zeta}\right\}\, .
\end{equation}
It is clear that $\tilde n_f$ is of order $O(\log N)$. To ensure the existence of the projections of the process on $M$ at each step, we introduce, as in
Section~\ref{subsec:dyn M}, the stopping time
\begin{equation}
 (\tilde n_{\tau},\tilde \tau)\, =\, \inf\{(n,t)\in \{1,\ldots,\tilde n_f\}\times [0,\tilde T] :\, 
 \Vert \mu_{\tilde T_{n-1}+t}-q_{\ga_{n-1}}\Vert_{-1,d}\geq \gs\}\, ,
\end{equation}
where $\ga_{n} = \proj_M(\mu_{\tilde T_{n}})$ when it exists.
This allows us to define
the random phases 
$\tilde \psi_{n-1}$ defined as
\begin{equation}
 \tilde \psi_{n-1}\, :=\, \proj_M(\mu_{(n\wedge \tilde n_{\tau}-1)\tilde T+t\wedge \tilde \tau^n})\, ,
\end{equation}
and the processes $\tilde \nu_{n,t}$
defined for $n=1,\ldots,\tilde n_f$ as
\begin{equation}
 \tilde \nu_{n,t}\, :=\, \mu_{(n\wedge \tilde n_\tau-1)\tilde T+t\wedge \tilde \tau^n}-q_{\tilde \psi_{n-1}}\, .
\end{equation}
This last process satisfies the mild equation
\begin{equation}\label{eq:mild tilde nu}
 \tilde \nu_{n,t}\, =\, e^{(t\wedge \tilde \tau^n) L_{{\tilde \psi_{n-1}}}}\tilde \nu_{n,0}-\int_0^{t\wedge\tilde \tau^n}
 e^{(t\wedge\tilde \tau^n-s)L_{{\tilde \psi_{n-1}}}}(D_{\tilde \psi_{n-1}}
 +R_{\tilde \psi_{n-1}}(\tilde \nu_{n,s}))\dd s + \tilde Z_{n,t\wedge \tilde \tau^n}\, ,
\end{equation}
where $\tilde Z_{n,t}$ is defined as
\begin{equation}
 \tilde Z_{n,t}(f)\, =\, \sum_{i=-d}^d 
 \frac{1}{N^i}\sum_{j=1}^{N^i} \int_{0}^{t} \partial_\theta\left[\left(
 e^{(t-s)L_{\tilde \psi_{n-1}}^*}f\right)^i\right](\varphi^i_j(\tilde T_{n-1}+s))\dd B^i_j(\tilde T_{n-1}+s)\, .
\end{equation}
Section~\ref{sec:control_noise} shows that the event
\begin{equation}
 \tilde A^N\, =\, \left\{ \sup_{1\leq n\leq \tilde  n_f}\sup_{t\in [0,\tilde T]} \Vert \tilde Z_{n,t}\Vert_{-1,d} \leq   \tilde T^{1/2}N^{-1/2+\zeta}\right\}\, ,
\end{equation}
satisfies $\bP(\tilde A^N)\rightarrow 1$ as $N\to\infty$.

\medskip

In the first step of this proof we have shown, since $\tilde \psi_0=\psi_0^1$, that, on the event $B^N_2$ (recall \eqref{eq:B2N}), we have 
$\Vert \tilde \nu_{1,0}\Vert_{-1,d}=\Vert \mu_{N,s_1}-q_{\psi_0^1}\Vert_{-1,d}\leq h_1$.
Our aim is to prove that on the event $B^N_3$ defined as
\begin{equation}
 B^N_3\, :=\, \tilde A^N\cap B^N_2\, ,
\end{equation}
we have $\Vert \tilde \nu_{n,0}\Vert_{-1,d}\leq h_n$
for all $n=1,\ldots,\tilde n_f$. This would imply, using the notations $s_2= \tilde T_{n_f}$ and $\psi_0^2=\proj_M(\mu_{N,s_2})$, that
$\Vert \tilde \mu_{N,s_2}-q_{\psi_0^2} \Vert_{-1,d}\leq N^{-1/2+2\zeta}$.
We place ourselves on the event $B^N_3$.
From the mild formulation \eqref{eq:mild tilde nu}, if $n<\tilde n_f$ and
$\Vert \tilde \nu_{n,0}\Vert_{-1,d}\leq h_n$ we get
\begin{multline}
 \Vert \tilde \nu_{n,t}\Vert_{-1,d}\, \leq\, C_L e^{-\gamma_L t\wedge \tilde \tau^{ n}}h_{n}\\
 +2C_L\tilde T\Vert D_{\tilde \psi_{n-1}}\Vert_{-1,d}
 +C_{L} \left(\tilde T + 2 \tilde T^{ 1/2}\right)\sup_{0\leq s\leq t}\Vert R_{\tilde \psi_{n-1}}(\tilde \nu_{n,s})\Vert_{-1,d}\\
 +\tilde T^{1/2}N^{-1/2+\zeta}\, .
\end{multline}
Consider the time $\tilde t^*$ defined as 
\begin{equation}
 \tilde t^* \, :=\, \inf\left\{t\in [0,\tilde T]:\, \Vert \tilde  \nu_{n,t}\Vert_{-1,d}\geq 2C_L h_n\right\}\, .
\end{equation}
For all $t\leq \tilde t^*$ we have
\begin{multline}
 \sup_{0\leq s\leq t}\Vert R_{\tilde \psi_{n-1}}(\tilde \nu_{n,s})\Vert_{-1,d}\\ \leq\, 
 C\bigg(\sup_{0\leq s\leq t}\Vert \tilde \nu_{n,s}\Vert_{-1,d}^2 
 + N^{-1/2} \max_{k=-d,\ldots,d} |\xi_{ N}^k|
 \sup_{0\leq s\leq t}\Vert \tilde \nu_{n,s}\Vert_{-1,d}\bigg)\\
 \leq\, C(C_L^2h_n^2+C_LN^{-1/2+\zeta}h_n)\, .
\end{multline}
The last quantity is smaller than $C(N, \varepsilon_{ 0}) h_{ n}$, where $C(N, \varepsilon_{ 0})\to 0$ as $N\to\infty$ and $ \varepsilon_{ 0}\to 0$. On the other hand we have shown in \eqref{eq:bound_D} that $\Vert D_{\tilde \psi_{n-1}}\Vert_{-1,d}\leq C N^{-1/2+\zeta}$.
Since $n<\tilde n_f$ we have $h_N > \frac{1}{2C_L}N^{-1/2+2\zeta}$, which means
that $N^{-1/2+\zeta}$ is negligible with respect to $h_n$ for $N$ large enough.
So for $N$ large enough, $\tilde t^*\geq \tilde T$
and we have (recall that $e^{-\gl \tilde T}\leq \frac{1}{4C_LC_P}$)
\begin{equation}
\Vert \tilde \nu_{n,\tilde T}\Vert_{-1,d}\, \leq\, \frac{1}{4 C_P} h_n + o(h_n)\,  \leq\, \frac{3}{8 C_P} h_n \, ,
\end{equation}
when $\gep_{ 0}$ is small enough. It remains to show that $\Vert \tilde \nu_{n+1,0}\Vert_{-1,d}\leq \frac{h_n}{2}$ to conclude the recursion.
We do not prove it in details, since it can be done by proceeding exactly as in the proof of Proposition~\ref{prop:bound_nu}, decomposing $\Vert \tilde \nu_{n,\tilde T}\Vert_{-1,d}$ and
showing that it can be written as
\begin{equation}
 \Vert \tilde \nu_{n+1,0}\Vert_{-1,d}\, \leq\, \Vert P^s_{\tilde \psi_{n-1}}\tilde \nu_{n,\tilde T}\Vert_{-1,d}+ O(h_n^2)
 \, ,
\end{equation}
which implies that $ \Vert \tilde \nu_{n+1,0}\Vert_{-1,d}\leq \frac{3}{8}h_n+ O(h_n^2)\leq \frac{h_n}{2}$
on the event $B^N_3$ when $\gep_{ 0}$ is small enough and concludes the recursion. Note that the estimate for $\tilde \psi_n-\tilde \psi_{n-1}$ obtained 
in \eqref{eq:bound_delta_psi} leads to
\begin{equation}
 |\psi_0^2-\psi_0^1|\, \leq\, \sum_{n=1}^{n_f} |\tilde \psi_n-\tilde \psi_{n-1}|\, \leq \, C\sum_{n=1}^{n_f} h_n
 \, \leq\, 2C h_1\leq  4C\gep_{ 0}\, ,
\end{equation}
on the event $B^N_2$, which gives $ \left\vert \psi_0^2-\theta_0 \right\vert\leq C'\gep_{ 0}$ for some $C'$.

\medskip

\textbf{Third step.} 
In the previous step, we have constructed a time $s_2$
such that $s_2\leq -\frac 1\gl \log \gep_{ 0}+C_1\log N $ for some constant $C_1$ and such that
$\Vert \tilde \mu_{N,s_2}-q_{\psi_0^2} \Vert_{-1,d}\leq N^{-1/2+2\zeta}$
with high probability. We can now consider a time $s_3=c\log N$ for $c=C_1+1$, which does not depend in $\gep_{ 0}$. For $N$ large enough, we obviously have $s_3>s_2$. In order to prove that $\Vert \tilde \mu_{N,s_3}-q_{\psi_0^3} \Vert_{-1,d}\leq N^{-1/2+2\zeta}$ with high probability, where $\psi_0^3=\proj_M( \mu_{N,s_3})$, it suffices to decompose the dynamics on the interval $[s_2,s_3]$ according to an iterative scheme with time step $\hat T$ satisfying $e^{-\gamma_L \hat T}\leq \frac{1}{4 C_L C_P}$ as does $T$ and apply exactly the same procedure as in Proposition \ref{prop:bound_nu}. 

This last step induces a phase shift $|\psi^3_0-\psi^2_0|\leq C  N^{-1/2+2\zeta}\log N\leq C\gep_{ 0}$, for $N$ large enough. This concludes the proof, with $t_0^N=N^{-1/2}s_3$.

\section{Estimates on the drift $ b$}
\subsection{The case of a symmetric disorder}
\label{sec:drift_symmetric}
We prove here Proposition~\ref{prop:drift_symmetric} and drop for simplicity the dependency in $ \psi$ and $ \delta$. We consider $ \xi= ( \xi^{ -d}, \ldots, \xi^{ d})$ such that $ \xi^{ -i}= \xi^{ i}$ for all $i=1, \ldots, d$ and aim at proving that $ b(\xi)=0$, where the drift $b(\xi) = \mathtt{p} \left( - \partial_{ \theta} \left( \left\lbrace \sum_{ k=-d} ^{ d} \xi^{ k} (J\ast q^{ k}) \right\rbrace q\right)\right)$ is given by \eqref{eq:drift_b}.  

The space of regular ($\cC^{ 2}$, say) test functions $f= (f^{ -d}, f^{ -(d-1)}, \ldots, f^{ d-1}, f^{ d})$ can be naturally decomposed into the direct sum of the space $\cO$ (resp. $\cE$) of odd (resp. even) test function in both variables $( \theta, i)$, that is $f\in \cO$ (resp. $f\in\cE$) if and only if $ f^{ -i}(- \theta)= - f^{ i}( \theta)$ (resp. $ f^{ -i}( - \theta) = f^{ i} (\theta)$) for all $ \theta\in \bbT$ and $i=0, \ldots, d$. One easily sees from the definition of $J(\cdot)$ in \eqref{eq:def_J} and the definition of $q$ in \eqref{eq:def_q} that $q \in \cE$ and $ \left((J\ast q^{ -d}), \ldots, (J\ast q^{ d})\right)\in \cO$. Let us denote $Q(\theta):= \sum_{ k=-d} ^{ d} \xi^{ k} (J\ast q^{ k})(\theta)$. Using that $ \xi^{ -i}= \xi^{ i}$, one obtains that $Q( \theta) = \xi^{ 0}(J\ast q^{ 0})(\theta) + \sum_{ k=1}^{ d} \xi^{ k}\left((J\ast q^{ k})(\theta) + (J\ast q^{ -k})(\theta)\right)$, so that we deduce that $Q$ is an odd function of $ \theta$ and that $ \theta \mapsto Q(\theta) q(\theta)\in \
cO$. Consequently 
$ 
\partial_{ \theta} \left(Q(\theta) q(\theta)\right)\in \cE$. Hence, in order to prove Proposition~\ref{prop:drift_symmetric}, it suffices to prove that
\begin{equation}
\forall h\in\cE,\ \mathtt{p}(h)=0.
\end{equation}
This is indeed the case since one easily sees from the definition \eqref{eq:def_Lpsiq} of the operator $L=L_{ \psi, \delta}$ that $L(\cE)\subset \cE$ and $L(\cO)\subset \cO$ and since $\mathtt{p}$ is the projection on the eigenfunction $ \partial_{ \theta} q \in \cO$. Proposition~\ref{prop:drift_symmetric} is proved.
\subsection{Small $ \delta$ asymptotics of the drift}
\label{sec:asymptotic_drift}
Our aim here is to prove Proposition~\ref{prop:xpansion v} that gives the first order expansion of the drift $b(\xi)$ defined in \eqref{eq:drift_b} as $ \delta\to 0$. Due to the rotational invariance of the system, we can work with the stationary solution $q_{0,\gd}$ that we denote $q_\gd$ throughout this section. We denote $\mathtt{p}_\gd$ as $\mathtt{p}_{ \psi=0, \delta}$ (recall \eqref{eq:def p psi}) and $D_\gd(\xi)$ as $D_{N,0,\gd}$, (recall \eqref{eq:D}). With these notations the drift $b$ is given by
\begin{equation}
 b(\xi)\, =\, \mathtt{p}_\gd(D_\gd(\xi))\, .
\end{equation}
When $ \delta=0$, it is straightforward to see that $ q_{ \delta}= (q^{ -d}_{ \delta}, \ldots, q^{ d}_{ \delta})$ is equal to $(q_{ 0}, \ldots, q_{ 0})$, where $q_{ 0}$ is the stationary solution of the nonlinear Fokker-Planck equation without disorder \eqref{eq:FP_without_disorder}. We refer to Section~\ref{sec:case_delta_0} below for precise definitions (see in particular \eqref{eq:q_0_no_des} and \eqref{eq:fixed point without disorder} where the normalisation factor $\cZ_{ 0}$ and the fixed-point parameter $r_{ 0}$ are defined). The following result (proved in Appendix \ref{sec:appendix expansion delta}) provides the next order of the approximation of $q_{ \delta}$ as $ \delta\to 0$.
\begin{lemma}\label{lem:exp q}
For $i=-d,\ldots,d$ we have
 \begin{equation}
 \label{eq:expansion_q_delta}
 q^i_\gd(\theta)\, =\, q_0(\theta) + \gd\go^i \kappa(\theta) q_0(\theta)+O(\gd^2)\, ,
\end{equation}
where
\begin{multline}
 \kappa(\theta)\, =\,2\theta
 +4\pi\frac{\int_\theta^{2\pi}e^{-2Kr_0\cos u}\dd u}{\cZ_0}-2\frac{\int_0^{2\pi} e^{2Kr_0\cos u}u\dd u}{\cZ_0}\\
 -4\pi  \frac{\int_0^{2\pi}e^{2Kr_0\cos u}\int_u^{2\pi} e^{-2Kr_0\cos v}\dd v\dd u}{\cZ_0^2}\, ,
\end{multline}
and where the error $O(\delta^{ 2})$ is uniform in $ \theta\in\bbT$.
\end{lemma}

The projection $\mathtt{p}_\gd$ also converges in some sense to the projection $\mathtt{p}_0$ on the tangent space of the stable circle of stationary profiles
of \eqref{eq:FP_without_disorder} at $q_0$. 
Moreover, the system given by \eqref{eq:FP_without_disorder} admits a nice Hilbertian structure, which allows to know
$\mathtt{p}_0$ explicitly. This allows us to obtain the following first order expansion of $\mathtt{p}_\gd$, whose proof is given in Appendix~\ref{sec:appendix expansion delta}.
\begin{lemma}\label{lem:exp proj p}For all coordinate by coordinate primitive $(\cU^{-d},\ldots, \cU^d)$ of $u$ smooth, we have
 \begin{equation}
 \label{eq:expand_p_delta}
  \mathtt{p}_\gd(u)\, =\, \frac{\cZ_0^2}{\cZ_0^2-4\pi^2}\sum_{k=-d}^d \gl^k \int_\bbT \left( 1-\frac{2\pi}{\cZ_0^2 q_0}  \right)\cU^k + O(\gd\Vert u\Vert_{-1,d})\, .
 \end{equation}
\end{lemma}
We have now the tools required to obtain the first order expansion of the drift $b(\xi)$. The result we want to prove is
\begin{proposition}\label{prop:exp b}For all $\xi$ such that $\sum_{k=-d}^d \xi^k=0$ we have
\begin{equation}
 b(\xi)\, =\, \gd \sum_{k=-d}^d \xi^k \go^k +O(\gd^2)\, .
\end{equation}
\end{proposition}

\begin{proof}[Proof of Proposition~\ref{prop:exp b}]
First remark that when $\gd=0$, we obtain, using Lemma \ref{lem:exp q}, that for all $i=-d,\ldots,d$:
\begin{equation}
 D_0^i(\xi)\, =\,
 \partial_\theta \left[q_0\sum_{k=-d}^d \xi^k J*q_0\right]\, =\, 0\, ,
\end{equation}
since $\sum_{k=-d}^d \xi^k=0$. We deduce, using again Lemma \ref{lem:exp q}, the following expansion for
$D^i_\gd(\xi)$:
\begin{multline}
 D^i_\gd(\xi)\, =\, \gd \go^i  \partial_\theta \left[ \kappa q_0\sum_{k=-d}^d \xi^k J*q_0\right]
 + \gd  \partial_\theta \left[q_0\sum_{k=-d}^d \xi^k \go^k J*(\kappa q_0)\right]+O(\gd^2)\\
 =\, \gd  \partial_\theta \left[q_0\sum_{k=-d}^d \xi^k \go^k J*(\kappa q_0)\right]+O(\gd^2)\, ,
\end{multline}
where we have used again the fact that $\sum_{k=-d}^d \xi^k=0$. Applying Lemma \ref{lem:exp proj p}, we deduce
\begin{equation}
 b(\xi)\, =\, \gd \left[\frac{\cZ_0^2}{\cZ_0^2-4\pi^2}\int_\bbT \left( 1-\frac{2\pi}{\cZ_0^2 q_0}  \right) q_0 J*(\kappa q_0) \right]
 \sum_{i=-d}^d \sum_{k=-d}^d  \gl^i \xi^k \go^k
 +O(\gd^2)\, ,
\end{equation}
and recalling that $\sum_{i=-d}^d \gl^i=1$ and denoting
\begin{equation}
 c_b\, :=\,  \frac{\cZ_0^2}{\cZ_0^2-4\pi^2} \int_\bbT \left( 1-\frac{2\pi}{\cZ_0^2 q_0}  \right)q_0 J*(\kappa  q_0) \, ,
\end{equation}
we simply obtain
\begin{equation}
 b(\xi)\, =\, \gd c_b \sum_{k=-d}^d \xi^k \go^k +O(\gd^2)\, .
\end{equation}
It remains to show that $c_b=1$.
Now using the fact that $J(\theta-\theta')=-K\sin\theta \cos\theta'+K\cos\theta\sin\theta'$, $\int_0^{2\pi}\sin(\theta)q_0(\theta)=0$
and $\int_0^{2\pi}\cos(\theta)q_0(\theta)=r_0$, we obtain
\begin{equation}
 \int_0^{2\pi}q_0(\theta)J*(\kappa q_0)(\theta)\dd\theta\, =\, Kr_0 \int_0^{2\pi} \sin(\theta')\kappa(\theta')q_0(\theta')\dd\theta'\, ,
\end{equation}
and
\begin{equation}
 \int_0^{2\pi}J*(\kappa q_0)(\theta)\dd\theta\, =\, 0\, .
\end{equation}
So the constant $c_b$ can be simplified as follows
\begin{equation}
 c_b\, =\, \frac{K r_0\cZ_0^2}{\cZ_0^2-4\pi^2}\int_0^{2\pi} \sin(\theta)\kappa(\theta)q_0(\theta)\dd\theta\, ,
\end{equation}
which leads to
\begin{equation}
c_b\, =\, \frac{2Kr_0 \cZ_0^2}{ \cZ_0^2-4\pi^2}\Bigg[\int_0^{2\pi}\sin\theta\frac{e^{2Kr_0\cos\theta}}
 {\cZ_0}\bigg(
 \theta+2\pi \frac{\int_\theta^{2\pi} e^{-2Kr_0\cos u}\dd u}{\cZ_0}\bigg)\dd \theta\Bigg]\, .
\end{equation}
Integrating by parts and using the fact that $\partial_\theta[e^{2Kr_0\cos(\theta)}]
=-2Kr_0\sin\theta e^{2Kr_0\cos\theta}$,  we obtain
\begin{equation}
 \int_0^{2\pi}\theta \sin\theta e^{2Kr_0\cos\theta}\dd\theta\, =\, -\frac{\pi e^{2Kr_0}}{Kr_0}+\frac{\cZ_0}{2Kr_0}\, ,
\end{equation}
and
\begin{equation}
 \int_0^{2\pi}\sin\theta e^{2Kr_0\cos\theta}\int_\theta^{2\pi} e^{-2Kr_0\cos u}\dd u\, =\, \frac{e^{2Kr_0}\cZ_{ 0}}{2Kr_0}
 -\frac{\pi}{Kr_0}\, ,
\end{equation}
which implies that $c_b = \frac{2Kr_0 \cZ_0^2}{ \cZ_0^2-4\pi^2}\left(\frac{1}{2Kr_0}-\frac{4\pi^2}{2Kr_0 \cZ_0^2}\right)=1$. Proposition~\ref{prop:exp b} is proved.
\end{proof}
Using Proposition \ref{prop:exp b}, we can now compute the first order of the variance $v^2$ of the limiting normal distribution of $b(\xi_N)$ when the disorder is i.i.d:
\begin{proof}[Proof of Proposition \ref{prop:xpansion v}]
From the Central Limit Theorem, we know that $\xi_N$ converges as $N\to\infty$ to a Gaussian distribution with mean $0$ and covariance matrix $\Sigma$ satisfying
\begin{equation}
\left\{
\begin{array}{clc}
 \Sigma_{k,k}  \,= & \gl^k(1-\gl^k)   &    k\in \{-d,\ldots,d\}\, , \\
 \Sigma_{k,l}  \, = & -\gl^k\gl^l      & \, k,l\in \{-d,\ldots,d\}, \, k\neq l\, .
\end{array}
\right.
\end{equation}
Applying Proposition \ref{prop:exp b} we obtain
\begin{equation}
 v^2\, =\, \gd^2\left(\sum_{k\in \{-d,\ldots,d\}} \gl^k(1-\gl^k)(\go^k)^2-\sum_{ k,l\in \{-d,\ldots,d\}, \, k\neq l}\gl^k\gl^l \go^k\go^l \right)
 +O(\gd^3)\, ,
\end{equation}
and since $\gl^{-k}=\gl^k$ and $\go^{-k}=-\go^k$ the terms with $l\neq -k$ cancel in the second sum, which gives the result.

\end{proof}

\begin{appendix}

\section{Construction of rigged-spaces}
\label{sec:appendix_rigged_spaces}
We specify here the construction of the Hilbert distributions spaces we work with in this paper. It is based on the notion of \emph{rigged Hilbert spaces} (see \cite{MR697382}, p. 81).
\subsection{Functional spaces on $\bbT$}
Consider $\bL_{ 0}^{ 2}:= \left\lbrace u\in \bL^{ 2},\ \int_{ \bbT} u(\theta)\dd \theta=0\right\rbrace$, the space of square integrable functions with zero mean value, endowed with the norm $ \left\Vert u \right\Vert_{ 2}:= \left( \int_{ \bbT} u(\theta)^{ 2} \dd \theta\right)^{ \frac{ 1}{ 2}}$. We call a \emph{weight} any strictly positive function $ \theta \mapsto w(\theta)$ on $\bbT$. For any weight $w$ on $\bbT$, define $H_{w}^{ 1}$ as the closure of $ \left\lbrace u\in\cC^{ 1}(\bbT), \int_{ \bbT} u(\theta) \dd \theta=0\right\rbrace$ w.r.t. the norm
\[ \left\Vert u \right\Vert_{ 1, w}:= \left( \int_{ \bbT} \left( \partial_{ \theta} u(\theta)\right)^{ 2} w(\theta) \dd \theta\right)^{ \frac{ 1}{ 2}}.\]
There is a continuous and dense injection of $H_{ w}^{ 1}$ into $\bL_{ 0}^{ 2}$ and the corresponding dual space can be identified as $H^{ -1}_{ 1/w}$, that is the closure of $ \left\lbrace u\in\cC^{ 1}(\bbT),\ \int_{ \bbT}u(\theta)\dd \theta=0\right\rbrace$ under the norm
\[ \left\Vert u \right\Vert_{ -1, 1/w}:=  \left(\int_{ \bbT} \frac{ \cU(\theta)^{ 2}}{ w(\theta)} \dd \theta\right)^{ \frac{ 1}{ 2}},\]where $\cU$ is the primitive of $u$ such that $\int_{ \bbT} \frac{ \cU}{ w}=0$.
\subsection{Functional spaces on $\bbT\times \bbR$}
The correct set-up of the paper is to consider test functions of both oscillators and frequencies, that is $(\theta, \omega) \mapsto u(\theta, \omega)$, where $ \theta\in\bbT$ and $ \omega\in\bbR$. Since the disorder is assumed to take a finite number of values $ \left\lbrace \omega^{ -d}, \ldots, \omega^{ d}\right\rbrace$, it is equivalent to consider vector-valued test functions $ \theta \mapsto (u^{ -d}(\theta), \ldots, u^{ d}(\theta))$ and it is straightforward to define the counterparts of the norms defined in the last paragraph for these vector-valued functions: Consider $\bL_{ 0, d}^{ 2}:= \left(\bL_{ 0}^{ 2}\right)^{ d}$ endowed with the product norm
\[ \left\Vert u \right\Vert_{ 2, d}:= \left( \sum_{ k=-d}^{ d} \lambda^{ k}\left\Vert u^{ k} \right\Vert_{ 2}^{ 2}\right)^{ \frac{ 1}{ 2}}.\]
In the same way, consider the space $H_{ w, d}^{ 1}$, closure of $ \left\lbrace (u^{ -d}, \ldots, u^{ d})\in\cC^{ 1}(\bbT), \int_{ \bbT} u^{ k}(\theta)\dd \theta=0\right\rbrace$ under the norm
\begin{equation}
\label{eq:norm1w}
\left\Vert u \right\Vert_{ 1, w, d}\, :=\,  \left(\sum_{ k=-d}^{ d} \lambda^{ k}\left\Vert u^{ k} \right\Vert_{ 1, w}^{ 2}\right)^{ \frac{ 1}{ 2}},
\end{equation}
as well as the space $H_{ 1/w, d}^{ -1}$ endowed with the norm 
\begin{equation}
\label{eq:normHm1d}
\left\Vert u \right\Vert_{ -1, 1/w, d}\, :=\,  \left(\sum_{ k=-d}^{ d} \lambda^{ k}\left\Vert u^{ k} \right\Vert^{ 2}_{ -1, 1/w}\right)^{ \frac{ 1}{ 2}}.
\end{equation}
Note that if $w_{ 1}$ and $w_{ 2}$ are bounded weights, the norms $ \left\Vert \cdot \right\Vert_{ 1, w_{ 1}}$ and $ \left\Vert \cdot \right\Vert_{ 1, w_{ 2}}$ (resp. $ \left\Vert \cdot \right\Vert_{ -1, 1/w_{ 1}}$ and $ \left\Vert \cdot \right\Vert_{ -1, 1/w_{ 2}}$) are equivalent. The same holds for the $(2d+1)$-dimensional norms.

\subsection{Fractional spaces}
Define also the fractional norm $\Vert\cdot\Vert_{\alpha, d}$ (where $ \alpha\geq 0$): consider $\Delta_d$ the Laplacian operator on each coordinate, $\Vert\cdot\Vert_0$ the $L^2$-norm on $\bbT$ and $\Vert u \Vert_{0, d}^2=\sum_k \gl_k \Vert u^k\Vert_{0}^2$ and define
\begin{equation}
\label{eq:norm_laplacian}
 \Vert u\Vert_{\ga, d}^2\, =\, \Vert (1-\Delta_d)^{\ga/2} u\Vert^2_{0, d}\, =\, \sum_{ k=-d}^{ d} \gl_k \Vert (1-\Delta)^{\ga/2} u^k\Vert^2_0\, .
\end{equation}
We denote as $H^{ \alpha}_{ d}$ the closure of regular functions with zero mean-value on $\bbT$ under the previous norm and $H^{ - \alpha}_{ d}$ the corresponding dual space.
\section{Spectral estimates and regularity results on semigroups}
\label{sec:appendix_spectral_semigroups}
The purpose of this paragraph is to establish spectral estimates on $L_{ \psi, \delta}$ and its adjoint as well as regularity estimates on their semigroups $ e^{ tL_{ \psi, \delta}}$ and $ e^{ tL_{ \psi, \delta}^{ \ast}}$.
\subsection{The case $ \delta=0$}
\label{sec:case_delta_0}
The analysis of the dynamics of \eqref{eq:eds_Kur} and \eqref{eq:FP_Kur_finite} is based on perturbations argument on the mean-field plane rotators system \eqref{eq:Kur_without_disorder} and \eqref{eq:FP_without_disorder}. The proof relies in particular strongly on the fact that \eqref{eq:Kur_without_disorder} is reversible, with an explicit free energy \cite{BGP,dahms}. However, one should note that the limit as $\gd\rightarrow 0$ of \eqref{eq:eds_Kur} or \eqref{eq:FP_Kur_finite} is slightly different to the mean-field model \eqref{eq:Kur_without_disorder} - \eqref{eq:FP_without_disorder}. In particular, \eqref{eq:FP_Kur_finite} becomes as $ \delta\to 0$
\begin{equation}\label{eq:FP_non_disordered}
 \partial_t q^i_t(\theta)\, =\, \frac12 \partial^2_\theta q^i_t(\theta)- \partial_\theta\left(q^i_t(\theta)\left(\sum_{k=-d}^d p^k J*q^k_t (\theta)\right)\right)\, ,\ i=-d, \ldots, d\, ,
\end{equation}
which corresponds to the situation where the disorder is no longer present but where the rotators have been (artificially) separated in different subpopulations. Following the terminology of \cite{MR3207725} where \eqref{eq:FP_non_disordered} has been already encountered, we call this system the \emph{non-disordered system}. It is shown in \cite{MR3207725}, Section~2.1, that the non-disordered system \eqref{eq:FP_non_disordered} presents most of the properties of the mean field plane rotators model \eqref{eq:FP_without_disorder}. In particular, for all $K>1$, one can show that \eqref{eq:FP_non_disordered} admits a unique circle $M_{0,nd}$ of synchronized profiles, that is stable as $t\to\infty$. $M_{ 0, nd}$ is given by the translations of the profile $q_{0,nd}=(q_0,\ldots,q_0)$, where $q_0$ is the profile generating the stable circle $M_0$ of non trivial solutions of \eqref{eq:FP_without_disorder}, namely
\begin{equation}
\label{eq:q_0_no_des}
 q_0(\theta)\, :=\, \frac{e^{2Kr_0\cos\theta}}{\cZ_{ 0}(2Kr_{ 0})}\, ,
\end{equation}
where $\cZ_{ 0}(x)= 2\pi I_{ 0}(x)$, $I_{ 0}(x)= \frac{ 1}{ 2\pi} \int_{0}^{2\pi} e^{ x\cos(\theta)}\dd \theta$ is the standard modified Bessel function of order $0$ and $r_0$ is the unique positive solution of the fixed-point problem
\begin{equation}\label{eq:fixed point without disorder}
 r_0\, =\, \Psi_0(2K r_0)\, ,\quad \text{with} \quad \Psi_0(x)\, :=\, \frac{\int_0^{2\pi} \cos (\theta) e^{x\cos \theta}\dd \theta}{\cZ_0(x)}\, .
\end{equation}
The derivation of these stationary solutions is highly similar to the procedure described in Section~\ref{sec:stationary_solutions_disorder} and we refer to the aforementioned references for more details. Note that one can draw a simple correspondance between the present definitions and the definitions of Section~\ref{sec:stationary_solutions_disorder} in the case of $ \delta=0$: namely, one readily sees that, for any $i=-d, \ldots, d$, $S^{ i}_{ 0}(\theta, x)= e^{ x\cos(\theta)} \cZ_{ 0}(x)$ (recall \eqref{eq:S_delta}) and $Z^{ i}_{ 0}(x)= \cZ_{ 0}(x)^{ 2}$, so that the definition of $ \Psi_{ \delta}$ when $ \delta=0$ (recall \eqref{eq:fixed_point_Psi}) coincides with $ \Psi_{ 0}$ given in \eqref{eq:fixed point without disorder}.

\subsection{Spectral estimates when $ \delta=0$}

Define the linearized operator around any stationary solution $q_{ 0, nd}\in M_{ 0, nd}$:
\begin{equation}
\label{eq:A}
(Au)^{ i}\, =\,  \frac{ 1}{ 2} \partial_{ \theta}^{ 2} u^{ i} - \partial_{ \theta} \left((J\ast q_{ 0})u^{ i} + q_{ 0} \sum_{ k=1}^{ d} J\ast u^{ k}\right),\ i=-d, \ldots, d\, ,
\end{equation}
with domain $\cD(A)=\left\lbrace (u^{ -d}, \ldots, u^{ d})\in \cC^{ 2}(\bbT)^{ 2d+1},\ \int_{ \bbT} u^{ k}(\theta)\dd \theta=0, k=-d, \ldots, d\right\rbrace$.
We recall the following result (see \cite{MR3207725}, Proposition 2.1):
\begin{proposition}
\label{prop:A_self_adjoint}
A is essentially self-adjoint with compact resolvent in $H_{ 1/q_{ 0}, d}^{ -1}$. Its spectrum lies in $(-\infty, 0]$, $0$ is a simple eigenvalue, with eigenspace spanned by $ \partial_{ \theta} q_{ 0, nd}$. The spectral gap between $0$ and the rest of the spectrum is denoted as $ \gamma_{ A}$.
\end{proposition}
One can deduce from Proposition~\ref{prop:A_self_adjoint} similar spectral properties of its dual $A^{ \ast}$ in $\bL_{ 0, d}^{ 2}$:
\begin{equation}
\label{eq:A_ast}
(A^{ \ast}v)^{ i}\, :=\,  \frac{ 1}{ 2} \partial_{ \theta}^{ 2} v^{ i} + (J\ast q_{ 0})\partial_{ \theta} v^{ i} - \int_{ \bbT} \left((J\ast q_{ 0})\partial_{ \theta} v^{ i}\right)\dd \theta- \sum_{ k=1}^{ d} \lambda^{ k} J\ast (q_{ 0} \partial_{ \theta} v^{ k}),\ i=-d, \ldots, d\, ,
\end{equation}
with domain $\cD(A^{ \ast})= \cD(A)$.
\begin{proposition}
\label{prop:A_ast_sectorial 2}
$A^\ast$ is essentially self-adjoint with compact resolvent in $H_{ 1/q_{ 0}, d}^{ 1}$. Its spectrum lies in $(-\infty, 0]$,
and $0$ is a simple eigenvalue and its spectral gap $ \gamma_{ A^{ \ast}}$ is equal to $ \gamma_{ A}$.
\end{proposition}

\begin{proof}[Proof of Proposition~\ref{prop:A_ast_sectorial 2}]
Let us introduce the operator $U$ defined from $H_{ q_{ 0}, d}^{ 1}$ to $H_{ 1/q_{ 0}, d}^{ -1}$ as
\[ Uf(\theta):=- \partial_{ \theta}( q_{ 0}(\theta) \partial_{ \theta} f(\theta)).\] $U$ is an isometry between $H_{ q_{ 0}, d}^{ 1}$ and $H_{ 1/q_{ 0}, d}^{ -1}$: $U$ realizes a bijection from $\lbrace  u\in\cC^{ \infty}(\bbT)^{ d}, \linebreak \int_{ \bbT} u^{ k}(\theta) \dd \theta=0, k=1, \ldots, d\rbrace$ into itself and for every $f,g \in H_{ q_{ 0}, d}^{ 1}$,
\begin{multline}
\left\langle Uf,Ug \right\rangle_{ -1, 1/q_{ 0}, d}\,= \, 
\sum_{ k}\int_{\bbT} \frac{ \left(q_{ 0}(\theta) \partial_{ \theta} f^{ k}(\theta)\right)\left(q_{ 0}(\theta) \partial_{ \theta} g^{ k}(\theta)\right)}
{ q_{ 0}(\theta)} \dd \theta
\\=\,   \sum_{ k}\int_{\bbT}  q_{ 0}(\theta) \partial_{ \theta} f^{ k}(\theta) \partial_{ \theta} g^{ k}(\theta)\dd \theta
\, =\,  \left\langle f ,g\right\rangle_{ 1, q_{ 0}, d}\, .
\end{multline}
Moreover, the following identity holds:
\begin{equation}
\label{eq:UA}
  A^{ \ast}\, =\,U^{-1} A U \, ,
\end{equation}
so the operators $A$ on $H^{-1}_{1/q_0,d}$ and $A^\ast$ on $H^1_{1/q_0,d}$ have the same structural and spectral properties.
\end{proof}

\subsection{Spectral estimates of $L_{ \psi, \delta}$ and its adjoint}
We are in position to deduce spectral estimates on the disordered operators $L_{ \delta}$ and its adjoint $L_{ \delta}^{ \ast}$ in $\bL_{ 0, d}^{ 2}$ (we drop the index $ \psi$ in this section for simplicity). 
\begin{proposition}
\label{prop:def_adjoint}
The adjoint $L_{ \delta}^{ \ast}$ of $L_{ \delta}$ in $\bL_{ 0, d}^{ 2}$ is given by
for all $i=1, \ldots, d$
\begin{multline}
(L^*_{ \delta} v)^i\,  =\, \frac12 \partial_{ \theta}^{ 2}v^i +\gd\go^i \partial_{ \theta}v^i+(\partial_{ \theta}v^i)\sum_{k=-d}^d \lambda^k J*q^k_{ \delta} - \int_{ \bbT}  \left((\partial_{ \theta} v^{ i}) \sum_{ k=-d}^{ d} \lambda^{ k}J\ast q_{ \delta}^{ k}\right)\dd \theta\\-\sum_{k=-d}^d \lambda^k J*(q^k_{ \delta} \partial_{ \theta}v^k)\, ,
\end{multline}
with domain $D(L_{ \delta}^{ \ast})= D(A)$.
\end{proposition}
\begin{proof}[Proof of Proposition~\ref{prop:def_adjoint}]
For all regular $u$ and $v$,
\begin{align*}
 \langle L_\delta^* v,u\rangle_{ 2, d}\, &=\, \langle v, L_\delta u\rangle_{ 2, d}\\
 &=\, \sum_{i=-d}^d \lambda^i \left\langle v^i, \frac12 \partial_{ \theta}^{ 2}u^i-\gd \go^i \partial_{ \theta}u^i -\partial_{ \theta}\left(u^i\sum_{k=-d}^d \lambda^k J*q_\delta^k+q_\delta^i\sum_{k=-d}^d \lambda^k J*u^k\right)\right\rangle_{ 2, d}\\
 &=\sum_{i=-d}^d\lambda^i\left \langle \frac12 \partial_{ \theta}^{ 2}v^i +\gd\go^i \partial_{ \theta}v^i+\partial_{ \theta}v^i\sum_{k=-d}^d \lambda^k J*q^k_\delta - \int_{ \bbT}  \left(\partial_{ \theta} v^{ i} \sum_{ k=-d}^{ d} J\ast q_{ \delta}^{ k}\right)\dd \theta, u^i \right\rangle_{ 2, d}\\ 
 &+\sum_{i=-d}^d \sum_{k=-d}^d\lambda^i\lambda^k \langle q^i_\delta \partial_{ \theta}v^i, J* u^k\rangle_{ 2, d}\\
& =\, \left\langle \frac12 \partial_{ \theta}^{ 2}v^i +\gd\go^i \partial_{ \theta}v^i+(\partial_{ \theta}v^i)\sum_{k=-d}^d \lambda^k J*q^k_\delta - \int_{ \bbT}  \left((\partial_{ \theta} v^{ i}) \sum_{ k=-d}^{ d} \lambda^{ k}J\ast q_{ \delta}^{ k}\right)\dd \theta, u^i\right\rangle_{ 2, d}\\ &- \left\langle \sum_{k=-d}^d 
 \lambda^k J*(q^k_\delta \partial_{ \theta}v^k), u^i\right\rangle_{ 2, d}\, ,
\end{align*}
which precisely gives \eqref{eq:L_delta_psi0_ast}.
\end{proof}
The main result of this section is the following
\begin{proposition}
\label{prop:L_ast_sectorial}
 There exists $\gd_{ 2}= \delta_{ 2}(K)>0$ such that for all $\gd\leq\gd_{ 2}$, everything that follows is true: the operator $L^*_\gd$ (resp. $L_{ \delta}$) is sectorial in $H^1_{q_0,d}$ (resp. $H_{ 1/q_{ 0}, d}^{ -1}$), its spectrum lies in a sector of the type $\{ \gl\in\bbC:\, |\arg(\gl)|>\pi/2+ \alpha\}$ for some $ \alpha>0$ and $0$ is an isolated eigenvalue for $L_{ \delta}^{ \ast}$ (resp. $L_{ \delta}$), at a distance from the rest of the spectrum denoted by $ \gamma_{L_{ \delta}^{ \ast}}$ (resp. $ \gamma_{ L_{ \delta}}$). Moreover, both $L_{ \delta}$ and $L_{ \delta}^{ \ast}$ generate a $\cC_{ 0}$-semigroup $t \mapsto e^{ tL_{ \delta}}$ (resp. $t \mapsto e^{ tL_{ \delta}^{ \ast}}) $in $\bL_{ 0, d}^{ 2}$ and $ e^{ tL_{ \delta}^{ \ast}}= \left(e^{ tL_{ \delta}}\right)^{ \ast}$.
 \end{proposition}
\begin{proof}[Proof of Proposition~\ref{prop:L_ast_sectorial}]
The result concerning the operator $L_{ \delta}$ has been proved in \cite{MR3207725}, Th. 2.5. For the sake of completeness, we recall here the main arguments concerning $L_{ \delta}^{ \ast}$ in $H_{ q_{ 0}, d}^{ 1}$ but we refer to \cite{MR3207725}, Section~6.2 for precise details.
Note that we need a precise control of the spectrum of $L_{ \delta}^{ \ast}$ around the origin. In particular, one has to ensure that the spectrum of $L_{ \delta}^{ \ast}$ remains in the negative part of the complex plane. We write $L^{ \ast}_{ \delta}$ as a perturbation for small disorder of the non-disordered case:
\begin{equation}
\label{eq:B}
L^{ \ast}_{ \delta} = A^{ \ast}+ B_{ \delta},
\end{equation}
where $A^{ \ast}$ is given in \eqref{eq:A_ast} and $B_{ \delta}$ is a small perturbation as $ \delta\to 0$. More precisely, following the exact same strategy as in \cite{MR3207725}, Proposition 6.5, p. 356, one obtains that the operator $B_{ \delta}$ is $A^{ \ast}$-bounded: there exist constants $a_{ \delta}$ and $b_{ \delta}$ (only depending on $ \delta$ and $K$) such that for all $u$ in the domain of (the closure of) $A^{ \ast}$
 \begin{equation}
 \label{eq:B_A_bounded}
 \left\Vert B_{ \delta}u \right\Vert_{ 1, q_{ 0}, d} \, \leq\,  a_{ \delta} \left\Vert u \right\Vert_{ 1, q_{ 0}, d} + b_{ \delta} \left\Vert A^{ \ast} u \right\Vert_{ 1, q_{ 0}, d},
 \end{equation}
 with $a_{ \delta}=O(\gd)$ and $b_{ \delta}=O(\gd)$, as $ \delta\to 0$. Note that the only things that differs between this result and \cite{MR3207725}, Proposition~6.5 is that we work here with an $H^{ 1}$-norm whereas the result in \cite{MR3207725} concerns an $H^{ -1}$-norm. 
 
Fix some $ \varepsilon>0$ (that will be specified later) and define $L_{ \delta, \varepsilon}^{ \ast}:= L_{ \delta}^{ \ast} - \varepsilon$ and $A_{ \varepsilon} := A- \varepsilon$, so that $L_{ \delta, \varepsilon}^{ \ast}= A_{ \varepsilon}^{ \ast} + B_{ \delta}$. Fix $ \alpha\in(0, \frac{ \pi}{ 2})$ and introduce the following subset of the complex plane
\[ \Sigma_{ \alpha}:= \left\lbrace \lambda\in\bbC,\ \left\vert \arg(\lambda) \right\vert< \frac{ \pi}{ 2} + \alpha\right\rbrace\cup\{0\}.\]
The operator $A_{ \varepsilon}$ (as $A$ itself) is self-adjoint in $H_{ -1, 1/q_{ 0}}$ and hence, sectorial. In particular, there exists $M>0$ such that $ \left\Vert R( \lambda, A_{ \varepsilon}) \right\Vert_{ H_{ 1/q_{ 0}, d}^{ -1}} \leq \frac{ M}{ \left\vert \lambda \right\vert}$, for all $ \lambda\in \Sigma_{ \alpha}$. Note that the constant $M$ is indeed independent of $ \varepsilon>0$ and that the previous inequality is also true for $A$ in place of $A_{ \varepsilon}$ (see \cite{MR3207725}, (6.12)). Using \eqref{eq:UA}, one obtains that $ \left\Vert R( \lambda, A_{ \varepsilon}^{ \ast}) \right\Vert_{ H_{ q_{ 0}, d}^{ 1}}\leq \frac{ M}{ \left\vert \lambda \right\vert}$. For $ \lambda\in \Sigma_{ \alpha}$, $u\in H_{ q_{ 0}, d}^{ 1}$,
\begin{align*}
\left\Vert B_{ \delta} R(\lambda, A^{ \ast} )u \right\Vert_{ 1, q_{ 0}, d} &\leq\,  a_{ \delta} \left\Vert R(\lambda, A^{ \ast}) u \right\Vert_{ 1, q_{ 0}, d} + b_{ \delta} \left\Vert A^{ \ast}R( \lambda, A^{ \ast})u \right\Vert_{ 1, q_{ 0}, d},\\
&\leq \, \frac{ M a_{ \delta}}{ \left\vert \lambda \right\vert} \left\Vert u \right\Vert_{ 1, q_{ 0}, d} + (M+1) b_{ \delta} \left\Vert u \right\Vert_{ 1, q_{ 0, d}},
\end{align*}
Choose $ \delta$ sufficiently small so that $b_{ \delta} (1+M)\leq \frac{ 1}{ 4}$ and $ \frac{ a_{ \delta} M}{ \varepsilon}\leq \frac{ 1}{ 4}$. Then for $ \left\vert \lambda \right\vert > \varepsilon\geq 4 M a_{ \delta}$, we have $ \left\Vert B_{ \delta}R( \lambda, A^{ \ast})u\right\Vert_{ 1}\leq \frac{ 1}{ 2} \left\Vert u \right\Vert_{ 1}$ so that the operator $1- B_{ \delta}R(\lambda, A^{ \ast})$ is invertible from $H^{ 1}_{ q_{ 0}, d}$ into itself, with norm smaller than $2$. A simple computation shows that in this case
\[(\lambda- (A^{ \ast}+B_{ \delta}))^{ -1}= R(\lambda, A^{ \ast}) (1- B_{ \delta} R(\lambda, A^{ \ast}))^{ -1},\]which gives that, for $ \lambda\in \Sigma_{ \alpha}$, $ \left\vert \lambda \right\vert> \varepsilon$, $\left\Vert R(\lambda, L_{ \delta}^{ \ast})\right\Vert_{ H_{ q_{ 0}, d}^{ 1}} \leq \frac{ 2M}{ \left\vert \lambda \right\vert}$.
Consequently, the spectrum of $L_{ \delta}^{ \ast}$ is contained in
\[ \Theta_{ \alpha, \varepsilon}:= \left\lbrace \lambda\in\bbC,\ \frac{ \pi}{ 2} + \alpha \leq \arg(\lambda) \leq \frac{ 3\pi}{ 2} - \alpha\right\rbrace \cup \left\lbrace \lambda\in\bbC,\ \left\vert \lambda \right\vert\leq \varepsilon\right\rbrace.\] In particular, $0\in \rho(L_{ \delta, 2\varepsilon}^{ \ast})$ and for all $ \lambda\in \bbC$ with $\Re(\lambda)>0$ (hence $ \left\vert \lambda \right\vert < \left\vert \lambda + 2 \varepsilon\right\vert$), $ \left\Vert R(\lambda, L_{ \delta, 2 \varepsilon}^{ \ast}) \right\Vert_{ H_{ q_{ 0}, d}^{ 1}}\leq \frac{ M}{ \left\vert \lambda+2 \varepsilon \right\vert}\leq \frac{ M}{ \left\vert \lambda\right\vert}$. The fact that this estimate can be extended to some $ \Sigma_{ \alpha^{ \prime}}$ for some $ \alpha^{ \prime}$ is a consequence of a Taylor's expansion argument (see \cite{MR3207725}, Proposition~6.2), so that $L_{ \delta, 2\varepsilon}^{ \ast}$ (and $L_{ \delta}^{ \ast}$) is indeed sectorial.

At this point, we cannot rule out the possibility that some elements of the spectrum of $L_{ \delta}^{ \ast}$ may lie in $ \Theta_{ \varepsilon, \alpha}\cap \left\lbrace \lambda\in\bbC,\ \Re(\lambda)>0\right\rbrace$. The last point of the proof is to show that one can choose $ \varepsilon$ and a smaller $ \delta$ such that this situation does not hold: choose $ \varepsilon= \frac{ \gamma_{A}}{ 2}>0$, where $ \gamma_{ A}$ is the spectral gap of $A$. In particular, the circle centered in $0$ with radius $ \varepsilon$ separates the eigenvalue $0$ (of multiplicity $1$) from the rest of the spectrum of $A^{ \ast}$. An application of \cite{Kato1995}, Theorem IV-3.18, p. 214, shows that one can choose $ \delta$ sufficiently small so that the spectrum of the perturbed operator $L_{ \delta}^{ \ast}$ is likewise separated by this circle: for such $ \delta$, there is a unique eigenvalue (with multiplicity $1$) within the boundary of this circle). But we know already that $0$ is an eigenvalue for the perturbed 
operator $L_{ \delta}^{ \ast}$. By uniqueness, we conclude that there is no eigenvalue in the positive part of the complex plane. We leave the details of this argument to \cite{MR3207725}, Section~6.2.5. 

Using \cite{Pazy1983}, Corollary 10.6, p. 41, $L_{ \delta}^{ \ast}$ is the generator of the adjoint of $t \mapsto e^{ t L_{ \delta}}$ in $\bL_{0, d}^{ 2}$, which is a $C_{ 0}$-semigroup. This concludes the proof of Proposition~\ref{prop:L_ast_sectorial}. 
\end{proof}

\subsection{Equivalence of norms}
For any $0\leq \beta\leq 1$, consider the interpolation norm $\Vert\cdot\Vert_{V^\beta}$ associated to the sectorial operator $1- L_{ \delta}^{ \ast}$ defined as
\begin{equation}
\label{eq:norm_V_beta}
 \Vert u\Vert_{V^\beta}\, =\, \Vert (1-L_{ \delta}^{ \ast})^\beta u\Vert_{1,q_{ 0}, d}\, .
\end{equation}
Recall also the definition of the fractional norm in \eqref{eq:norm_laplacian}.
\begin{lemma}
\label{lem:equiv_norm_beta}
Under the assumptions of Proposition~\ref{prop:L_ast_sectorial}, for any $0\leq \beta\leq 1$, there exists $c_{ 1}, C_{ 1}>0$ such that for all $u$,
\begin{equation}
\label{eq:equiv_norm_beta}
 c_1 \Vert u\Vert_{1+2\beta, d}  \, \leq\, \Vert u\Vert_{V^\beta}\, \leq\, C_1 \Vert u\Vert_{1+2\beta, d}\, .
\end{equation}
\end{lemma}
\begin{proof}[Proof of Lemma~\ref{lem:equiv_norm_beta}]
We can decompose $L_{ \delta}^{ \ast}$ as follows:
\begin{equation}
 L_{ \delta}^{ \ast}\, =\, \frac{1}{2}\Delta_d + R\, ,
\end{equation}
where, for all $i=1, \ldots, d$
\begin{multline}
 (R v)^i\, =\, \gd\go^i \partial_{ \theta}v^i+\partial_{ \theta}v^i\sum_{k=-d}^d \lambda^k J*q^k_0-\sum_{k=-d}^d  \lambda^k J*(q^k_0 (\partial_{ \theta}v^k))\\
 - \int_\bbT \left(\partial_{ \theta}v^i(\theta)\sum_{k=-d}^d \lambda^k J*q^k_0(\theta)\right)\dd\theta\, .
\end{multline}
Since $R$ only contains first order derivatives and $J$ and $q^k_0$ are smooth, it is easy to see that for all $u\in H^{2}_d$, we have
\begin{equation}
 \Vert R u \Vert_{1, d}\, \leq\, C\Vert u\Vert_{2, d}\, .
\end{equation}
One deduces immediately from this estimate that there exists a constant $C>0$ such that for $u\in H_{ d}^{ 2}$, 
\begin{equation}
\label{eq:comp_L_Delta}
\left \Vert \left[2(1-L_{ \delta}^{ \ast})-(1-\Delta_d) \right] u\right\Vert_{1, d}\, \leq\, C\Vert u\Vert_{2, d}\, .  
\end{equation}
Consequently, the operator $[2(1-L_{ \delta}^{ \ast})-(1-\Delta_d)](1-\Delta_d)^{-1/2}$ is bounded in $H_{ d}^{ 1}$. Since $1-L_{ \delta}^{ \ast}$ is sectorial in $H_{ d}^{ 1}$ with the same domain as $ \Delta_{ d}$, an application of \cite{Henry1981}, Theorem 1.4.8 shows that the norms $ \left\Vert (1- L_{ \delta}^{ \ast})^{ \beta}\cdot \right\Vert_{ 1, d}$ and $ \left\Vert (1- \Delta_{ d})^{ \beta} \cdot \right\Vert_{ 1, d}$ are equivalent. The norm equivalence \eqref{eq:equiv_norm_beta} follows directly from the definitions \eqref{eq:norm_V_beta} and \eqref{eq:norm_laplacian}.
\end{proof}

\subsection{Regularity of semigroups}
Recall here the definition of the projection $P^{ 0}_{ \psi, \delta}$ on the kernel $ \Span(\partial_{ \theta} q_{ \psi, \delta})$ of $L_{ \psi, \delta}$ defined in Section~\ref{sec:lin stab}. We drop here the dependance on $ \psi$ for simplicity. The corresponding projection on the kernel of $L_{ \delta}^{ \ast}$ is given by $P_{ \delta}^{ 0, \ast}$. This kernel is one-dimensional, spanned by some $ \theta \mapsto v_{ 0}(\theta)$ and there exists a linear form $\tilde p$, bounded on $H_{ d}^{ 1}$ such that, for all $u\in H_{ d}^{ 1}$, $P^{ 0, \ast}_{ \delta}u= \tilde p(u) v_{ 0}$. Note that it is easy to see that $v_{ 0}$ is a regular ($\cC^{ \infty}$) function on $\bbT$.

\begin{proposition}
\label{prop:reg_semigroup_adjoint}
Suppose the assumptions of Proposition~\ref{prop:L_ast_sectorial} are true. For any $ \gamma\in [0, \gamma_{L_{ \delta}^{ \ast}})$, any $ \beta\in[0, 1]$ and all $t>0$, $u\in H_{ d}^{ 1}$,
\begin{equation}
\label{eq:reg_semigroup_adjoint_m1}
\left\Vert  e^{t L^*_{ \delta}} (1-P_{ \delta}^{ 0, \ast})u\right\Vert_{ 1+2 \beta, d}\, \leq\,  C  \frac{ e^{ - \gamma t}}{ t^{ \beta}}\Vert
(1-P_{ \delta}^{ 0, \ast})u\Vert_{ 1, d}\, ,
\end{equation}
and
\begin{equation}
\label{eq:reg_semigroup_adjoint}
\left\Vert  e^{t L^*_{ \delta}} u\right\Vert_{ 1+2 \beta, d}\, \leq\,  C \left(1 + \frac{ e^{ - \gamma t}}{ t^{ \beta}}\right)\Vert u\Vert_{ 1, d}\, ,
\end{equation}
and for all $\beta\geq 0$, $\beta'\geq 0$ such that $\beta+\beta'\leq 1$ and all $h\in H^{1+2\beta+ 2 \beta^{ \prime}}_d$,
\begin{equation}
\label{eq:reg_semigroup_adjoint_2}
\left\Vert \left( e^{t L^*_{ \delta}}-1\right) h\right\Vert_{ 1+2\beta', d} \, \leq\, t^\beta\Vert (1-P_{ \delta}^{ 0, \ast}) h\Vert_{ 1+2\beta'+2\beta, d}\, .
\end{equation}
\end{proposition}

\begin{proof}[Proof of Proposition~\ref{prop:reg_semigroup_adjoint}]
Following Proposition~\ref{prop:L_ast_sectorial}, $L^*_\gd P_{ \delta}^{ 0, \ast}=0$ and $L^*_\gd (1-P_{ \delta}^{ 0, \ast})$ is sectorial in $H^1_d$, with spectrum lying in $\{ \gl\in\bbC:\, |\arg(\gl)|>\pi/2+\gep'\}- \gamma_{L_{ \delta}^{ \ast}}$ for some $\gep'>0$. 

Let us first prove \eqref{eq:reg_semigroup_adjoint_m1} and \eqref{eq:reg_semigroup_adjoint}.
Using \cite{Henry1981}, Theorem 1.4.3, (recall that $ \gamma< \gamma_{L_{ \delta}^{ \ast}}$), we obtain that for all $t>0$:
\begin{equation}
 \Vert (-L^*_\delta)^\beta e^{t L^*_\gd}(1-P_{ \delta}^{ 0, \ast})u\Vert_{1, d}\, \leq\, C_\beta t^{-\beta}e^{- \gamma t}\Vert u\Vert_{1,d}\, .
\end{equation}
Now as in the proof of Lemma~\ref{lem:equiv_norm_beta}, we can apply \cite{Henry1981}, Theorem 1.4.8 to show that the norms induced by $(-L^*_\gd)^\beta$ and $(1-L^*_\gd)^\beta$ are equivalent on the range of $(1-P_{ \delta}^{ 0, \ast})$ and we obtain for all $u\in H^1_d$
\begin{equation}
 \Vert e^{tL^*_\gd}u \Vert_{1+ 2\beta, d}\, \leq\, C\Vert (1-L^*_\gd)^{\beta}e^{tL^*_\gd}(P_{ \delta}^{ 0, \ast}u +(1-P_{ \delta}^{ 0, \ast})u)\Vert_{1, d} \, \leq\, C'_\beta (1+e^{- \gamma t}t^{-\beta})\Vert u\Vert_{1, d}\, .
\end{equation}
We have used here in particular the fact that for all $u\in H_{ d}^{ 1}$,
\[\left\Vert e^{ tL_{ \delta}^{ \ast}} P_{ \delta}^{ 0, \ast}u\right\Vert_{ 1+2 \beta, d} \leq  \left\vert \tilde p(u) \right\vert \left\Vert e^{ t L_{ \delta}^{ \ast}} v_{ 0}\right\Vert_{ 1+2 \beta, d} = \left\vert \tilde p(u) \right\vert \left\Vert v_{ 0}\right\Vert_{ 1+2 \beta, d}\leq C \left\Vert u \right\Vert_{ 1, d},\]
since $ \left\Vert v_{ 0} \right\Vert_{ 1+2 \beta}< +\infty$. Concerning \eqref{eq:reg_semigroup_adjoint_2}, remark that
\begin{equation}
  e^{t L^*_{ \delta}}-1 \, =\,   \left(e^{t L^*_{ \delta}}(1-P_{ \delta}^{ 0, \ast})-1\right) (1-P_{ \delta}^{ 0, \ast})\, ,
\end{equation}
so applying Theorem 1.8.4 of \cite{Henry1981}
we have
\begin{multline}
\left\Vert \left( e^{t L^*_{ \delta}}-1\right) h\right\Vert_{ 1+2\beta', d} \, \leq\,  C \left\Vert (1-L^*_\gd)^{\beta'}\left( e^{t L^*_{ \delta}}(1-P_{ \delta}^{ 0, \ast})-1\right) (1-P_{ \delta}^{ 0, \ast})h\right\Vert_{ 1, d} \\
\leq \, C''_\beta t^\beta \Vert (1-L^*_\gd)^{\beta'+\beta} (1-P_{ \delta}^{ 0, \ast})h\Vert_{1,d}
\, \leq \, C'''_\beta t^\beta \Vert (1-P_{ \delta}^{ 0, \ast})h\Vert_{1+2\beta'+2\beta,d}\, .
\end{multline}This concludes the proof of Proposition~\ref{prop:reg_semigroup_adjoint}.
\end{proof}
One can deduce from Proposition~\ref{prop:reg_semigroup_adjoint} a similar regularity result concerning the semigroup $t \mapsto e^{ tL_{ \delta}}$:
\begin{proposition}
\label{prop:reg_semigroup} For all $K\geq1$, all $ 0\leq \delta< \delta(K)$, the semigroup $t \mapsto e^{ t L_{ \delta}}$ is continuous from $H_{ d}^{ -2}$ to $H_{ d}^{ -1}$: for all $h\in H_{ d}^{ -2}$, $t>0$,
\begin{equation}
\label{eq:reg_semigroup_m1}
\left\Vert e^{ tL_{ \delta}} h\right\Vert_{ -1, d} \, \leq\,  C \left(1 + \frac{ 1}{ \sqrt{t}}\right) \left\Vert h \right\Vert_{ -2, d},
\end{equation}
and for all $ \varepsilon\in (0, 1/2)$, $t>0$, $ u\geq0$,
\begin{equation}
\label{eq:cont_semigroup_m1}
\left\Vert e^{ (t+u) L_{ \delta}}h - e^{ tL_{ \delta}}h\right\Vert_{ -1, d}\,  \leq\,  C u^{ \varepsilon} \left(1 + \frac{ 1}{ t^{ 1/2+ \varepsilon}}\right) \left\Vert h \right\Vert_{ -2, d}.
\end{equation}
\end{proposition}
\begin{proof}[Proof of Proposition~\ref{prop:reg_semigroup}]
 Let $ \beta\in[0, 1]$, $t>0$, $h\in H_{ d}^{ -1}$ and $v$ a regular test function. Consider $(h_{ l})_{ l\geq1}$ a sequence of elements of $\bL_{ 0, d}^{ 2}$ converging to $h$ in $H^{ -1}_{ d}$. For all $l\geq1$,
 \begin{align*}
 \left\vert \left\langle e^{ tL_{ \delta}}h_{ l}\, ,\, v\right\rangle_{ d} \right\vert &= \left\vert \left\langle e^{ tL_{ \delta}}h_{ l}\, ,\, v\right\rangle_{ 2, d} \right\vert= \left\vert \left\langle h_{ l}\, ,\, e^{ t L_{ \delta}^{ \ast}}v\right\rangle_{ 2, d} \right\vert,\\
 &\leq \left\Vert h_{ l} \right\Vert_{ -(1+2 \beta), d} \left\Vert e^{ tL_{ \delta}^{ \ast}}v \right\Vert_{ 1+2 \beta, d} \leq C\left\Vert h_{ l} \right\Vert_{ -(1+2 \beta), d} \left(1+ \frac{ 1}{ t^{ \beta}}\right) \left\Vert v \right\Vert_{ 1, d},
 \end{align*}
 where we used \eqref{eq:reg_semigroup_adjoint} in the last inequality. Since $h_{ l}$ converges to $h$ in $H^{ -1}$, one can make $l\to\infty$ in the previous inequality and obtain $\left\vert \left\langle e^{ tL_{ \delta}}h\, ,\, v\right\rangle_{ d} \right\vert \leq C\left\Vert h\right\Vert_{ -(1+2 \beta), d} \left(1+ \frac{ 1}{ t^{ \beta}}\right) \left\Vert v \right\Vert_{ 1, d}$ and since this is true for all regular $v$, one deduces that
 \begin{equation}
 \label{eq:reg_semigroup_mbeta}
 \left\Vert e^{ t L_{ \delta}}h \right\Vert_{ -1, d}\, \leq\,  C \left(1+ \frac{ 1}{ t^{ \beta}}\right) \left\Vert h \right\Vert_{ -(1+2 \beta), d},
 \end{equation}
 which gives \eqref{eq:reg_semigroup_adjoint_m1} when $ \beta= \frac{ 1}{ 2}$. In the same way, an immediate corollary of \eqref{eq:reg_semigroup_adjoint_2} is that for all $ \beta\geq0, \beta^{ \prime}\geq0$ such that $ \beta+ \beta^{ \prime}\leq 1$, for all $t>0$
 \begin{equation}
 \label{eq:reg_semigroup_2}
 \left\Vert \left(e^{ tL_{ \delta}} - 1\right)h \right\Vert_{ -(1 + 2 \beta + 2 \beta^{ \prime}), d}\, \leq \, t^{ \beta} \left\Vert h \right\Vert_{ -(1+ 2 \beta^{ \prime}), d}\, .
 \end{equation}
 We now turn to the proof of \eqref{eq:cont_semigroup_m1}. Fix $ \varepsilon\in(0, 1/2)$ and apply \eqref{eq:reg_semigroup_mbeta} for $ \beta= 1/2 + \varepsilon$ and \eqref{eq:reg_semigroup_2} for $ \beta= \varepsilon$ and $ \beta^{ \prime}= \frac{ 1}{ 2}$,
 \begin{align*}
 \left\Vert e^{ (t+u)L_{ \delta}}h - e^{ t L_{ \delta}}h \right\Vert_{ -1, d}&\leq \, C \left(1 + \frac{ 1}{ t^{ 1/2+ \varepsilon}}\right) \left\Vert \left(e^{ uL_{ \delta}}- 1\right)h \right\Vert_{ -(2+ 2 \varepsilon), d},\\
 &\leq \, C u^{ \varepsilon}\left(1 + \frac{ 1}{ t^{ 1/2+ \varepsilon}}\right) \left\Vert h \right\Vert_{ -2, d}.
 \end{align*}This concludes the proof of Proposition~\ref{prop:reg_semigroup}.
\end{proof}
\section{Projections}\label{sec:appendix projections}
The purpose of this section is to prove several regularity results concerning the projection $P^{ 0}_{ \psi, \delta}u= \mathtt{p}_{ \psi, \delta}(u) \partial_{ \theta} q_{ \psi, \delta}$ (recall Section~\ref{sec:lin stab} and \eqref{eq:def p psi}) and the projection on the manifold $M$ $\proj_{ M}(\cdot)$ defined in Lemma~\ref{lem:existence projM}.
\begin{proof}[Proof of Lemma \ref{lem:existence projM}]
We first prove that $\psi\mapsto \mathtt{p}_\psi$ is smooth. This follows from the fact that the whole operator $L_{ \psi}$ is regular in $ \psi\in \bbT$: we prove indeed that the mapping $ \psi \mapsto L_{ \psi}$ is in fact real holomorphic, in the sense of Kato \cite{Kato1995}, p.375. Since the problem is invariant by rotation, it suffices to study the regularity of $L_{ \psi}$ is a neighborhood of $ \psi=0$. From the definition of the stationary solution $q$ in \eqref{eq:def_q}, it is straightforward to see that one can expand $q_{ \psi}$ in series of $ \psi$ around $ \psi=0$:
\[q_{ \psi}(\theta) = q_{ 0}(\theta) + \sum_{ k\geq1} \frac{ \psi^{ k}}{ k!}\partial_{ \psi}^{ k} {q_{ \psi}}_{ \vert_{ \psi=0}}(\theta).\]From this expansion, one deduces a similar expansion for $L_{ \psi}$ around $ \psi=0$: for all $f$ regular
\[L_{ \psi}f= L_{ 0}f + \sum_{ k\geq 1} \psi^{ k} U_{ k} f,\] where each $U_{ k}$ is a differential operator of order $1$, so that each $U_{ k}$ is relatively-bounded w.r.t $L_{ 0}$. In particular the hypotheses of \cite{Kato1995}, Theorem~2.6, p. 377 are satisfied. In particular, $(L_{ \psi})_{ \psi}$ forms a real-holomorphic family.  In particular, the mapping $ \psi \mapsto P^{ 0}_{ \psi}$ is also regular (\cite{Kato1995}, Theorem~1.7, p. 368),
and so is the mapping $\psi \mapsto\mathtt{p}_\psi $. Then the mapping $f(\psi,h)=\mathtt{p}_\psi(h-q_\psi)$ satisfies for each fixed $\psi_0$, $f(\psi_0,q_{\psi_0})=0$ and $\partial_\psi f(\psi_0,q_{\psi_0})=-\mathtt{p}_{\psi_0}\partial_\psi q_{\psi_0}=-1$. So by the implicit function theorem, for all $h$ in a certain neighborhood of $q_{\psi_0}$, there exists a unique $\psi=:\proj_M(h)$ such that $f(\psi, h)=0$ and $h\mapsto \proj_M(h)$ is smooth.
\end{proof}
The next result states that the first order of the projection $\proj_{ M}$ around $q_{ \psi}$ is given by the linear form $\mathtt{p}_{ \psi}$ defined in \eqref{eq:def p psi}.
\begin{lemma}\label{lem:first order projM}
For $ \psi\in \bbT$, $h\in H^{-1}_d$ such that $\proj_M(q_\psi+h)$ is well-defined, we have
\begin{equation}
\proj_M(q_\psi+h)\, =\, \psi + \mathtt{p}_\psi(h) + O(\Vert h\Vert_{-1,d}^2)\, .
\end{equation}
\end{lemma}
\begin{proof}[Proof of Lemma \ref{lem:first order projM}]
Consider the real $u$ such that $\proj_M(q_\psi+h)=\psi +u$.
Due to the smoothness of $\proj_M$, we have $u=O(\Vert h\Vert_{-1,d})$. The real number $u$ satisfies
\begin{equation}
 \mathtt{p}_{\psi+u}(q_\psi+h-q_{\psi+u})\, =\, 0\, .
\end{equation}
A first order expansion leads to
\begin{equation}
 \mathtt{p}_\psi(h-u \partial_\psi q_\psi)\, =\, O(u^2)\, ,
\end{equation}
which gives the result, since $\mathtt{p}_\psi(\partial_\psi q_\psi)=1$.
\end{proof}
\section{Expansions in $\gd$}\label{sec:appendix expansion delta}
The aim of this section is to obtain first order asymptotic of the drift in Theorem~\ref{th:main} for small $\gd$. We use the notations $q_\gd$, $\mathtt{p}_\gd$ as in Section \ref{sec:asymptotic_drift}, putting the emphasis on the dependency of the different terms in $\gd$. We denote also as $r_{ \delta}>0$ the unique positive solution to the fixed point relation $r_{ \delta} = \Psi_{ \delta}(2Kr_{ \delta})$ (recall \eqref{eq:fixed_point}). We begin with a result concerning $r_{ \delta}$ as $ \delta\to 0$:
\begin{lemma}
\label{lem:derivative_r} The mapping $\gd \mapsto r_\gd$ is $C^\infty$ and its derivative $r^{ \prime}(0)$ at $ \delta=0$ is zero, so that as $ \delta\to 0$:
\begin{equation}
 r_\gd\, =\, r_0+O(\gd^2)\, ,
\end{equation}
where $r_{ 0}$ is the unique non-trivial solution of the fixed-point problem without disorder \eqref{eq:fixed point without disorder}.
\end{lemma}
\begin{proof}[Proof of Lemma~\ref{lem:derivative_r}]
Consider the $C^\infty$ mapping $g(r,\gd)=\Psi_\gd(2Kr)-r$. This mapping satisfies
$\partial_r g(r_0,0)=2K\partial_x \Psi_0(2Kr_0)-1$. The fixed-point function $r \mapsto \Psi_{ 0}(2Kr_{ 0})$ is strictly convex when $K>1$ (\cite{Pearce}, Lemma~4), with derivative at the origin strictly greater than $1$. One concludes that the derivative at the fixed point $r_{ 0}>0$ is strictly smaller than $1$. Since this derivative is precisely equal to $2K \partial_{ x}\Psi_{ 0}(2Kr_{ 0})$, this shows that
$\partial_r g(r_0,0)<0$. So the implicit function Theorem implies that $\gd \mapsto r_\gd$ is $C^\infty$. Using \eqref{eq:fixed_point}, one obtains that 
\begin{equation}
r^{ \prime}(0)\, =\,  \partial_{ \delta} \Psi_{ \delta}\vert_{ \delta=0}(2Kr_{ 0}) + r^{ \prime}(0) 2K \partial_{ x}\Psi_{ 0}(2Kr_{ 0}).
\end{equation}
Since $2K\partial_{ x}\Psi_{ 0}(2Kr_{ 0})<1$, the proof of Lemma~\ref{lem:derivative_r} will be finished once we have proved that $\partial_{ \delta} \Psi_{ \delta}\vert_{ \delta=0}(2Kr_{ 0}) = 0$. One has (recall the definition of $\cZ_{ 0}$ in Section~\ref{sec:case_delta_0})
\begin{multline}
\label{aux:derivative_Psi_delta}
\partial_{ \delta}\Psi_{ \delta}\vert_{ \delta=0}(2Kr_{ 0})\, =\,  \sum_{ k=-d}^{ d} \lambda^{ k} \Bigg( \frac{ \int_{0}^{2\pi} \cos(\theta) \partial_{ \delta} S_{ \delta}^{ k}\vert_{ \delta=0}(\theta, 2Kr_{ 0}) \dd \theta}{ \cZ_{ 0}(2Kr_{ 0})^{ 2}} \\ - \frac{ \int_{0}^{2\pi} \cos(\theta) S_{ 0}(\theta, 2Kr_{ 0})}{ \cZ_{ 0}(2Kr_{ 0})^{ 4}} \partial_{ \delta} Z_{ \delta}^{ k}\vert_{ \delta=0}(2Kr_{ 0})\Bigg)\, .
\end{multline}
Some straightforward calculations show that, for all $k=-d, \ldots, d$, $ \theta\in \bbT$
\begin{multline}
\label{aux:derivative_S_delta}
\partial_{ \delta}S^{ k}_{ \delta}\vert_{ \delta=0}(\theta, 2Kr_{ 0})\, =\,  2 \omega^{ k} e^{ 2Kr_{ 0}\cos(\theta)} \Bigg(\theta \int_{0}^{2\pi} e^{ 2Kr_{ 0}\cos(u)}\dd u + 2\pi \int_{ \theta}^{2\pi} e^{ -2Kr_{ 0} \cos(u)}\dd u\\ - \int_{0}^{2\pi} u e^{ -2Kr_{ 0}\cos(u)}\dd u\Bigg)
\end{multline}
and
\begin{align}
\label{aux:derivative_Z_delta}
\partial_{ \delta} Z_{ \delta}^{ k}\vert_{ \delta=0}(2Kr_{ 0})&=\,  2 \omega^{ k} \Bigg( 2\pi \int_{0}^{2\pi} e^{ 2Kr_{ 0}\cos(\theta)} \int_{ \theta}^{2\pi} e^{ -2Kr_{ 0}\cos(u)}\dd u \dd \theta \nonumber\\ &\qquad\qquad +\cZ_{ 0}(2Kr_{ 0}) \int_{0}^{2\pi} u \left( e^{ 2Kr_{ 0}\cos(u)} - e^{ -2Kr_{ 0}\cos(u)}\right)\dd u \Bigg), \nonumber\\
&= \, 4\pi \omega^{ k} \int_{0}^{2\pi} e^{ 2Kr_{ 0}\cos(\theta)} \int_{ \theta}^{2\pi} e^{ -2Kr_{ 0}\cos(u)}\dd u \dd \theta\, .
\end{align}
Since $ \sum_{ k=-d}^{ d} \lambda^{ k} \omega^{ k}=0$, one obtains from \eqref{aux:derivative_Psi_delta}, \eqref{aux:derivative_S_delta} and \eqref{aux:derivative_Z_delta} that $ \partial_{ \delta} \Psi_{ \delta}\vert_{ \delta=0}(2Kr_{ 0})=0$. This concludes the proof of Lemma~\ref{lem:derivative_r}.
\end{proof}
We now turn to the proof of Lemma \ref{lem:exp q}:
\begin{proof}[Proof of Lemma \ref{lem:exp q}]
Obviously, for $\theta\in \bbT$,
\begin{align*}
q^{ i}_{ \delta}(\theta) &\, =\,  q_{ 0}(\theta) + \delta \partial_{ \delta}q_{ \delta}\vert_{ \delta=0}(\theta) + O(\delta^{ 2}),
\end{align*}
where the error $O(\delta^{ 2})$ does not depend on $ \theta\in\bbT$. The fact that $r^{ \prime}(0)=0$ (Lemma~\ref{lem:derivative_r}) implies that $\partial_{ \delta}q_{ \delta}\vert_{ \delta=0}(\theta)$ only depends on the derivatives of $S_{ \delta}$ and $Z_{ \delta}$ w.r.t. $ \delta$, not w.r.t. $x$. Namely,
\begin{align*}
\partial_{ \delta}q^{ i}_{ \delta}\vert_{ \delta=0}(\theta)\, =\,  \frac{ \partial_{ \delta}S^{ i}_{ \delta}\vert_{ \delta=0}(\theta, 2Kr_{ 0})}{ \cZ_{ 0}(2Kr_{ 0})^{ 2}} - \frac{ \partial_{ \delta}Z^{ i}_{ \delta}\vert_{ \delta=0}(2Kr_{ 0})S^{ i}_{ 0}(\theta, 2Kr_{ 0})}{ \cZ_{ 0}(2Kr_{ 0})^{ 4}}\, .
\end{align*}
The expansion found in \eqref{eq:expansion_q_delta} is a simple consequence of \eqref{aux:derivative_S_delta}, \eqref{aux:derivative_Z_delta} and the expression of $\cZ_{ 0}$ in Section~\ref{sec:case_delta_0}.
\end{proof}

\begin{proof}[Proof of Lemma \ref{lem:exp proj p}]
In the case $ \delta=0$, the projection $\mathtt{p}_{ 0}$ defined in \eqref{eq:def p psi} is given by $P^{ 0}_{ 0}(u)= \mathtt{p}_{ 0}(u) (\partial_{ \theta}q_{ 0}, \ldots, \partial_{ \theta}q_{ 0})= \mathtt{p}_{ 0}(u) \partial_{ \theta}q_{ 0, nd}$. Since in this case, the operator $L_{ 0}= A$ defined in \eqref{eq:A} is essentially self-adjoint in $H_{ 1/q_{ 0}, d}^{ -1}$ (Proposition~\ref{prop:A_self_adjoint}), the projection $\mathtt{p}_{ 0}$ as a natural representation in terms of the scalar product $ \left\langle \cdot\, ,\, \cdot\right\rangle_{ -1, 1/q_{ 0}, d}$ associated to the norm defined in \eqref{eq:normHm1d}, namely
\begin{equation}
\mathtt{p}_{ 0}(u)\, =\,  \frac{ \left\langle \partial_{ \theta}q_{ 0, nd}\, ,\, u\right\rangle_{ -1, 1/q_{ 0}, d}}{ \left\Vert \partial_{ \theta} q_{ 0, nd} \right\Vert_{ -1, 1/q_{ 0}, d}^{ 2}}\, .
\end{equation}
Using the notations of Section~\ref{sec:appendix_rigged_spaces}, we deduce that 
\begin{align*}
\left\Vert \partial_{ \theta} q_{ 0, nd} \right\Vert_{ -1, 1/q_{ 0}, d}^{ 2}\, = \, \int_{0}^{2\pi} \frac{ \left(q_{ 0}(\theta) - \frac{ 2\pi}{ \cZ_{ 0}^{ 2}}\right)^{ 2}}{ q_{ 0}} \dd \theta \, =\,  1- \frac{ 4\pi^{ 2}}{ \cZ_{ 0}^{ 2}}\, ,
\end{align*}
and
\begin{align*}
\left\langle \partial_{ \theta}q_{ 0, nd}\, ,\, u\right\rangle_{ -1, 1/q_{ 0}, d}&=\,  \sum_{ k=-d}^{ d} \lambda^{ k} \int_{0}^{2\pi} \frac{ \cU^{ k}(\theta)\left(q_{ 0}(\theta) - \frac{ 2\pi}{ \cZ_{ 0}^{ 2}}\right)}{ q_{ 0}(\theta)} \dd \theta\\ &= \, \sum_{ k=-d}^{ d} \lambda^{ k} \int_{0}^{2\pi} \cU^{ k}(\theta)\left(1 - \frac{ 2\pi}{ \cZ_{ 0}^{ 2}q_{ 0}(\theta)}\right) \dd \theta\, ,
\end{align*}
which precisely gives the first order of \eqref{eq:expand_p_delta}. The validity of \eqref{eq:expand_p_delta} comes from the definition of the projection $P^{ 0}_{ \delta}$ in \eqref{eq:def p psi} and the fact that $L_{ \delta}$ is a relatively bounded perturbation of order $\gd$ of the operator $L_{ 0}= A$.
\end{proof}

\end{appendix}

\section*{Acknowledgments}

C.P. acknowledges the support of the ERC Advanced Grant ``Malady'' (246953).
We thank D. Bl\"omker and G. Giacomin for fruitful discussions and valuable advice.

\def\cprime{$'$}

\end{document}